\documentclass[a4paper,12pt,english]{smfart}  
\usepackage{amssymb,amsmath,amsthm,geometry,graphicx} 
\usepackage{amsfonts,amscd}
\usepackage{psfrag,xspace} 
\usepackage[latin1]{inputenc}
\usepackage[T1]{fontenc}
\usepackage[all,dvips,cmtip]{xy} 
\usepackage[active]{srcltx}
\usepackage{bbm}
\usepackage{tikz}
\tikzset{math3d/.style={x= {(1cm,0cm)}, z={(0cm,1cm)},y={(0.353cm,0.353cm)}}}

\usepackage{psfrag} %to use figure
\usepackage{fullpage,url}% Two options [cm] or [in] (default)
\usepackage{latexsym} %Para usar \Box
\usepackage{epsfig} %graphic package for encapsulated postscript
\usepackage[all]{xy}%diagrams
\usepackage{mathrsfs}%certain calligraphic fonts, like \mathscr
\usepackage{verbatim,epic,eepic}
\newcommand{\nn}{\mathbb{N}}

\newcommand{\zz}{\mathbb{Z}}

\newcommand{\cc}{\mathbb{C}} %Complex numbers
\newcommand{\pp}{\mathbb{P}} % Projective line

\newcommand{\gd}{\mathcal{D}} % ``
\newcommand{\gl}{\mathcal{L}} % ``
\newcommand{\gt}{\mathcal{T}} %
\newcommand{\lra}{\longrightarrow}
\newcommand{\ra}{\rightarrow}

\newcommand{\lcor}{\left\langle\! \left\langle}
\newcommand{\rcor}{\right\rangle\! \right\rangle}
\newcommand{\twprod}[1]{\bullet_{#1}^{\tw}}

\newcommand{\redprod}[1]{\bullet_{#1}^{\red}}
\newcommand{\Zprod}[1]{\bullet_{#1}^{Z}}

\def\cf{\textit{cf.}\kern.3em}

\def\ie{\textit{i.e.},\ }
\def\resp{\textit{resp.}\kern.3em}
\renewcommand{\k}{\kern2pt}
\numberwithin{equation}{section} \makeatletter
\renewcommand\thesection{\arabic{section}} \makeatother
\renewcommand\thesubsection{\thesection.\arabic{subsection}}
\renewcommand\thesubsubsection{\thesubsection.\alph{subsubsection}}
\newcommand*{\longhookrightarrow}{\ensuremath{\lhook\joinrel\relbar\joinrel\rightarrow}}

\newtheorem{thm}[equation]{Theorem}
\newtheorem{cor}[equation]{Corollary}
\newtheorem{lem}[equation]{Lemma}
\newtheorem{prop}[equation]{Proposition}
\theoremstyle{definition} %This declaration affects to everything below
\newtheorem{defn}[equation]{Definition}
\newtheorem{rem}[equation]{Remark}

\newtheorem{exmp}[equation]{Example}%
\newtheorem{notn}[equation]{Notation}%
\newtheorem{assumption}[equation]{Assumption}%

\DeclareMathOperator{\Bl}{Bl}%
\DeclareMathOperator{\pt}{pt}%
\DeclareMathOperator{\Mir}{Mir}%
\DeclareMathOperator{\Gr}{gr}%
\DeclareMathOperator{\amb}{amb}%
\DeclareMathOperator{\QDM}{QDM}%
\DeclareMathOperator{\red}{red}%
\DeclareMathOperator{\rank}{rk}%
\DeclareMathOperator{\Hom}{Hom}%
\DeclareMathOperator{\im}{Im}%\r
\DeclareMathOperator{\Id}{id}%
\DeclareMathOperator{\Pic}{Pic}%
\DeclareMathOperator{\Spec}{Spec}

\DeclareMathOperator{\vir}{vir}%
\DeclareMathOperator{\topp}{top}%
\DeclareMathOperator{\tw}{{tw}}%
\renewcommand{\k}{\kern2pt}
\newcommand{\la}{\lambda}
\newcommand{\La}{\Lambda}
\newcommand{\al}{\alpha}
\newcommand{\de}{\delta}
\newcommand{\De}{\Delta}
\newcommand{\eps}{\epsilon}

\newcommand{\A}{\mathbb{A}}

\newcommand{\C}{\mathbb{C}}
\newcommand{\D}{\mathbb{D}}
\newcommand{\G}{\mathbb{G}}
\newcommand{\I}{\mathbb{I}}

\newcommand{\N}{\mathbb{N}}
\newcommand{\Q}{\mathbb{Q}}
\newcommand{\R}{\mathbb{R}}
\newcommand{\Z}{\mathbb{Z}}
\renewcommand{\gg}{\mathcal{G}}

\newcommand{\go}{\mathcal{O}}
\newcommand{\gp}{\mathcal{P}}

\newcommand{\gz}{\mathcal{Z}}

\newcommand{\isom}{\stackrel{\sim}{\longrightarrow}}

\newcommand{\ov}[1]{\overline{#1}}

\DeclareMathOperator{\Quot}{Quot}%

\newcommand{\<} {\left\langle}
\renewcommand{\>} {\right\rangle}

\newcommand{\specmori}{\mathbf S}  
\newcommand{\spectoric}{\mathbf T}
\newcommand{\convmori}{\ov{\mathbf D}}
\newcommand{\convtoric}{\mathbf D}
\newcommand{\freemori}{\mathbf V}
\newcommand{\freetoric}{\mathbf U}
\newcommand{\mirrortoric}{\mathbf W}

\newcommand{\ctopvb}{c_\top(\vb)} % Top Chern class of the fiber bundle.
\newcommand{\ctop}{c_\top} % Top Chern class.
\newcommand{\hatctop}{\widehat{c}_{\mathrm top}} % Differential operator associated to c_top with d.

\newcommand{\bat}{B} % Notation for the Batyrev ring.
\newcommand {\sq}{\Box} % Symbole des opÃ©rateurs carrÃ©s diffÃ©rentiels.
\newcommand{\raysbase} {\De\!(1)^{\!\scriptscriptstyle\mathit{base}}}

\newcommand{\raysfiber}  {\De\!(1)^{\!\scriptscriptstyle\mathit{bund}}}
\newcommand{\eulerclass}{\mathfrak{E}} % L'operateur differentiel d'Euler.
\newcommand{\eulerfield}{\widehat{\mathfrak{E}}} % L'operateur differentiel d'Euler.
\newcommand{\lb}{\mathcal L} % either the line bundle or invertible sheaf
\newcommand{\vb}{\mathcal E} % either the vector bundle or locally free sheaf
\newcommand{\tordiv}{L}      % the chosen toric divisor of a line bundle L.

\newcommand{\mori}[1]{\text{NE\hspace{-0.1mm}}(#1)} %Le cone de Mori entier
\DeclareMathOperator{\Lm}{Lm} % leading monomials for orders in polynomial rings.
\newcommand{\mc}{m_{c_{{top}}}}
\renewcommand{\top}{\mathrm{top}} 

\DeclareMathOperator{\ev}{e}% the evaluation from moduli space to variety.

\newcommand{\GKZid}{\mathbb{G}}   % idÃ©al GKZ dans \D (avec prime, dans \D').
\newcommand{\GKZidsheaf}{\mathcal{G}}   % idÃ©al GKZ
\newcommand{\GKZmod}{\mathbb{M}} % module GKZ, quotient de \D par \GKZid ou \D' par \GKZid'.
\newcommand{\GKZmodsheaf}{\mathcal{M}} % sheaf of ring given by the GKZ module
\DeclareMathOperator{\QSR}{QSR} % The quantum SR ideal
\DeclareMathOperator{\Lin}{Lin} % The linear ideal
\DeclareMathOperator{\SR}{SR} % The stanley reisner ideal
\newcommand{\GKZcom}{G} % the commutative GKZ ideal

\newcommand{\mapyx}{\phi} % either map from N' to N (projection) or from Y to X.

\DeclareMathOperator{\proj}{Proj} %

\newcommand{\qsrpol}{R} % the commutative gkz operator.
\newcommand{\0}{\mathbf{0}} % the large radius limit.
\newcommand{\res}{\mathrm{res}}
\newcommand{\dimevenH}{s} % The dimension of H^{2*}(X)

\newcommand{\dual}{\scriptscriptstyle\vee}
\newcommand{\tang}{\gt}
\begin{document}

\title{Quantum $\gd$-modules for toric nef complete intersections}

\author{Etienne Mann}

\address{Université d'Angers, Laboratoire Angevin de Recherche en Mathématiques LAREMA
  , UMR 6093, Département de mathématiques, Bâtiment I, Faculté des Sciences, 2 Boulevard Lavoisier, F-49045 Angers cedex 01
France}
\email{etienne.mann@univ-angers.fr}
\urladdr{http://www.math.univ-angers.fr/~mann/}
%\urladdr{http://www.math.univ-montp2.fr/~mann/index.html}

\author{Thierry Mignon}
\address{ Université de Montpellier, Institut Montpelliérain Alexander Grothendieck, UMR 5149, Case courier 051
Place Eugène Bataillon
F-34 095 Montpellier CEDEX 5 }
\email{thierry.mignon@math.univ-montp2.fr}
\urladdr{http://www.math.univ-montp2.fr/~mignon/}

  \begin{abstract} 
    Let $X$ be a smooth projective toric variety with $k$ ample line bundles.  Let $Z$ be the zero
    locus of $k$ generic sections. It is well-known that the ambient quantum $\mathcal{D}$-module of
    $Z$ is cyclic \ie is defined by an ideal of differential operators.  In this paper, we give an
    explicit construction of this ideal as a quotient ideal of a GKZ system associated to the toric
    data of $X$ and the line bundles. This description can be seen as a ``left cancellation
    procedure''.  We consider some examples where this description enables us to compute generators of
    this ideal, and thus to give a presentation of the ambient quantum $\gd$-module.

  \end{abstract}

\subjclass{14N35, 53D45, 14F10} 
\keywords{Quantum differential
    modules, Gromov-Witten invariants, Batyrev rings, GKZ systems, mirror symmetry, $\gd$-modules, toric geometry}

  % \thanks{E.M is supported by the grant of the Agence Nationale de la
  %   Recherche ``New symmetries on Gromov-Witten theories'' ANR- 09-JCJC-0104-01.}

\maketitle

\setcounter{tocdepth}{1}
\tableofcontents

%%%%%%%%%%%%%%%%%%%%%%%%%%%%%%%%%%%%%%%%%%%%%%%%%%%%%%%%%
\section{Introduction}
\label{sec:introduction}

Mirror symmetry has many different formulations in mathematics:
equivalence of derived categories (known as Homological Mirror
Symmetry by Kontsevich \cite{Kontseich-1995-HMS}), isomorphism of
Frobenius manifolds (see \cite{Bms}), comparison of Hodge numbers for
Calabi-Yau varieties (see for example
\cite{Batyrev-Dual-polyhedra-1994}), isomorphism of Givental's conesb
(see \cite{Givental-1998-Mirror-complete-intersection}), isomorphism
of pure polarized TERP structures (see \cite{2006-Hertling-tt*}) or
variation of non-commutative Hodge structures (see
\cite{Katzarkov-Pantev-Kontsevich-ncVHS}).

Inspired by the works of Givental (see for examples \cite{Givental-Equivariant-GW} and
\cite{Givental-1998-Mirror-complete-intersection}), many authors have considered quantum cohomology with a
differential module approach : see Kim \cite{Kim-1999-QH-for-G/P} and Rietsch (with Marsh and
Pech-Williams) \cite{Rietsch-Mirror-Toda-2012}
\cite{rietsch-Marsh-Grassmiannians-mirro-symmetry-2013arXiv1307.1085M}
\cite{Rietsch-Pech-Williams-LG-quadrics-2014} for homogeneous spaces, see Coates-Corti-Lee-Tseng
\cite{CCLTsqcwps} and Guest-Sakai \cite{2008-Guest-Sakai-orbifold-QDM} for weighted projective
spaces, see also the works of Iritani \cite{2006-Iritani-QDM-Floer}, \cite{iritani_convergence},
\cite{2008-Iritani-QDM-general-mirror-transform} and \cite{Iritani-2009-Integral-structure-QH}, the
book of Cox-Katz \cite{Cox-Katz-Mirror-Symmetry} and the one of Guest \cite{Guest-book-QDM-2010}.

From the small quantum product on a smooth projective variety $Z$, one can define a trivial vector
bundle over $\convtoric\times \cc$ where $\convtoric$ is an open subset of $H^{2}(Z,\cc)$ whose
fibers are $H^{2*}(Z,\cc)$.  This holomorphic bundle is endowed with a flat meromorphic connection
and a non-degenerate pairing. These data collectively define the quantum $\gd$-module of $Z$, which
is denoted by $\QDM(Z)$.  When $Z$ is a smooth \textbf{toric} Fano variety, Givental (see also
Iritani \cite{Iritani-2009-Integral-structure-QH} for toric weak Fano orbifolds) gives an explicit
presentation of this $\gd$-module using GKZ systems (Gelfand-Kapranov-Zelevinsky) in other words
$\QDM(Z)$ is isomorphic to $\mathcal{D}/\mathcal{G}_{Z}$ where $\mathcal{G}_{Z}$ is the GKZ ideal
associated to the toric data of $Z$.  When $Z$ is Fano, restricting this isomorphism to $\convtoric
\times \{0\}$ gives an isomorphism between the quantum cohomology ring of $Z$ and a commutative algebra
constructed by Batyrev in \cite{batyrev-quantum-1993}.

 % To prove this isomorphism, he uses the equality, up to a mirror map, between the so called
% $I$ and $J$ functions.

% In the very nice article \cite{Iritani-2009-Integral-structure-QH},
% Iritani enriches this quantum $\gd$-module by adding a natural
% integral structure \ie he defines a $\zz$-local system which is
% compatible with the connection.  We call \textit{quantum
%   $\gd$-module}, denoted by $\QDM(X)$, the trivial bundle endowed with
% a flat connection, a flat non-degenerated pairing and a natural
% integral structure.  This $\zz$-local system is natural in the
% following sense. Assume that $X$ has a mirror (for instance $X$ is a
% weak Fano toric orbifolds) that is a Laurent polynomial such that its
% Brieskorn lattice (which is a vector bundle with a flat connection) is
% isomorphic to the quantum $\gd$-module of $X$.  On this Brieskorn
% lattice, we have a natural integral structure that comes from the
% Lefschetz's thimbles.  The integral structure defined by Iritani is
% natural because it corresponds to the natural one on the mirror.
% Notice that the bundle, the connection, the pairing and the integral
% structure is part of the definition of a TERP structure defined by
% Hertling in \cite{2006-Hertling-tt*} or a variation of non-commutative
% Hodge structure defined by Kontsevich, Katzarkov and Pantev in
% \cite{Katzarkov-Pantev-Kontsevich-ncVHS}.

In this paper, we investigate the non toric case where $Z$ is a nef complete intersection subvariety
in a smooth toric variety $X$. To be more precise, let $\lb_{1}, \ldots ,\lb_{k}$ be ample line
bundles on $X$. Let $Z$ be the zero locus of a generic section of $\vb:=\oplus_{i=1}^{k}\gl_{i}$.
Denote by $\iota:Z\hookrightarrow X$ the closed embedding. By Lefschetz theorem, we have
$H^{*}(Z,\mathbb{C})= \im \iota^{*} \oplus \ker \iota_{*}$. The sub-vector space $\im \iota^{*}$ is
called the ambient part of the cohomology of $Z$, denoted by $H^{*}_{\amb}(Z)$. As $H^{*}_{\amb}(Z)$
is stable by the small quantum product of $Z$, we can define a sub $\mathcal{D}$-module, denoted by
$\QDM_{\amb}(Z)$, whose fibers are $H^{*}_{\amb}(Z)$.  A natural question is to find an explicit
presentation of $\QDM_{\amb}(Z)$.  It is well known that the GKZ ideal associated to $\mathcal{E}$,
denoted by $\mathcal{G}_{(X,\mathcal{E})}$, is part of the equations. Cox and Katz addressed in the
book \cite[p.94-95 and p.101]{Cox-Katz-Mirror-Symmetry} the following question: what differential
equations shall we add to $\mathcal{G}_{(X,\mathcal{E})}$ to get an isomorphism with $
\QDM_{\amb}(Z)$?

% As we
% have that 
% $\mathcal{G}(\mathcal{E^{\vee}})\subset
% (\mathcal{G}(\mathcal{E}^{\vee}):\widehat{c_{\top}(\mathcal{E})})$ Cox-Katz asked : 

Before giving an answer to this question in Theorem \ref{thm:1,intro}, we need to introduce some
notations.  Denote by $c_{\top}(\mathcal{E})$ the top Chern class of $\mathcal{E}$ and by
$\widehat{c}_{\top}\in \mathcal{D}$ its associated operator (see Notation \ref{notn:quantization}).
Denote by $(\mathcal{G}_{(X,\mathcal{E})}:\widehat{c}_{\top})$ the left quotient ideal that is the
left ideal of $\mathcal{D}$ defined by
\begin{displaymath}
  (\mathcal{G}_{(X,\mathcal{E})}:\widehat{c}_{\top}):=\langle P\in \mathcal{D} \mid
  \widehat{c}_{\top}P \in \mathcal{G}_{(X,\mathcal{E})} \rangle.
\end{displaymath}

\begin{thm}[See Theorem \ref{thm:main_theorem}]\label{thm:1,intro}
  Let $\lb_{1}, \ldots ,\lb_{k}$ be ample line bundles on $X$, and
  assume that $\dim_\C X\geq k+3$. Let $Z$ be the zero of a generic
  section of $\vb:=\oplus_{i=1}^{k}\gl_{i}$. Denote by 
  $\iota:Z\hookrightarrow X$ the closed embedding. The ambient $\mathcal{D}$-module
  $\QDM_{\amb}(Z)$  is isomorphic to $\mathcal{D}/(\mathcal{G}_{(X,\mathcal{E})}:\widehat{c}_{\top})$. 
\end{thm}
The quotient ideal $(\mathcal{G}_{(X,\mathcal{E})}:\widehat{c}_{\top})$ can be seen as a precise
statement for the ``left
cancellation procedure'' that appears in the works of Golyshev \cite[\S 2.9 and 2.10]{Golyshev-classification-pb-2007} and Guest-Sakai \cite[p.287]{2008-Guest-Sakai-orbifold-QDM}.

Reichelt-Sevenheck used this presentation of $\QDM_{\amb}(Z)$ to prove a mirror theorem for non
affine Landau-Ginzburg model (see \cite[Theorem 6.11]{Reichelt-SevenheckNonaffinLG-2012}).

To prove our main theorem, we proceed in several steps.
 \begin{enumerate}
 \item In the first section, we review some standard facts on twisted quantum $\mathcal{D}$-module
   $\QDM(X,\mathcal{E})$ which is of rank $\dim_{\mathbb{C}}H^{*}(X)$ and is defined via the
   Gromov-Witten invariants twisted by $\mathcal{E}$. We have a surjective morphism $\varphi:
   \QDM(X,\mathcal{E}) \to \QDM_{\amb}(Z)$ and we construct an explicit quotient of
   $\QDM(X,\mathcal{E})$ which gives an isomorphism with $\QDM_{\amb}(Z)$ (see Proposition
   \ref{prop:isomorphisme_QDM_reduit_et_ambiant}).
 \item Then we prove that we have an isomorphism of $\mathcal{D}$-modules
   $\varphi:\mathcal{D}/\mathcal{G}_{(X,\mathcal{E})} \to \QDM(X,\mathcal{E})$. To show this
   statement, we first define a surjective morphism.  Then we prove that
   $\mathcal{D}/\mathcal{G}_{(X,\mathcal{E})}$ is locally free of rank $\dim_{\cc}H^{*}(X)$. The
   freeness is proved in Section \ref{sec:GKZ_modules}. To compute the rank, we restrict
   $\mathcal{D}/\mathcal{G}_{(X,\mathcal{E})}$ to $\convtoric\times \{0\}$ and we get a commutative algebra.  This
   algebra is a twisted version of the standard Batyrev algebras in \cite{batyrev-quantum-1993}.  In
   Section \ref{sec:Batyrev_algebras}, we prove that the spectrum of this algebra is locally free of
   rank $H^{*}(X)$ over some explicit open subset of $\convtoric$ (see Theorem \ref{thm:Batyrev_locally_free}).
 \item Using the isomorphism $\varphi$ constructed above, we define a morphism $\overline{\varphi} :
   \mathcal{D}/(\mathcal{G}_{(X,\mathcal{E})}:\widehat{c}_{\top}) \to \QDM_{\amb}(Z)$
   which is surjective. To prove that $\overline{\varphi}$ is an isomorphism, we prove that
   $\mathcal{D}/(\mathcal{G}_{(X,\mathcal{E})}:\widehat{c}_{\top})$ is locally free
   of rank $\dim_{\cc}H^{*}_{\amb}(Z)$. The freeness is proved in Section \ref{sec:GKZ_modules}. To
   compute the rank, we restrict
   $\mathcal{D}/(\mathcal{G}_{(X,\mathcal{E})}:\widehat{c}_{\top})$ to $\convtoric\times
   \{0\}$ and we get a commutative algebra. In Section
   \ref{sec:Batyrev_algebras}, we prove that the spectrum of this algebra is locally free of rank
   $H^{*}_{ \amb}(Z)$ over some explicit open subset of $\convtoric$ (see Theorem \ref{thm:Batyrev_locally_free}).
\end{enumerate}

The plan of this article is the following.

Section \ref{sec:quantum-D-modules} contains a brief discussion of the twisted quantum $\gd$-module
$\QDM(X,\vb)$.

In Section \ref{sec:Batyrev_algebras}, we define and study  twisted Batyrev algebras  for a
quasi-projective toric variety. The main result of this section is Theorem \ref{thm:Batyrev_locally_free}. Notice that this section can be read independently of
the rest of the paper.%  The first Subsection
% \ref{subsec:notation_toric_varieties} is devoted to some recalls on
% toric geometry. In Subsection \ref{subsec:total_space} we construct
% the fan of the total space of the vector bundle $\vb$.  In Subsection
% \ref{subsec:Batyrev rings}, we define the Batyrev rings.  Subsection
% \ref{subsection:primitive,collections} is devoted to some recalls on
% the primitive collections.  In Subsection
% \ref{subsection:monomial,Groebner}, we prove that the quantum
% Stanley-Reisner ideal has a Groebner basis indexed by primitive
% collections (See Theorem \ref{theorem:grobner}).  In Subsections
% \ref{subsection:flat,Bat} and \ref{subsec:residual_Batyrev_ring}, we
% prove the Theorem  \ref{thm:flat_batyrev_version1} and Proposition
% \ref{prop:flat_residual_batyrev} quoted above in Theorem \ref{thm:3,intro}.

In Section \ref{sec:GKZ_modules}, we prove that the GKZ modules of $\mathcal{E}^{\vee}$ and its
residual are locally free sheaves. Using Section \ref{sec:Batyrev_algebras}, we compute their
ranks. The main result of this section is Theorem \ref{thm:GKZ_locally,free}.
 % We prove first
% % that the GKZ module $\mathcal{D}/\GKZidsheaf$ is coherent in Theorem
% % \ref{thm:N,coherent}. Then we prove that it is locally free of rank
% % $\rank{F}$ in Theorem \ref{thm:freeness_of_M}. We finish by a result
% % on the residual GKZ module $\mathcal{D}/ \Quot(\hatctop,\GKZidsheaf)$
% % (see Theorem \ref{thm:Res,locally,free}).  These results use Theorem
% % \ref{thm:flat_batyrev_version1} and Proposition
% % \ref{prop:flat_residual_batyrev} of the previous section.
 
In Section \ref{sec:isomorphism-theorems}, we
state and prove Theorem \ref{thm:1,intro} in Subsection
\ref{subsec:main,thm}.

In Section \ref{sec:exampl-:-hypers}, we give two explicit
computations of the generators of the quotient ideal
for $\pp^{n}$ with the line bundle
$\mathcal{O}(a)$ with $a\in\{1, \ldots ,n+1\}$ and the blow-up of
$\pp^{n}$ at a point with the line bundle $\mathcal{O}(aH+bE)$ with
$b\in \{-1, \ldots ,1-n\}$ and $a+b\in \{1,2\}$.

For the sake of completeness, lacking references in the literature, we review the axioms for twisted
Gromov--Witten theory in Appendix \ref{sec:twisted-axiom-for-GW}.

%  Appendix \ref{appendix:flat} is a complete proof of the flatness of the connection $\nabla$ using the twisted axioms. 

\textbf{Acknowledgements:} We thank Thomas Reichelt, Claude Sabbah and
Christian Sevenheck for useful discussions. The seminar in Paris
organised by Serguei Barannikov and Claude Sabbah on the
non-commutative Hodge structures was the starting point of this paper.
We also thank Antoine Douai for helping in the organisation of the
workshop in Luminy on the work of Iritani.  We are also grateful to
Hiroshi Iritani that pointed out the reference
\cite{mavlyutov_chiral_2000} (see Remark \ref{rem:mavlyutov}) and to
Claire Voisin of the reference \cite{Cataldo-Migliorini-2002-LEF}.  The first author is
 supported by the ANR \textit{New symmetries in Gromov-Witten theories}
 number ANR- 09-JCJC-0104-01 and both authors are member of the ANR \textit{Mirror symmetry and irregular
   singularities coming from physics}  ANR-13-IS01-0001-01.

\begin{notn}
 We use calligraphic
  letters for the sheaves such as
  $\gd,\gg,\GKZmodsheaf, \GKZmodsheaf^{\res}$.
  We use bold letters for modules or ideals on non commutative rings such as
  $\mathbb{D},\GKZid, \GKZmod, \GKZmod^{\res}$.
\end{notn}

%%%%%%%%%%%%%%%%%%%%%%%%%%%%%%%%%%%%%%%%%%%%%%%%%%%%%%%%%%%%%%%%
\section{Twisted and reduced quantum $\gd$-modules with geometric interpretation}
\label{sec:quantum-D-modules}
%%%%%%%%%%%%%%%%%%%%%%%%%%%%%%%%%%%%%%%%%%%%%%%%%%%%%%%%%%%%%%%%

Let $X$ be a smooth projective complex variety of dimension $n$ and  
 $\lb_{1}, \ldots,\lb_{k}$ be globally generated line bundles. 
Denote by $\vb$ the sum $\vb:=\lb_{1}\oplus \cdots \oplus \lb_{k}$.

\begin{notn}\label{notn:base_of_cohomology} 
For $0\leq i\leq 2n$,
denote by $H^i(X):=H^i(X,\C)$ the complex cohomology group of classes
of degree $i$. Also denote by $H^{*}(X)$ the complex cohomology
ring $\oplus_{i=0}^{2n} H^i(X)$~; the even part of this ring will be
written $H^{2*}(X)$. Put $\dimevenH=\dim_\C H^{2*}(X)$ and $r=\dim_\C
H^2(X)$.

% As a convention, we will write $H^2(X,\Z)$ (resp. $H_2(X,\Z)$) for the
% group of singular cohomology (resp. homology) \emph{ modulo torsion}. Since elements of $H_2(X,\Z)$ will only be used for integration, this convention is harmless.

%   bien $X_{0,l,d_1}\isom X_{0,l,d_2}$ par exemple. MÃªme si c'est
%   vrai, lorsqu'on fait des sommes du type : $\sum_{d\in
%     H_2(X,\Z)}...$, est ce que, si on veut la vraie somme (pas
%   modulo torsion), il faut multiplier la somme $\sum_{d\in
%     H_2(X,\Z)/Tor}..$ par le cardinal du groupe de torsion dans
%   $H_2$ ? Ceci n'est pas clair.}

We fix, once and for all, a homogeneous basis $(T_0, \ldots ,T_{\dimevenH-1})$
of $H^{2*}(X)$ such that $ T_{0}=\mathbf{1}$ is the unit for the cup
product and that the classes $T_1, \ldots ,T_r$ form a basis of $
H^2(X,\Z)$ modulo torsion. %  Denote by $(t_0, \ldots ,t_{\dimevenH-1})$ the associated
% coordinates and put $\tau:=\sum_{a=0}^{\dimevenH-1}t_{a}T_{a}$ and $
% \tau_{2}:=\sum_{a=1}^{r}t_{a}T_{a}$.
 Also denote by
$(T^0,\ldots,T^{\dimevenH-1})$ the Poincar\'e dual in $H^{2*}(X)$ of
$(T_0,\ldots,T_{\dimevenH-1})$.

As a convention, we will write $H_2(X,\Z)$ for the degree $2$ integer homology 
modulo torsion.
Denote by $(B_1, \ldots ,B_r)$ the dual basis of $(T_1,\ldots,T_r)$ in
$H_2(X,\Z)$. The associated coordinates will be denoted by $(d_1,
\ldots ,d_r)$.

We denote by $\tang_X$ the tangent bundle of $X$ and by $\omega_X$ its canonical
sheaf.
\end{notn}
As a convention, we will make no notational distinction between vector
bundles and locally free sheaves, writing --for example-- $\vb$ 
for both.

%%%%%%%%%%%%%%%%%%%%%%%%%%%%%%%%%%%%%%%%%%%%%%%%%%%%%%%%%%%%%%%%
%%%%%%%%%%%%%%%%%%%%%%%%%%%%%%%%%%%%%%%%%%%%%%%%%%%%%%%%%%%%%%%%
\subsection{Twisted quantum $\gd$-module}
\label{subsec:twisted-quantum-D-module}
%%%%%%%%%%%%%%%%%%%%%%%%%%%%%%%%%%%%%%%%%%%%%%%%%%%%%%%%%%%%%%%%
%%%%%%%%%%%%%%%%%%%%%%%%%%%%%%%%%%%%%%%%%%%%%%%%%%%%%%%%%%%%%%%%
%In this subsection, we define the twisted quantum $\gd$-module $\QDM(X,\vb)=(F,\nabla,S,F_{\zz})$.

%%%%%%%%%%%%%%%%%%%%%%%%%%%%%%%%%%%%%%%%%%%%%%%%%%%%%%%%%%%%%%%%%%%%%%
%%%%%%%%%%%%%%%%%%%%%%%%%%%%%%%%%%%%%%%%%%%%%%%%%%%%%%%%%%%%%%%%%%%%%%
\subsubsection{Twisted Gromov-Witten invariants}
\label{subsubsec:twisted_product}
%%%%%%%%%%%%%%%%%%%%%%%%%%%%%%%%%%%%%%%%%%%%%%%%%%%%%%%%%%%%%%%%%%%%%%
%%%%%%%%%%%%%%%%%%%%%%%%%%%%%%%%%%%%%%%%%%%%%%%%%%%%%%%%%%%%%%%%%%%%%
% First recall the definition of the twisted Gromov-Witten invariant
% (\cf \cite{Givental-Equivariant-GW} and
% \cite{Givental-Coates-2007-QRR} or \cite[Section
% 11.2.1]{Cox-Katz-Mirror-Symmetry} and \cite{Pand-after-Givental}).

Let $\ell$ be in $\N$ and $d$ be in $H_2(X,\Z)$.  Denote by $X_{0,\ell,d}$
the moduli space of stable maps of degree $d$ from rational curves
with $\ell$ marked points to $X$.  The universal curve over $X_{0,\ell,d}$
is $X_{0,\ell+1,d}$
\begin{displaymath}
  \xymatrix{X_{0,\ell+1,d}\ar[d]^-{\pi}\ar[rr]^-{\ev_{\ell+1}}&&X \\ X_{0,\ell,d}}
\end{displaymath}
where $\pi$ is the map that forgets the $(\ell+1)$-th point and stabilises,
and $\ev_{\ell+1}$ is the evaluation at the $(\ell+1)$-th marked point.
% Recall that a \textit{convex bundle} $\gn$ on $X$ is a vector bundle
% such that, for any stable map $f:C\ra X$ where $C$ is a rational nodal
% curve, $H^{1}(C,f^*\gn)=0$.  Notice that if $\gn$ be a globally
% generated vector bundle (not necessarily splitted) of rank $b$ then
% $\gn$ is convex (\cite[Lemma
% 10]{Fulton-Pandharipande-1997-Notes-stable-maps}) and the sheaf
% $\gn_{0,\ell,d}:=R^0\pi_*\ev_{\ell+1}^*\gn$ is locally free of rank
% $\int_{d}c_{1}(\gn)+b$.
By Lemma 10 in \cite{Fulton-Pandharipande-1997-Notes-stable-maps}) the sheaf
 $\vb_{0,\ell,d}:=R^0\pi_*\ev_{\ell+1}^*\vb$ is locally free of rank
 $\int_{d}c_{1}(\vb)+k$.

 For $j$ in $\{1, \ldots ,\ell\}$, we define the surjective morphism $\vb_{0,\ell,d}\to
 \ev_{j}^*\vb$ by evaluating the section at the $j$-th marked point. We define $\vb_{0,\ell,d}(j)$
 to be the kernel of this map~; that is, we have the following exact sequence
  \begin{align}\label{eq:suite,exact,E,0,l,d}
    \xymatrix{
0\ar[r]&\vb_{0,\ell,d}(j)\ar[r]&\vb_{0,\ell,d}\ar[r]&\ev_{j}^*\vb\ar[r]&0}
  \end{align}
  For any $j\in\{1, \ldots ,\ell\}$ the bundle $\vb_{0,\ell,d}(j)$ has
  rank $\int_{d}c_{1}(\vb)$.
  For $i\in\{1,\ldots,\ell\}$, denote by $\psi_{i}$ the first Chern class of the line bundle
  on $X_{0,\ell,d}$ whose fiber at a point $(C,x_{1}, \ldots ,x_{\ell},$ $
  f:~C\to X)$ is the cotangent space ${\gt^*{C}}_{x_{i}}$. 
  
  \begin{defn}\label{defi:twisted-GW-invariants}
    Let $\ell$ be in $\N$, $\gamma_{1}, \ldots ,\gamma_{\ell}$ be
    classes in $ H^{2*}(X)$, $d$ be in $H_2(X,\Z)$ and $(m_{1}, \ldots
    ,m_{\ell})$ be in $ \nn^{\ell}$. For $j$ in $ \{1, \ldots
    ,\ell\}$, the ($j$-th) \textit{twisted Gromov-Witten invariant
      with descendants}  is defined  by
    \begin{align*}
      \left\langle \tau_{m_{1}}(\gamma_{1}), \ldots
        ,\widetilde{\tau_{m_{j}}({\gamma}_{j})}, \ldots ,
        \tau_{m_{\ell}}(\gamma_{\ell})\right\rangle_{0,\ell,d}:=\int_{[X_{0,\ell,d}]^{\vir}}
      c_\top
      (\vb_{0,\ell,d}(j))\prod_{i=1}^{\ell}\psi_{i}^{m_{i}}\ev_{i}^*\gamma_{i}
    \end{align*}
    where $\ev_{i}:X_{0,\ell,d}\to X$ $(1\leq i\leq \ell)$ is the
    evaluation morphism to the $i$-th marked point and
    $[X_{0,\ell,d}]^{\vir}$ is the virtual class of $X_{0,\ell,d}$ (see \cite{Fatenchi-Behrend-intrinsic-normal-cone}).
  \end{defn}

\subsubsection{Twisted quantum product}%\ 
\label{subsubsec:twisted_quantum_product}
%%%%%%%%%%%%%%%%%%%%%%%%%%%%%%%%%%%%%%%%%%%%%%%%%%%%%%%%%%%%%%%%%%%%%%
%%%%%%%%%%%%%%%%%%%%%%%%%%%%%%%%%%%%%%%%%%%%%%%%%%%%%%%%%%%%%%%%%%%%%%

% The quantum product depends on the parameter $\tau_2$ in $H^2(X)$.  The Picard group $\Pic(X)$ acts
% on $ H^2(X)$ in the following way~: for $\lb$ in $\Pic(X)$, $\lb.\tau_2=\tau_{2}+2\sqrt{-1} \pi
% c_{1}(\lb)$. The number $e^{\int_d \tau_2}$ being invariant by this action, the quantum product is
% naturally defined over $ H^2(X)/\Pic(X)=H^2(X)/2\sqrt{-1}\pi H^2(X,\Z)$.

% Let us extend the locus of the parameter. 

\begin{notn}\label{notn:La,Pi,specmori,spectoric}
Denote by $\mori{X}\subset H_2(X,\Z)$ the Mori cone of $X$, generated as a semi-group by numerical
classes of irreducible curves in $X$
$$
\mori{X}=\left\{\sum_{\stackrel{\text{\tiny $C$ irreducible curve,}}{\text{\tiny finite sum}}}n_C [C],\quad  n_C\in \N, [C]\text{ numeric class of $C$}\right\}.
$$  
The semigroup algebras of $\mori{X}$ and $H_2(X,\Z)$ will be
respectively denoted by $\La$ and $\Pi$~:
$$
 \La=\C[\mori{X}]=\C[Q^d, d\in \mori{X}],\qquad \Pi=\C[H_2(X,\Z)]=\C[Q^d, d\in H_2(X,\Z)],
$$
where $Q^d$ are indeterminates satisfying relations~: $Q^d.Q^{d'}=Q^{d+d'}$. Associated schemes to $\La$ and $\Pi$ are~: 
$$
\specmori:=\Spec \La,\qquad \mathbf{T}:=\Spec \Pi.
$$
\end{notn}
 The scheme $\specmori$ is an irreducible, possibly singular, affine
 variety of dimension $r$. Points of $\specmori$ are
 characters
% \footnote{\label{footnote1} By a character of a semi-group
%   $R$ of $H_2(X,\Z)$ we mean an application $q~: R\lra \C$ such that
%   $q(0)=1$ and $q(d+d')=q(d).q(d')$ for any $d,d'$ in $R$. If $R$ is a
%   group the image of $q$ is in $\C^*$.  If $q$ is such a character, we
%   will write $q^d:=q(d)$.  A character $q$ of a semi-group $R$ gives a
%   complex point $\Spec \C\lra \Spec \C[R]$ which will also be denoted
%   by $q$~; this correspondence is a bijection.  Notice that, if $d$ is
%   in $R$, $Q^d$ is a function on $\Spec_\C \C[R]$ and we have~:
%   $Q^d(q)=q^d$.  }
 of $\mori{X}$. If $q$ is such a character, denote
 by $q^d$ its evaluation on $d$ in $\mori{X}$.  Since $X$ is
 projective, the Mori cone is strictly convex and there exists a unique
 character sending any $d$ in $ \mori{X}\setminus\{0\}$ to $0$~; it corresponds to the maximal ideal 
 $\langle Q^d, d\in \mori{X}\setminus\{0\}\rangle$. We
 will denote this point by $\0$.

 The scheme $\spectoric\simeq (\C^*)^r$ is an algebraic torus of rank $r$.
 In \cite{Cox-Katz-Mirror-Symmetry}, the point $\0\in\specmori\setminus \spectoric$ is called the \emph{large radius
 limit} of $\spectoric$.

% The scheme $\mathbf{T}\subset \specmori$ is a torus of dimension
% $r=\rank H_2(X,\Z)$.  A point of $\mathbf{T}$ is a character of
% $H_2(X,\Z)$. We will identify $\mathbf T$ and $ H^2(X)/2\sqrt{-1}\pi
% H^2(X,\Z)$ via the natural surjective morphism of complex variety~:
% \begin{align}\label{eq:map:Psi,T->H2(X)/H2(Z)}
% \Psi~:~ H^2(X,\C) & \lra  \spectoric\\
% \tau &\longmapsto  q_\tau := \left[d\in H_2(X,\Z)\mapsto q_\tau^d=e^{\int_d \tau}\right]\nonumber
% \end{align}
% The kernel of $\Psi$ is $2\sqrt{-1}\pi H^2(X,\Z)$.
% Thus, the large radius limit $0$ in $\specmori\supset \spectoric$ is a limit in $ H^2(X)/2\sqrt{-1}\pi H^2 (X,\Z) $.

The \textit{small twisted quantum product} can now be defined.
Let $q$ be in $\specmori$ and   $\gamma_{1},\gamma_{2}$ be in $ H^{2*}(X)$. 
The twisted small  quantum product is defined by
\begin{align}
  \label{defi:quantum_product_parameter_q} 
    \gamma_{1}\twprod{q}\gamma_{2}&:= \sum_{a=0}^{s-1}\sum_{d\in H_2(X,\Z)}
      q^d\left\langle \gamma_{1},
      \gamma_{2},\widetilde{T}_{a}\right\rangle_{0,3,d}T^{a}
\end{align}
  whenever this sum is convergent.
Notice that this twisted quantum product is the non-equivariant limit of
$\bullet^{\mathbf{e}_{\lambda}}_{\tau}$ in \cite[p.5]{iritani_quantum_2011}. Remark 2.2 in
\cite{iritani_quantum_2011} implies that the twisted quantum product $\bullet^{\tw}_{q}$ is associative,
  commutative, with unity $T_{0}:=\mathbf{1}$.

% Definitions (\ref{defi:quantum_product_parameter_q}) and 
% (\ref{eq:35}) are compatible~: For any $\tau_2$ in $
% H^2(X)$, $\Psi(\tau_2)$ is in $\mathbf T$ and $
% \gamma_{1}\twprod{\tau_2}\gamma_{2}=
% \gamma_{1}\twprod{\Psi(\tau_2)}\gamma_{2}$.

\begin{assumption} \label{ass:convergenceset}
We assume that $(\omega_X \otimes \lb_1\otimes\cdots\otimes \lb_k)^{\dual }$ is nef.
\end{assumption}

Iritani proves in \cite{iritani_convergence}, that under this assumption, there exists an open subset $\convmori$ of $\specmori$ containing $\mathbf{0}$ such that~: $$\forall q\in \convmori, \forall \gamma_1,\gamma_2\in H^{2*}(X), \gamma_1\twprod{q}\gamma_2  \text{\ is convergent.}$$

% This assumption is easily shown to be true when the line bundle
% $(\omega_X \otimes \lb_1\otimes\cdots\otimes \lb_k)^{\dual }$ is ample, that is when the complete intersection
% variety $Z$ defined by $\vb$ is Fano.  In other cases, such as
% Calabi-Yau subvarieties of toric varieties considered below, one may

\begin{notn}\label{notn:parametre_M}
We denote by $\convtoric$ the complex nonsingular variety $\convtoric:=\convmori\cap \spectoric$. 
\end{notn}
%We have~:
%$$
%\begin{array}{rcclcclcc}
%\mathbf{0}&\in &\convmori &\text{\tiny (quantum product is convergent)} & \subset & \specmori & \text{\tiny ($=\Spec \C[\mori{X}]$)}\\
%& & \cup & &  & \cup\\
% & & \convtoric &\text{\tiny ($=\convmori\cap \spectoric$)} &\subset & \spectoric &\text{\tiny  (rank $r$ torus)}& \not\ni &\0
%\end{array}
%$$
% As a convention, we will denote neighbourhood of $0$ in $\specmori$ by an overlined capital letter, and its intersection with $\spectoric$ by the same capital letter without overlining ($\specmori$ is a compactification of $\spectoric$ in the neighbourhood of the large radius limit).

% \begin{proof}This is a classical proof, as soon as the twisted
%   Gromov-Witten axioms are known. The twisted axioms are stated in Appendix
%   \ref{sec:twisted-axiom-for-GW}.   Such proves are given by
%    Pandharipande in \cite{Pand-after-Givental}, Proposition 3, for a
%    smooth hypersurface of $\pp^n$ and by Iritani in Remark 2.2. of
%    \cite{iritani_quantum_2011}, in the general case.
% \end{proof}

%%%%%%%%%%%%%%%%%%%%%%%%%%%%%%%%%%%%%%%%%%%%%%%%%%%%%%%%%%%%%%%%%%%%%%
%%%%%%%%%%%%%%%%%%%%%%%%%%%%%%%%%%%%%%%%%%%%%%%%%%%%%%%%%%%%%%%%%%%%%%
\subsubsection{Twisted quantum $D$-module}%\ 
%%%%%%%%%%%%%%%%%%%%%%%%%%%%%%%%%%%%%%%%%%%%%%%%%%%%%%%%%%%%%%%%%%%%%%
%%%%%%%%%%%%%%%%%%%%%%%%%%%%%%%%%%%%%%%%%%%%%%%%%%%%%%%%%%%%%%%%%%%%%%

% \begin{proof}
%   Commutativity follows from the definition.  Associativity follows from twisted
%   WDVV equation (see Proposition \ref{prop:tw-WDVV} equality \eqref{eq:WDVV2}) .
%   The unity follows from the twisted fundamental class equation (see Proposition
%   \ref{prop:tw-SE}). Notice that if $d=0$ then $\rank(L_{0,l,d}(j))=0$.
% \end{proof}

Let $(B_1,\ldots, B_r)$ be the basis of $H_2(X,\Z)$ fixed in Notation \ref{notn:base_of_cohomology}.
For $a\in\{1,\ldots, r\}$, put $q_a=Q^{B_a}$. 
We have: 
$$\Pi=\C[H_2(X,\Z)]\isom\C[q_1^{\pm},\ldots,
q_r^{\pm}] ;$$ 
If $d=\sum_{a=1}^r d_a B_a$ we get
$Q^d=\prod_{a=1}^r q_a^{d_a}$. Viewing the $q_a$'s as coordinates of $\spectoric$ we get, for any $q\in\spectoric$,
$q^d= \prod_{a=1}^r q_a^{d_a}$.

Let $z$ be another variable~;  we write $\C$ for $\Spec \C[z]$. We define $r+1$ differential operators on 
$\spectoric\times \C$ by~:
 \begin{align*}
\delta_{a}&:=q_{a}\partial_{q_{a}}, a\in\{1,\ldots, r\},\mbox{ and } \delta_{z}:=z\partial_{z}. 
\end{align*}

%Recall that $t_{0}$ is the coordinate on $H^{0}(X)$.
%\begin{notn}\label{notn:twisted,connection}%[\cf. Section 2.1 of \cite{Pand-after-Givental}]
We denote by $F$ the trivial holomorphic vector bundle of fiber $H^{2*}(X)$ over $\convtoric\times
\C$ % \ie
% $$F:= \left[\ H^{2*}(X)\times  \left( \convtoric\times \cc\right) \to \left( \convtoric\times \cc\right)\ \right] $$
together with the following meromorphic connection:
  \begin{align}\label{eq:definition_de_nabla_twist}
%    \nabla_{\partial_{t_{0}}}:=\partial_{t_{0}}+\frac{1}{z}\mathbf 1
 %   \twprod{q}, \ \
    \nabla_{\delta{a}}:=\delta_{a}+\frac{1}{z}T_{a}\twprod{q}   ,\ \
    \nabla_{\delta_{z}}&:=\delta_{z}-\frac{1}{z}\eulerclass\twprod{q}+\mu
  \end{align} where $ \mu$ is the diagonal morphism defined by
  $\mu(T_{a}):=\frac{1}{2}\left({\deg(T_{a})}-(\dim_{\cc}X-\rank
    \vb)\right)T_{a} $ and $ \eulerclass:=c_{1}(\tang_X) -c_{1}(\vb)$. 
The couple $(F,\nabla)$ is called the \textit{twisted Quantum $D$-module} of $(X,\vb)$ and denoted by $\QDM(X,\vb)$.
%This global section
%  $\eulerclass$ of $F$ corresponds to the Euler field. Notice that the
  % twisted product $\twprod{q}$ does not depend on $t_{0}$ because of
  % the twisted fundamental class Axiom (\cf Axiom
  % \ref{item:tw-SE}).
%\end{notn}

% In the untwisted case, it is known that $\nabla$ is a flat connection
% and its flat sections can be described explicitly. Let us give the
% equivalent property in the twisted case.
We define a multi-valued
meromorphic section $L^{\tw}$ of $\Hom(F,F)$ by~:  
\begin{align}\label{eq:definition_of_L}
 L^{\tw}(q,z)\gamma= q^{-T/z}\gamma- \sum_{a=0}^{s-1}\sum_{\stackrel{H_2(X,\Z)}{d\neq 0}}q^{d}\left\langle
      \frac{q^{-T/z}\gamma}{z+\psi},\widetilde{T_{a}}\right\rangle_{0,2,d} T^{a}
\end{align}
where $\frac{1}{z+\psi}:=\sum_{k\in \nn} (-1)^{k}\psi^{k}z^{-k-1}$ and
$q^{-T/z}:=q^{-T_1/z}\ldots q^{-T_r/z}:=e^{-z^{-1}\sum_{a=1}^{r}T_{a}\log (q_{a})}$ and $\log(q_a)$ is the multi-valued function, or any determination of the logarithm  on a simply connected open subset of $\convtoric$.

% The following Proposition is the ``twisted'' version
% of Proposition 2.4 in \cite{Iritani-2009-Integral-structure-QH}.
Define a pairing by: $(\gamma_{1},\gamma_{2})^{\tw}:=\int_{X}\gamma_{1}\cup\gamma_{2}\cup
\ctopvb.$
This pairing is degenerated and its kernel is $\ker \mc$ where $\mc:  H^{2*}(X)\to  H^{2*}(X)$ sends
$\alpha$ to $\ctopvb \cup \alpha.$
\begin{prop}\label{prop:nabla,flat+Ltw}%\ 
 \begin{enumerate}
  \item\label{item:5bis}  The connection $\nabla$ is  flat. 
  \item\label{item:6bis} For $a$ in $\{1, \ldots ,r\}$ and $\gamma\in H^{2*}(X)$ we have
\begin{align*}
  \nabla_{\delta_{a}}L^{\tw}(q,z)\gamma&=0,  &
  \nabla_{\delta_{z}}L^{\tw}(q,z)\gamma=L^{\tw}(q,z)\left(\mu-\frac{c_{1}(\tang_X)-c_{1}(\vb)}{z}\right)\gamma 
\end{align*}
\item\label{item:7bis} For any endomorphism $u$ of $H^{2*}(X)$, we put
  $z^{u}:=\exp(u\log z)$. The multi-valued cohomological function
  $L^{\tw}(q,z)z^{-\mu}z^{c_{1}(\tang_X)-c_{1}(\vb)}$ is
  a fundamental solution of $\nabla$.
\item For any $\gamma_{1},\gamma_{2} \in H^{2*}(X)$, we have
  \begin{displaymath}
    (L^{\tw}(q,-z)\gamma_{1},L^{\tw}(q,z)\gamma_{2})^{\tw}=(\gamma_{1},\gamma_{2})^{\tw}.
  \end{displaymath}
\end{enumerate}
\end{prop}

% Notice that, as a fundamental solution,  $L^{\tw}$ is convergent above $H^{0}(X)\times\convtoric\times \cc$.
 
 \begin{proof}
   This proof is completely parallel to the one of Proposition 2.4 in
   \cite{Iritani-2009-Integral-structure-QH}, using the twisted axioms (see Appendix \ref{sec:twisted-axiom-for-GW}).
 \end{proof}

\subsection{Quantum $\gd$-module  for complete intersection subvarieties}\label{subsec:geom-interpr-reduc} 
%\ %%%%%%%%%%%%%%%%%%%%%%%%%%%%%%%%%%%%%%%%%%%%%%%%%%%%%%%%%%%%%%%%%%%%%%Âµ
%%%%%%%%%%%%%%%%%%%%%%%%%%%%%%%%%%%%%%%%%%%%%%%%%%%%%%%%%%%%%%%%%%%%%%Âµ

\begin{assumption}\label{ass:ample,dim,3} In this section, we assume that $\dim_\C X\geq k+3$ and 
  that the line bundles $\lb_1,\ldots,\lb_k$ are ample. This makes it
  possible to use Hyperplane and Hard Lefschetz Theorems.
%  \footnote{Theorem 2.3.1 of \cite{Cataldo-Migliorini-2002-LEF}
 %   generalised it to ``lef'' line bundles that is a positive multiple
 %   is generated by its global sections and which define a semismall
  %  map}
\end{assumption}

\begin{notn}\label{notn:generic_section}
Fix a generic section of $\vb$, and denote by $Z$ the projective subvariety defined by this section.
 By Bertini's theorem, $Z$ is a smooth complete intersection subvariety of $X$.
Denote by $\iota:Z\hookrightarrow X$ the corresponding closed embedding.
\end{notn}

By Lefschetz theorem we have
\begin{align}
  \label{eq:Lef}
  H^{2*}(Z)=\im \iota^*\oplus \ker \iota_{*}
\end{align}
and $\ker \iota_{*}\subset H^{\dim_{\cc}Z}(Z)$.
We put $ H^{2*}_{\amb}(Z):=\im \iota^*$, this is the part of the cohomology of $Z$ coming from the ambient space $X$.
% We have the following commutative diagram
% \begin{align}\label{diag:lefchetz}
%   \xymatrix{H^{2*}(X)\ar[rr]^{ \mc} \ar@{->>}[rd]^{p} \ar@{->>}@/_1pc/ 
% [rdd]_{\iota^{*}}& &H^{2*}(X)  \ar@{->>}[ld]_{p}\\
% &\overline{H^{2*}(X)}\ar@{-->}[d]^-{f}\\
% & H^{2*}_{\amb}(Z)\ar\ar@/_1pc/@{^{(}->} 
% [ruu]_{\iota_{*}}
% }
% \end{align}
% where $p$ is the natural projection and $f:\overline{\gamma}\mapsto
% \iota^*\gamma$. 
% By the decomposition \eqref{eq:Lef}, the morphism
% $f$ is an isomorphism. 
We have an isomorphism $H^2(X)\simeq  H^2(Z)$.

\begin{rem} \label{rem:mavlyutov} It should be possible to improve
Assumption \ref{ass:ample,dim,3}, at least for toric varieties. For example, if 
$X$ is a toric projective variety of dimension at least $3$, $k=1$ and $\lb_1$ is a nef (not necessary ample) line bundle on $X$,
then Theorem 5.1 of \cite{mavlyutov_chiral_2000}  ensures that $Z$ is a smooth connected hypersurface satisfying~: $H^{2*}(Z)=\im \iota^*\oplus \ker \iota_*$.
\end{rem}

\begin{prop}[See Corollary 2.3 in \cite{iritani_quantum_2011}]\label{prop:group-morphism} Using Notation \ref{notn:generic_section}, and under Assumption \ref{ass:ample,dim,3},
 for any
  $\gamma_1,\gamma_2\in H^{2*}(X)$%, $\tau_2\in H^2(X)$,
$$
\iota^*(\gamma_1\twprod{q}\gamma_2)=\iota^*(\gamma_1)\Zprod{q}\iota^*(\gamma_2),
$$
where $\Zprod{}$ is the quantum product on $Z$.\qed
\end{prop}
% \begin{proof}
% The proof is given in Proposition 4 of
% \cite{Pand-after-Givental}, for a smooth hypersurface of $\pp^n$.  The
% general case is treated by Iritani (\cite{iritani_quantum_2011},
% Corollary 2.3.) using functoriality of virtual classes (\cf
% \cite{Kim-Kresch-Pantev-lci-virtual-class}).
% \end{proof}

We define the trivial vector bundle, denoted by $F^{Z}_{\amb}$, of fiber $H^{2*}_{\amb}(X)$ over
$\convtoric_{Z}\times \cc$ where $\convtoric_{Z}$ is the subset of $H^2(Z,\C)/\Pic(Z)$ where the
quantum product on $Z$ is convergent\footnote{We use the same parameter $q$ because of the
  isomorphism $\iota^*: H^2(X)\simeq H^2(Z)$}.  The connection $\nabla^{Z}$ is defined via the same
formula as $ \nabla$ with the quantum product of $Z$ and $\eulerclass:=c_{1}(\tang_Z)$ and
  \begin{displaymath}
    \mu^{Z}(\psi_{a})=\frac{1}{2}\left({\deg (\psi_{a})}-{\dim_{\cc}Z}\right)\psi_{a}.
  \end{displaymath}
where $(\psi_{a})$ is a basis of $H^{2*}(Z)$. Proposition \ref{prop:group-morphism} implies that
this bundle is stable by $\nabla^{Z}$. We denote by $\QDM_{\amb}(Z)=(F_{\amb}^{Z},\nabla^{Z})$.

% Moreover, on $\QDM(Z)$ Iritani defined \textit{the
%   $\widehat{\Gamma}$-integral structure} (see Definition 2.9
% \cite{Iritani-2009-Integral-structure-QH}) via
% $\Zintegralmap(K(Z))$ where for any $w$ in $K(Z)$, he puts
% \begin{displaymath}
% \Zintegralmap(w)=(2\pi)^{-(n-k)/2}L^{Z}(t_{0},q,z)z^{-{\mu}^{Z}}z^{c_{1}(\tang_Z)}{\widehat{\Gamma}(\tang_Z)}(2\sqrt{-1}\pi)^{\deg /2} {\Ch}({w}).
% \end{displaymath}
% In Proposition 2.10 of \cite{Iritani-2009-Integral-structure-QH}, he proves that $\Zintegralmap(w)\otimes_{\zz}\cc$ is the set of flat sections of $\QDM(Z)$. 

% We consider the trivial sub-bundle of $F^{Z}$ whose fibers are $
% H^{2*}_{\amb}(Z)$. This sub-bundle is stable by $\nabla^{Z}$ and the
% pairing is still non-degenerated on it.  We denote $\QDM_{\amb}(Z)$
% this sub-quantum $\gd$-module.  By Proposition
% \ref{prop:group-morphism}, the base space of this bundle $F^{Z}$ could
% be restricted to $\convtoric\times\cc$.  
%We put $K_{\amb}(Z):=\iota^{*}K(X)$.  We have that
% $K_{\amb}(Z)\otimes_{\zz}\cc$ is isomorphic to $H^{2*}_{\amb}(Z)$ via the Chern character. So $\Zintegralmap(K_{\amb}(Z))\otimes_{\zz}\cc$ is the set of flat sections of $\QDM_{\amb}(Z)$ that is $\Zintegralmap(K_{\amb}(Z))$ define a
% $\widehat{\Gamma}$-integral structure on $\QDM_{\amb}(Z)$.

% The integral structure put
% on ${\QDM}(X,\vb)$ in \S \ref{subsubsec:integral,structure} is compatible with the one defined by Iritani, that is we have the following theorem.

 \begin{cor} \label{cor:surj,twisQDM,lciQDM}
Using Notation \ref{notn:generic_section}, and under Assumption \ref{ass:ample,dim,3}. The morphism
$\iota^{*}$ induces a surjective morphism $\iota^{*}:\QDM(X,\vb) \twoheadrightarrow \QDM_{\amb}(Z)$. 
\end{cor}

\begin{proof}
  It is clearly a surjective morphism of vector bundles.
 Proposition \ref{prop:group-morphism} implies that
 $\iota^{*}(\nabla_{\delta_{a}}\gamma)=\nabla^{Z}_{\delta_{a}}\iota^{*}\gamma$. The adjunction formula gives~:
$c_{1}(\tang_Z)=\iota^*(c_{1}(\tang_X) -c_{1}(\vb))$. Since the dimension of
$Z$ is the dimension of $X$ minus the rank of $\vb$, we deduce that
$\mu^{Z}(\iota^{*}\gamma)=\iota^{*}({\mu}({\gamma}))$.
This implies  $\iota^{*}(\nabla_{\delta_{z}}\gamma)=\nabla^{Z}_{\delta_{z}}\iota^{*}\gamma$.
\end{proof}

\subsection{Reduced quantum $ D$-module }
\label{subsec:reducedQDM}
%%%%%%%%%%%%%%%%%%%%%%%%%%%%%%%%%%%%%%%%%%%%%%%%%%%%%%%%%%%%%%%%
%%%%%%%%%%%%%%%%%%%%%%%%%%%%%%%%%%%%%%%%%%%%%%%%%%%%%%%%%%%%%%%%

Consider the quotient $\overline{ H^{2*}(X)}:= H^{2*}(X)/\ker \mc$ and call it the \textit{reduced cohomology ring of $(X,\vb)$}.
In this section, we define a "reduced"  quantum product on $\overline{ H^{2*}(X)}$, which enables us to define a "reduced" quantum $\gd$-module.

This reduced quantum $D$-module turns out to be isomorphic to the ambient part of the quantum $D$-module of the subvariety $Z$ defined in Subsection \ref{subsec:geom-interpr-reduc}. 

\smallskip

% In this subsection we define the reduced quantum $\gd$-module, denoted by
% $\overline{\QDM}(X,\vb)$, which is a 
% quadruple
% $\left(\overline{F},\overline{\nabla},\overline{S},\overline{F}_{\zz}\right)$. The pairing $\overline{S}$ is
% non-degenerated.

%\medskip

% Recall that $\mc$ is the endomorphism
% \begin{align*}
% \mc:  H^{2*}(X)&\lra  H^{2*}(X)\\
%  \al &\longmapsto \ctopvb \cup \alpha.
% \end{align*}

Since $ \mc$ is
a graded morphism, the reduced cohomology ring $\overline{H^{2*}(X)}= H^{2*}(X)/\ker \mc$ is naturally
graded. For $\gamma\in H^{2*}(X)$, we denote by $\overline{\gamma}$ its
class in $\overline{H^{2*}(X)}$.
Denote by $\overline{F}$ the trivial bundle with fiber $ \overline{ H^{2*}(X)}$ over
$\convtoric\times \cc$.  
For any $\gamma_{1},\gamma_{2}\in H^{2*}(X)$, define the reduced pairing $(\cdot,\cdot)^{\red}$ which is a bilinear form  on
$\overline{H^{2*}(X)}$ by 
\begin{align}\label{eq:defi,()red}
  (\overline{\gamma}_{1},\overline{\gamma}_{2})^{\red}:=(\gamma_{1},\gamma_{2})^{\tw}.
\end{align}
The reduced pairing is a well defined and a non degenerate bilinear form.
% We define the pairing $\overline{S}$ as we did  for $S$ but changing
% $(\cdot,\cdot)^{\tw}$ by $(\cdot,\cdot)^{\red}$ (\cf before Proposition \ref{prop:S,plat,Ltw,isometrie}).  From
% \eqref{eq:defi,()red}, for any $s_{1},s_{2}\in \Gamma(H^{0}(X)\times\convtoric\times\cc,\gf)$, we deduce that
% \begin{align}\label{eq:Sred,Stw}
%   \overline{S}(\overline{s}_{1},\overline{s}_{2})&=S(s_{1},s_{2})
% \end{align}
Put $s'=\dim_\C\overline{H^{2*}(X)}$.
Let $(\phi_{0}, \ldots ,\phi_{s'-1})$ be a homogeneous basis of
$\overline{H^{2*}(X)}$ and denote $(\phi^{0}, \ldots ,\phi^{s'-1})$ its
dual basis with respect to $(\cdot,\cdot)^{\red}$.
%\begin{defn}\label{defi,reducedGW}
Let ${\gamma}_{1}, \ldots ,{\gamma}_{\ell}$ be classes in ${H^{2*}(X)}$.
Let $d$ be in $H_2(X,\Z)$. 
Using Definition \ref{defi:twisted-GW-invariants}, we define the \textit{reduced
      Gromov-Witten invariant} by
  \begin{align*}
    \langle\overline{\gamma}_{1}, \ldots
    ,\overline{\gamma}_{\ell}\rangle_{0,\ell,d}^{\red}:=
    \langle{\gamma}_{1}, \ldots
    ,\widetilde{\ctopvb{\gamma}}_{\ell}\rangle_{0,\ell,d}=\int_{[X_{0,\ell,d}]^{\vir}}\ctop(\mathcal{E}_{0,n,d})\prod_{i=1}^{\ell}
    \ev_{i}^{*}\gamma_{i}
\end{align*}
%\Etienne{Je pense que le referee n'a pas compris la formule}
  By the twisted $S_{\ell}$-symmetric axiom (\cf Axiom
  \ref{item:twisted,symmetry}), the reduced Gromov-Witten invariants
  are well defined on the class in $\overline{H^{2*}(X)}$.
  Notice that the reduced Gromov-Witten invariants are symmetric with respect to  the $\ell$ entries.

The \textit{reduced quantum product} is
  \begin{align*}
    \overline{\gamma}_{1}\redprod{q}\overline{\gamma}_{2}:=
    \sum_{a=0}^{s'-1}\sum_{d\in H_2(X,\Z)}q^{d}\left\langle
      \overline{\gamma}_{1},
      \overline{\gamma}_{2},{{\phi_{a}}}\right\rangle^{\red}_{0,3,d}{\phi^{a}}.
  \end{align*}
%\begin{rem}
  The convergence domain of $\redprod{q}$ contains
  $\convtoric$. We will restrict ourselves to $\convtoric$.
Define the following connection on the trivial bundle $\overline{F}$~:
\begin{align*}
  \forall a\in \{1, \ldots ,r\},\quad \overline{\nabla}_{\delta{a}}&:=\delta_{a}+\frac{1}{z}\overline{T}_{a}\redprod{q}
  \\
  \overline{\nabla}_{\delta_{z}}&:=\delta_{z}-\frac{1}{z}\overline{\eulerclass}\redprod{q}+\overline{\mu}
\end{align*}
where $\overline{\mu}$ is the diagonal morphism defined by
$\overline{\mu}({\phi_{a}}):=\frac{1}{2}\left({\deg(\phi_a)}-({\dim_{\cc}X-\rank
    \vb})\right){\phi_{a}}$ and
$\overline{\eulerclass}:=\overline{c_{1}(\tang_X)} -\overline{c_{1}(\vb)}$.
\begin{defn}\label{defi:reduced,QDM}
  The couple $(\overline{F},\overline{\nabla})$ is called the \textit{reduced quantum $D$-module of $(X,\vb)$} and denoted by $\overline{\QDM}(X,\vb)$.
\end{defn}

\begin{prop}\label{prop:isomorphisme_QDM_reduit_et_ambiant}
%\overline{\QDM}(X,\vb)\ar[ru]^{\sim}_{f}
\begin{enumerate}
\item The connection $\overline \nabla$ is flat.
\item Under assumption \ref{ass:ample,dim,3}, let $Z$ be the  subvariety defined by a generic section of $\vb$. There exists an isomorphism of $\mathcal{D}$-modules $f:\ov{\QDM}(X,\vb)\isom {\QDM}_{\amb}(Z)$ making the following diagram 
commutative:
\begin{displaymath}
    \xymatrix{%
    &\QDM(X,\vb)\ar@{->>}[dr]^{\iota^{*}} \ar@{->>}[dl]_{p}&\\ 
    \overline{\QDM}(X,\vb)\ar[rr]^{\sim}_{f} & &\QDM_{\amb}(Z) % 
    }
  \end{displaymath}
\end{enumerate}
where $p$ is the natural projection on the quotient.
\end{prop}

\begin{proof}
For any $\gamma_1,\gamma_2\in H^{2*}(X)$ and any $a\in\{0,\ldots,s-1\}$ we have~:
\begin{align*}%\label{eq:36}
     \overline{\gamma_{1}\twprod{q}\gamma_{2}}&=\overline{\gamma}_{1}\redprod{q}\overline{\gamma}_{2}\mbox{
     \quad and \quad}
     \overline{\mu(T_{a})}=\overline{\mu}(\overline{T_{a}}).
\end{align*}
It follows that, for any $\gamma \in H^{2*}(X)$,
  \begin{align}\label{eq:connexions_compatibles}
      \overline{\nabla \gamma}=\overline{\nabla}\overline{\gamma}.
  \end{align}
and  $\ov\nabla$ is flat since $\nabla$ is.

As for the second point, consider the following diagram, where we make use 
of notations of \S. \ref{subsec:geom-interpr-reduc}. 
\begin{align}\label{diag:lefchetz}
  \xymatrix{H^{2*}(X)\ar[rr]^{ \mc} \ar@{->>}[rd]^{p} \ar@{->>}@/_1pc/ 
[rdd]_{\iota^{*}}& &H^{2*}(X)  \ar@{->>}[ld]_{p}\\
&\overline{H^{2*}(X)}\ar@{-->}[d]^-{f}\\
& H^{2*}_{\amb}(Z)\ar\ar@/_1pc/@{^{(}->} 
[ruu]_{\iota_{*}}
}
\end{align}
The morphism $f$ is well defined by $f:\overline{\gamma}\mapsto
\iota^*\gamma$. 
By the decomposition \eqref{eq:Lef}, 
$f$ is an isomorphism. 

This diagram and Corollary \ref{cor:surj,twisQDM,lciQDM} gives the required isomorphism between vector bundles~; Formula (\ref{eq:connexions_compatibles}) 
ensures that the connections are compatible. 
\end{proof}

\begin{rem}
The reduced quantum $D$-module does exist even if the assumption \ref{ass:ample,dim,3} is not satisfied~; that is if the subvariety $Z$ is not well defined. 
It is used in \cite{Sevenheck-GKZ-log-Frob-manifold}.
\end{rem}

We now come to the reduced fundamental solutions.
   
\begin{lem}\label{lem:L_et_kerm_c}
For any $(q,z)$ in $\convtoric\times \C$, we have~: $\ L^{\tw}(q,z)(\ker \mc)=\ker \mc.$
\end{lem} 

 \begin{proof}
   Let $\gamma$ be in $\ker \mc$ and $\al\in H^{2*}(X)$. Since $L^{\tw}(q,z)$ is an
   automorphism of $H^{2*}(X)$ and $\ker \mc$ is the kernel of the twisted pairing
   $(\cdot,\cdot)^{\tw}$ we find, using Proposition \ref{prop:nabla,flat+Ltw}:
 \begin{align*}
   \left(\al, L^{\tw}(q,z)\gamma\right)^{\tw}&= \left(L^{\tw}(q,-z)^{-1}\al,\gamma\right)^{\tw} = 0.
\end{align*}
Then $L^{\tw}(q,z)\gamma$ belongs to $\ker \mc$.
\end{proof}

This lemma implies that we can define a reduced $L$ function~: for
any $(q,z)\in\convtoric\times \C$ put
\begin{align} \label{eq:defi;Lred}
\overline{L}(q,z)\overline{\gamma}=\overline{L^{\tw}(q,z)\gamma}
\end{align}
% In the same spirit of \S \ref{subsubsec:integral,structure}, we also
% get an induced integral structure on $\overline{\QDM}(X,\vb)$.
% Denote by
% \begin{displaymath}
%   \overline{K(X)}:=K(X)\ / \ \{\grothclass \mid \Ch(\grothclass )\in \ker \mc\}.
% \end{displaymath}
% The Chern character $\Ch:K(X)\to H^{2*}(X)$ induces a \textit{reduced
%   Chern character} $\overline{\Ch}:\overline{K(X)}\to
% \overline{H^{2*}(X)}$ which become an isomorphism after tensored by
% $\cc$.
% For any $\overline{\grothclass }\in
% \overline{K(X)}$, we put 
% \begin{displaymath}
%   \redintegralmap(\overline{\grothclass }):=(2\pi)^{-(n-k)/2}\overline{L}(t_{0},q,z)z^{-\overline{\mu}}z^{\overline{c_{1}(\tang_X\otimes \vb^{\dual})}}\overline{\widehat{\Gamma}(\tang_X)}\overline{\widehat{\Gamma}(\vb)^{-1}}(2\sqrt{-1}\pi)^{\deg /2} \overline{\Ch}(\overline{\grothclass }).
% \end{displaymath}
% In the same spirit of Definition \ref{defi:integral,structure},
% \textit{the reduced $\widehat{\Gamma}$-integral structure} on
% $\overline{\QDM}(X,\vb)$ is given by
% $\redintegralmap(\overline{K(X)})$ and we denote it by $\overline{F}_{\zz}$.

The following corollary follows from Proposition \ref{prop:nabla,flat+Ltw}.
\begin{cor}\label{cor:reduced,quantum,D,module}
We have the following properties.
  \begin{enumerate}
  \item\label{item:9} A fundamental solution of $\overline{\nabla}$ is given by
    $\overline{L}(q,z)z^{-\overline{\mu}}z^{\overline{c_{1}(\tang_X)}-\overline{c_{1}(\vb)} }$.
  \item\label{item:10} For any $\gamma_{1},\gamma_{2}\in H^{2*}(X)$, we have
\begin{align*}
  ({\overline{L}(q,-z)}\overline{s}_{1},{\overline{L}(q,z)}\overline{s}_{2})^{\red}
  &=(\overline{s}_{1},\overline{s}_{2})^{\red} 
%\overline{J^{\tw}}(q,z)&=(\overline{L^{\tw}}(q,z))^{-1}\overline{\mathbf{1}}.
\end{align*}
%\item\label{item:12} For any $\grothclass $ in $K(X)$, we have
%$    \redintegralmap(\overline{\grothclass })=\overline{\twintegralmap(\grothclass )}.$

% \item\label{item:11}
%  For any ${\grothclass }_{1},{\grothclass }_{2}$ in ${K(X)}$, we have
%   \begin{align*}
%     \overline{S}(\redintegralmap(\overline{\grothclass }_{1}),\redintegralmap(\overline{\grothclass }_{2}))&=\int_{X}c_{\topp}(\vb)\Td(\tang_X)\Td(\vb)^{-1}\Ch(\grothclass _{1}\otimes \check{\grothclass _{2}}).
%  \end{align*}
\end{enumerate}
\end{cor}

%%%%%%%%%%%%%%%%%%%%%%%%%%%%%%%%%%%%%%%%%%%%%%%%%%%%%%%%%%%%%%%%
\section{Batyrev algebras for toric varieties with a split vector bundle} 
\label{sec:Batyrev_algebras}
%%%%%%%%%%%%%%%%%%%%%%%%%%%%%%%%%%%%%%%%%%%%%%%%%%%%%%%%%%%%%%%%%

From now on, the smooth projective variety $X$ is a toric variety.
In \cite{batyrev-quantum-1993}, Batyrev constructs an algebra
from the fan of a smooth toric projective variety. If the variety is Fano, this algebra is its quantum cohomology ring. 
In this section, we define and study similar objects for  toric varieties endowed with a split vector bundle. 

%%%%%%%%%%%%%%%%%%%%%%%%%%%%%%%%%%%%%%%%%%%%%%%%%%%%%%%%%%%%%%%%
%%%%%%%%%%%%%%%%%%%%%%%%%%%%%%%%%%%%%%%%%%%%%%%%%%%%%%%%%%%%%%%%
\subsection{Fan for the total space of a split vector bundle}
\label{subsection:Fan}
%%%%%%%%%%%%%%%%%%%%%%%%%%%%%%%%%%%%%%%%%%%%%%%%%%%%%%%%%%%%%%%%
%%%%%%%%%%%%%

Denote by $N$ a $n$-dimensional lattice and by $M$ its dual lattice.
Consider a fan $\Sigma$ of $N_\R=N\otimes \R$ and denote by
$\Sigma(l)$ the set of $l$-dimensional cones of $\Sigma$.  The set of
rays is $\Sigma(1)=\{ \theta_{1}, \ldots ,\theta_{m}\}$,
and for any $\theta\in\Sigma(1)$ we denote by $w_\theta$ the generator
of $ \theta\cap N$.

Let $X$ be the variety defined by $\Sigma$. We assume that $X$ is smooth and projective.

Let 
$\lb_1,\ldots,\lb_k$ be $k$ globally generated line bundles on $X$. Put $\vb=\oplus_{i=1}^k\lb_i$. Let $\tordiv_1,\ldots, \tordiv_k$ be $k$ toric divisors  of $X$, such that $\lb_i \simeq\go(\tordiv_i)$. To any ray $\theta\in \Sigma(1)$,
there is an associated toric Weil divisor denoted by $D_\theta$~; 
we write, in a unique way:
$$
\tordiv_i=\sum_{\theta\in \Sigma(1)} \ell_\theta^i D_\theta,\quad \ell_\theta^i\in\Z,\ i=1,\ldots, k
$$

Consider the $n+k$ dimensional lattice $N':=N\oplus \Z^k$. Let $(\eps_1,\ldots, \eps_k)$ be the canonical basis of $\Z^k$.
Denote by~:
$$
\mapyx~: N'=N\times \Z^k \lra N
$$ 
the natural projection.  Define a fan $\De$ in $N'_\R:= N'\otimes \R$ in the following way~:

\begin{itemize}\label{def:construction_de_Delta}
\item The rays of $\Delta$ are indexed by $\Sigma(1)\cup \{1,\ldots,k\}$~: 
$$
\begin{cases}
\text{For }\theta\in\Sigma(1), &\text{ put }  v_{\theta}:=(w_\theta,0)+\sum_{i=1}^{k}\ell_{\theta}^{i}(0,\eps_{i}),\\
\text{For }i\in\{1,\ldots,k\}, &\text{ put } v_i:=(0,\eps_{i}).\\
\end{cases}
$$
Then,
$$\De(1):=\{\rho_\theta:=\R^+v_\theta, \theta\in\Sigma(1)\}\cup \{\rho_i:=\R^+v_i,\ i \in\{1,\ldots,k\}\}.$$

\item a strongly convex polyhedral cone $ \sigma$ is in $\Delta$ if and only if $
  \mapyx(\sigma)\in \Sigma$.
\end{itemize}

\begin{notn}\label{notn:rays} In the following, for any $\rho\in\De(1)$, we denote by $v_\rho\in N$ the generator of $\rho$.
It will be convenient to make the distinction between rays $\rho_\theta$ coming from the base variety $X$, and rays $\rho_i$ coming from the split vector bundle $\vb$. We put~:
$$
\raysbase =\{\rho_\theta,\theta\in\Sigma(1)\},\ \raysfiber = \{\rho_1,\ldots,\rho_k\} \text{\quad so that }\De(1)=\raysbase\sqcup \raysfiber.
$$

\end{notn}

Let $Y$ be the toric variety associated to the fan $\Delta$. As $ X $ is smooth, $
Y$ is also smooth.  The scheme morphism induced by the projection $\mapyx~: N'\lra N$ is denoted by the same letter
$
\mapyx: Y  \lra X
$. 

\begin{prop}[\cite{2011-Cox-Little-Schenck}, Proposition 7.3.1 and Exercise 7.3.3]  The toric variety $ Y$ is the total space of the vector
  bundle $\vb^{\dual}$, dual of $\vb=\oplus_{i=1}^k\lb_i$. The natural projection is the toric morphism
  $\mapyx~: Y\to X$. \qed
\end{prop}

We will make use of the following easy result about cohomology classes:

\begin{prop} \label{prop:isom_cohomologie_espace_total-base}
The projection $\mapyx~: Y\lra X$ induces an isomorphism:
$$\mapyx^*~: H^{*}(X) \stackrel{\sim}{\lra} H^{*}(Y).$$
For $i\in\{1,\ldots, k\}$,  let $D_{\rho_i}$ be the toric divisor of $Y$ corresponding to the ray $\rho_i$, then
$$ \mapyx^*[\tordiv_i]= [-D_{\rho_i}]\text{ in } H^2(Y).$$
\end{prop}

To any toric Weil divisor $D=\sum a_\theta D_\theta$ of $X$, there is an associated piecewise linear function $\psi_D$, defined on the support 
$|\Sigma|=N_\R$ of $\Sigma$
and linear on each cone, such that $\psi_D(w_\theta)=-a_\theta$.
Since the line bundles $\lb_i$ are globally generated, the functions $\psi_{\tordiv_i}$ are concave. This gives~:
\begin{lem}\label{lem:Fan_is_convex}
The support 
$|\De|=\cup_{\sigma\in \De} \sigma$ of the fan $\De$ in  $N'_\R$ is convex.
\end{lem}
%\Etienne{Le referee veut une preuve, dis-moi si tu la trouves ok.}

% \Thierry{Je la trouvais OK, mais j'ai changé un peu la formulation dans un sens que je trouvais un poil plus clair (cf. description de $|\De|$)
% et j'ai prolongé à tous les $k$, ce qui ne coûte presque rien.
% Tu peux modifier (ou alléger) comme tu veux !}

\begin{proof}
  First assume for simplicity that $k=1$ \ie $N'=N\times \mathbb{Z}$ and
  $L=\sum_{{\theta\in\Sigma(1)}}\ell_{\theta}D_{\theta}$. Let $\psi_{L}$ the concave piecewise linear
  function such that $\psi_{L}(w_{\theta})=-\ell_{\theta}$. Notice that
  $v_{\theta}=(w_{\theta},\ell_{\theta})=(w_{\theta},-\psi_{L}(w_{\theta}))$.

Let $\sigma$ be a cone of $\Delta$ of the form 
 $\sigma=\sum_{\theta\in\tau(1)} \rho_\theta+\rho_1$, where $\tau$ is the cone of $\Sigma$ obtained by projection of $\sigma$~: $\tau=\phi(\sigma)$.
A points $p$ of $\sigma$ can be written in $N_\R\times \R$ as~:
\begin{align*}
p	&=\sum_{\theta\in\tau(1)} t_\theta\left(w_\theta,-\psi_L(w_\theta)\right)+t_1(0_{N_\R},1),\qquad  t_\theta,t_1\in\R^+\\
	&=\left(\sum_{\theta\in\tau(1)} t_\theta w_\theta, -\psi_L\left( \sum_{\theta\in\tau(1)}w_\theta \right)+t_1\right) \mbox{\quad (by linearity of $\psi_L$ on $\tau$)}\\
\end{align*}
so that
\begin{displaymath}
  \sigma=\{(p_N,p_1)\in N_\R\times\mathbb{R}\mid p_N\in \phi(\sigma),\quad p_1\geq-\psi_{L}(p_N)\}.
\end{displaymath}

By definition, the support of $\Delta$ is the union of such cones $\sigma$. Since $|\Sigma|=N_\R$, one get~:
$$
|\Delta| = \{(p_N,p_1)\in\N_\R\times \R \mid p_1\geq-\psi_{L}(p_N)\}.
$$
Now, consider $p=(p_N,p_1)\in N, q=(q_N,q_1)$  two points in $|\Delta|$ and $t\in[0,1]$. Since
$\psi_L$ is concave, we have~:
$tp_1+(1-t)q_1\geq -\psi_L(tp_N+(1-t)q_N)$, and $(tp+(1-t)q)\in|\De|$ as required.

In case $k\geq 2$, we get
$$
|\Delta| = \{(p_N,p_1,\ldots, p_k)\in\N_\R\times \R^k \mid p_1\geq-\psi_{L_1}(p_N),\ldots, p_k\geq-\psi_{L_k}(p_N)\}.
$$
and $|\De|$ is also convex.
\end{proof}

\begin{exmp}Consider the fan of $ \pp^{1}$ given by
  $(N=\zz,w_{1}=1,w_{2}=-1)$, $\lb=\go(2)$ and $\tordiv=2D_{\theta_1}$.  The
  fan $ \Delta$ is given by the rays $ v_{\theta_1}=(1,2), v_{\theta_2}=(-1,0)$
  an $ v_{\tordiv}=(0,1)$ (\cf Figure \ref{fig:P1}). 
 \begin{figure}[ht]
 \centering
 \begin{tikzpicture}[scale=1]
   \fill[color=gray!25] (0,0) -- (0,3) -- (-3,3) -- (-3,0) -- cycle;
  \fill[color=gray!50] (0,0) -- (1.5,3) -- (0,3) -- cycle;
  \foreach \k in {-3,-2,-1,0,1,2,3}
  {\draw [dotted, very thin](\k,-1) -- (\k,3);}
    \foreach \k in {-1,0,1,2,3}
  {\draw [dotted, very thin](-3,\k) -- (3,\k);}
\draw [very thick,>=latex,->,black] (0,0) -- (1,2) node [right,near end] {$v_{\rho_{\theta_1}}$};
\draw [thin,black] (0,0) -- (1.5,3)  node [right ,near end] {$\rho_{\theta_1}$};
\draw [very thick,>=latex,->,black] (0,0) -- (-1,0) node [above,near end] {$v_{\rho_{\theta_2}}$};
\draw [thin, black] (0,0) -- (-3,0)  node [above ,near end] {$\rho_{\theta_2}$};
\draw [very thick,>=latex,->,black] (0,0) -- (0,1)  node [above left ,near end] {$v_{\rho_{\tordiv}}$};
\draw [thin, black] (0,0) -- (0,3)  node [above left ,near end] {$\rho_{\tordiv}$};
\draw (0,0) node  {$\bullet$};
\draw  (-6,1) node [above,right] {Fan $\De$ in $N'_\R$,};
\draw  (-6,0.4) node [above,right] {$N'=N\times \Z$.};
\draw [>=latex,->] (0,-1.5) -- (0,-2.2) node [midway, right] {$\mapyx$};
\draw [very thick,>=latex,->,black] (0,-3) -- (1,-3) node [above,near end] {$w_{\theta_1}$};
\draw [thin,black] (0,-3) -- (3,-3)  node [above ,near end] {$\theta_1$};
\draw [very thick,>=latex,->,black] (0,-3) -- (-1,-3) node [above,near end] {$w_{\theta_2}$};
\draw [thin, black] (0,-3) -- (-3,-3)  node [above ,near end] {$\theta_2$};
\draw (0,-3) node  {$\bullet$};
\draw  (-6,-2.4) node [above,right] {Fan $\Sigma$ in $N_\R$, };
\draw  (-6,-2.8) node [above,right] {$N=\Z$.};
\draw (4,0.5) node {$\leftrightsquigarrow$};
\draw (4,-3) node {$\leftrightsquigarrow$};
\draw (6.5,1) node  {$Y$, total space} ;
\draw (6.5,0.5) node  {of $\go_{\pp^1}(2)^{\dual }$} ;
\draw [>=latex,->] (6.5,-1.5) -- (6.5,-2.2) node [midway, right] {$\mapyx$};
\draw (6.5,-3) node {$X=\pp^1$} ;
\end{tikzpicture}
\caption{Fans $\Sigma$ and $ \Delta$ associated to $X=\pp^1$, $\tordiv=2D_{\theta_1}$}
\label{fig:P1}
  \end{figure}
\end{exmp}

%%%%%%%%%%%%%%%%%%%%%%%%%%%%%%%%%%%%%%%%%%%%%%%%%%%%%%%%%%%%%%%%
%%%%%%%%%%%%%%%%%%%%%%%%%%%%%%%%%%%%%%%%%%%%%%%%%%%%%%%%%%%%%%%%
\subsection{Definition and properties of Batyrev algebras associated to $(X,\vb)$.}
\label{subsection:Batyrev_algebra}
%%%%%%%%%%%%%%%%%%%%%%%%%%%%%%%%%%%%%%%%%%%%%%%%%%%%%%%%%%%%%%%%
%%%%%%%%%%%%%%%%%%%%%%%%%%%%%%%%%%%%%%%%%%%%%%%%%%%%%%%%%%%%%%%%

\subsubsection{Mori cone}
Let $X, Y$ and $\vb$ be as in section \ref{subsection:Fan}.
Using Proposition \ref{prop:isom_cohomologie_espace_total-base}, we will identify $H^2(X)$ and $H^2(Y)$, as well as the Mori cones of $X$ and $Y$.
%
%Denote by $\mori{Y}\subset H_2(Y,\Z)$ the  integral Mori cone of $Y$~:
%$$
%\mori{Y}=\left\{\sum_{\stackrel{\text{\tiny $C$ irreducible curve,}}{\text{\tiny finite sum}}}n_C [C],\quad  n_C\in \N, [C]\text{ numeric class of $C$}\right\}.
%$$
%The Mori cone in $H_2(Y,\Z)$ is dual to the nef cone in $H^2(Y,\Z)$.

\begin{notn}\label{notn:d_rho}
For any class $d$ of $H_2(Y,\Z)$ and ray $\rho$ of $\De(1)$ corresponding to the weil divisor $D_\rho$, we put
$$
d_\rho:=D_\rho.d=\int_d D_\rho \in\Z.
$$

There is an exact sequence~:
\begin{equation}\label{eq:toric_exact_sequence}
0\lra H_2(Y,\Z) \lra \Z^{\De(1)} \lra   N'  \lra 0,
\end{equation}
Where $N'=N\oplus \Z^k$ is the lattice defined in section \ref{subsection:Fan} and 
where the image of $d\in H_2(Y,\Z)$ is
$(d_\rho)_{\rho\in\De(1)}\in \Z^{\De(1)}$. 
We identify $H_2(Y,\Z)$ and its image in $\Z^{\De(1)}$.

For any real number $a$, we put $a^+= \max(a,0), a^-= \max(-a,0)$.
Also put, for any $d\in H_2(Y,\Z)$, $d^+=(d_\rho^+)_{\rho\in\De(1)}$ and $d^-=(d_\rho^-)_{\rho\in\De(1)}$. 
With the identification above, we have~:
$$
d=d^+-d^-.
$$
If $a$ is an element of $H_2(Y,\Z)\subset\Z^{\De(1)}$, we say that $a$ is \emph{supported by a cone} if the set 
$\{\rho\in\De(1)\mid a_\rho\neq 0\}$ is contained in a cone of $\De$.
\end{notn}
We will use the following facts~: 
\begin{lem}\label{lem:facts_Mori-cone} Let  $d$ be in $H_2(Y,\Z)$.
\begin{enumerate}
\item If $d^+$ is supported by a cone, then $-d\in \mori{Y}$.
\item If $d\in\mori Y\setminus\{0\}$, then $d^+$ is not supported by a cone.
\end{enumerate}
\end{lem}
\begin{proof}
\begin{enumerate}
\item  We have to show that, for any nef toric divisor $T$, $T.(-d)\geq 0$. Let $T$ be such a divisor and let $\psi$ be the piecewise linear concave function associated to $T$~:
  \begin{align*}
    T.d &= \sum _{\rho} -\psi(v_\rho) d^+_\rho-\sum _{\rho} -\psi(v_\rho) d^-_\rho\\
&=  -\psi(\sum _{\rho} v_\rho d^+_\rho)+\sum _{\rho} d^-_\rho \psi(v_\rho) &\text{ (\small $d^+$ supported by $\sigma$)}\\
&\leq  -\psi(\sum _{\rho}d^+_\rho v_\rho) +\psi(\sum _{\rho} d^+_\rho v_\rho)=0
&\text{ ($\psi$ \small concave and  $\sum_{\rho} d_\rho^+v_\rho=\sum_{\rho}d_\rho^-v_\rho$).}
  \end{align*}
\item If $d^+$ is supported by a cone, then $-d\in \mori{Y}$ and $d\in-\mori Y\cap \mori Y=0$.
\end{enumerate}
\end{proof}

\subsubsection{Twisted Batyrev algebra of $(X,\vb)$}

Let $\La$ be the semi-group algebra of $\mori{X}$, as defined in Notation \ref{notn:La,Pi,specmori,spectoric}.
Since the Mori cones of $X$ and $Y$ are identified, we have~:
\begin{align}\label{notn:anneau_Lambda}
\La &= \C[\mori{Y}]=\C[Q^d, d\in \mori{Y}].
\end{align}
Fix a set of indeterminates $(x_\rho)_{\rho\in\De(1)}$. We put~: 
$$
\La[x_\rho]:=\La[x_\rho, \rho\in\De(1)].
$$
For any $d\in H_2(Y,\Z)$ denote by $\qsrpol _d$ the polynomial~:
\begin{equation*}\label{def:QSR_polynomial}
\qsrpol _d:=x^{d^+}-Q^d x^{d^-}=\prod_{\rho\in \De(1)} x_\rho^{d_\rho^+} - Q^d \prod_{\rho\in\De(1) } x_\rho^{d_\rho^-}.
\end{equation*}
Let $M'$ be the dual lattice of $N'=N\oplus\Z^k$. For any $u\in M'$ denote by $Z_u$ the linear polynomial~:
\begin{equation*}\label{def:Lin_polynomial}
Z_u:=\sum_{\rho\in \De(1)} \langle u,v_\rho\rangle x_\rho.
\end{equation*}

\begin{defn}\label{def:batyrev_algebra} 
Consider the ring $\La[x_\rho]$ defined above.
The \textit{quantum Stanley-Reisner} ideal  of $\La[x_\rho]$ is the ideal $\QSR $ generated by the polynomials $\qsrpol _d$~:
\begin{align}
\label{def:QSR_ideal}
\QSR:=\left\langle \qsrpol _d=x^{d^+}-Q^d x^{d^-},\ d\in\mori{Y} \right\rangle
\end{align}
The \textit{linear ideal} of $\La[x_\rho]$, is the ideal $\Lin $ generated by the polynomials $Z_u$~:
\begin{align}
\label{def:Lin_ideal}
\Lin:=\left\langle Z_u=\sum_{\rho\in \De(1)} \langle u,v_\rho\rangle x_\rho, \ u\in M'  \right\rangle
\end{align}
The \emph{twisted Batytrev algebra of $(X,\vb)$} is the $\La$-algebra~:
$$
\bat:=\La[x_\rho]/(\QSR+\Lin).
$$
\end{defn}

\begin{rem}
\begin{enumerate}
\item Up to isomorphism, $\bat$ is well defined since it does not depend on the specific choice of the fan $\De$ (\ie choices of the fan $\Sigma$ and
toric divisors $\tordiv_i$). 
\item For any fan defining a smooth quasi-projective variety $Y$, we can define, as in Definition \ref{def:batyrev_algebra}, the (untwisted) \emph{Batyrev algebra of $Y$}.
However, for Proposition \ref{prop:quotient_by_QSR} and first point of Theorem \ref{thm:Batyrev_locally_free} to be true,
the support of the fan must be convex (in our case, this is equivalent to each $\tordiv_i$ being nef) of maximal dimension, and the anticanonical divisor $-K_Y$ must be nef.
\end{enumerate}
\end{rem}

The quantum Stanley-Reisner ideal $\QSR$ defined above is a deformation, parametrized by $\Spec(\La)$, of the following ideal~:
\begin{align}\label{def:SR_ideal} 
\SR & =\<x^{d^+}, d\in\mori{Y}\>~;
\end{align}
$\SR$ is the Stanley-Reisner ideal associated to the simplicial complex defined by $\De$ 
(see \cite{Bruns-Herzog-CM}). We have~:

\begin{prop} \label{prop:cohomology_of_toric_varieties}  There is a
natural isomorphism 
\begin{align*}
\C[x_\rho]/(\SR+\Lin)&\isom    H^{2*}(Y,\C)= H^{2*}(X,\C)\\
{x}_\rho &\longmapsto  [D_\rho]
\end{align*}
where $[D_\rho]\in  H^2(Y)$ is the class of the toric divisor $D_\rho$.
\end{prop}
\begin{proof} Since $\De$ is convex (Lemma \ref{lem:Fan_is_convex}) and $Y$ is quasi-projective, the proof of \cite{Fulton-toric} in the complete case can be adapted to our case, which shows that there is a well  defined isomorphism~
$\Z[x_\rho]/ (\SR+\Lin) \stackrel{\sim}{\longrightarrow} H^{2*}(Y,\Z)$ sending $x_\rho$ to $[D_\rho]$.
\end{proof}

%%%%%%%%%%%%%%%%%%%%%%%%%%%%%%%%%%%%%%%%%%%%%%%%%%%%%%%%%%%%%%%%
%%%%%%%%%%%%%%%%%%%%%%%%%%%%%%%%%%%%%%%%%%%%%%%%%%%%%%%%%%%%%%%%
\subsubsection{Residual Batyrev algebra of $(X,\vb)$}
\label{subsec:residual_Batyrev_ring}
%%%%%%%%%%%%%%%%%%%%%%%%%%%%%%%%%%%%%%%%%%%%%%%%%%%%%%%%%%%%%%%%
%%%%%%%%%%%%%%%%%%%%%%%%%%%%%%%%%%%%%%%%%%%%%%%%%%%%%%%%%%%%%%%%

From Proposition \ref{prop:isom_cohomologie_espace_total-base} there exists an isomorphism $H^2(X)\simeq H^2(Y)$~; via this isomorphism, we
have, for any toric  divisor $\tordiv_i$ and its corresponding ray $\rho_{\tordiv_i}$, $[\tordiv_{i}]=[-D_{\rho_{\tordiv_i}}]$.  

\begin{notn}\label{def:ctop_et_xtop}
Put~:
\begin{align*}
c_\top&:=  \prod_{i=1}^k [\tordiv_i]=\prod_{\rho\in\raysfiber} [-D_\rho] \in H^{2k}(X)\\
x_\top &:= \prod_{\rho\in\raysfiber} (-x_\rho)\in \La[x_\rho].
\end{align*}
Then $c_\top$ is the top Chern class of the fiber bundle $\vb=\oplus_{i=1}^k \lb_i$, and  $x_\top$ is sent to $c_\top$ via the morphism defined in Proposition \ref{prop:cohomology_of_toric_varieties}.
\end{notn}

\begin{defn}\label{def:residual_batyrev}Consider the algebra $\La[x_\rho]$ and the ideals $\QSR$ and $\Lin$ defined in \ref{def:batyrev_algebra}. Put $\GKZcom=(\QSR+\Lin)$.

The \emph{quotient ideal} of $\GKZcom$ by $x_\top$ is~:
$$
(\GKZcom:x_\top):=\{P\in \La[x_\rho],\quad  x_\top P\in \GKZcom\}.
$$
The \emph{residual Batyrev algebra} of $(X,\vb)$ is the $\La$-algebra~:
$$
\bat^{\res}:=\La[x_\rho]/(\GKZcom:x_\top),
$$
\end{defn}

\subsubsection{Main Theorem for Batyrev algebras}

The main properties of twisted and residual Batyrev algebras of $(X,\vb)$ are summed up in the following result~:

\begin{thm}\label{thm:Batyrev_locally_free}
Let $X$ be a toric smooth projective variety endowed with a split vector bundle $\vb=\oplus_{i=1}^k\lb_i$. 
Assume that each line bundle $\lb_i$ is nef, as well as $\omega_X\otimes {\lb_1}^{\dual}\otimes \cdots \otimes {\lb_k}^{\dual}$. 

Put $\La=\C[Q^d, d\in\mori{X}]$, $\specmori:=\Spec \La $ and denote by $\0$ the maximal ideal~
$\langle Q^d, d\neq 0\rangle$. Let $c_\top$ be the top chern class of $\vb$ and let
 $\mc$ be the morphism of multiplication  by $\ctop$ in $H^ {2*}(X)$.

There exists a Zariski neighbourhood $\freemori$ of $\0\in\specmori$ such that~:
\begin{enumerate}
\item Over $\freemori$, the twisted Batyrev algebra $\bat$ of $(X,\vb)$ is a locally free $\La$-module of rank \\ $\dim H^{2*}(X)$.
\item Over $\freemori$, if the line bundles $\lb_i$ are ample, then
 the residual Batyrev algebra $B^{\res} $ of $(X,\vb)$ is a locally free $\La$-module of rank $(\dim H^{2*}(X)-\dim\ker \mc)$.
\end{enumerate}
\end{thm}

\begin{rem}
A convenient neighbourhood $\freemori$ will be defined in Lemma \ref{lem:freeness_neighbourhood} and could 
 be explicitly computed by elimination algorithm (see \ref{rem:algo}). If $Y$ is Fano, $\freemori$ is the whole scheme $\specmori$. 
\end{rem}

The proof of Theorem \ref{thm:Batyrev_locally_free} will be given in section \ref{subsection:Proof_of_thm_Batyrev_locally_free}. 
We will actually rephrase its first part and show that the scheme morphism $\Spec B\ra \specmori$ is finite, flat, of degree $\dim H^{2*}(X)$ over $\freemori$.

Let us first study the quotient by the ideal $\QSR$ defined in \ref{def:QSR_ideal}.

\subsection{Quotient by the Quantum Stanley Reisner ideal}

In this section, we show~:

\begin{prop}\label{prop:quotient_by_QSR} 
Put ${\mathbf Q }:= \Spec(\La[x_\rho]/\QSR)$. 
Under assumptions of Theorem \ref{thm:Batyrev_locally_free},
the morphism $\mathbf Q  \ra \specmori$ is  flat of relative dimension 
$\dim X+k=\dim Y$. The schemes $\mathbf Q $ and $\specmori$ are Cohen-Macaulay.
\end{prop}

We will prove this proposition by 
performing a Gröbner degeneration of the Quantum Stanley-Reisner ideal. For that, we first need to consider a graded version of this ideal and
define a weight function on the monomials of the graded algebra. We then compute the initial ideal corresponding to this weight function in term of primitive classes introduced by Batyrev in \cite{batyrev-quantum-1993}.

\subsubsection{Graded $\QSR$ ideal}

Consider a new variable $h$ and define the graded $\La$-algebra $\La[x_\rho,h]$ with the grading given by  $\deg(h)=1$ and $ \deg(x_\rho)=1$.

Let $P$ be a polynomial in $\La[x_\rho]$. The homogenisation of $P$ in $\La[x_\rho,h]$ is~:
$$
P^h:=h^{\deg P}P\left(\frac{x_\rho}{h}\right)\in \La[x_\rho,h].
$$

Recall that the toric divisor $K_Y=-\sum_{\rho\in\De(1)}D_\rho$ is a canonical divisor of $Y$.
For any
$d\in H_2(Y,\Z)$, we have $\deg(x^{d^+})-\deg(x^{d^-})=\sum_\rho
D_\rho .d=-K_Y.d$. 
It follows that, for any $d$ in $H_2(Y,\Z)$,
$$
R_d^h=x^{d^+}h^{k^+}-Q^dh^{k^-}x^{d^-},\quad \text{where }k=K_Y.d.
$$		
\begin{defn}\label{def:graded_QSR-ideal} 
The \textit{graded quantum Stanley-Reisner} ideal  of $\La[x_\rho,h]$ is the homogeneous ideal $\QSR^h $ generated by the polynomials $\qsrpol^h _d$.
\end{defn}
\begin{rem}
\begin{enumerate}
\item If $-K_Y$ is nef, we get~:
$$
\QSR^h:=\left\langle \qsrpol^h _d=x^{d^+}-Q^dh^{-K_Y.d}x^{d^-},\ d\in\mori{Y} \right\rangle
$$
\item  The graded ideal $\QSR^h$ could be different from the ideal generated by the whole set of homogeneous polynomials $\{P^h, P\in\QSR\}$.
Under our assumptions, we conjecture that they are actually equal. 
\end{enumerate}
\end{rem}

\subsubsection{Weight function and monomial order}
Fix, once and for all, a strictly concave piecewise-linear function $\varphi$ of $|\De|$, rational on $N'$.
Since $\De$ is quasi-projective, such a function exists, corresponding to an ample $\Q$-divisor 
$A_\varphi=\sum_{\rho\in\De(1)} -\varphi(v_\rho)D_\rho$.

Define a weight function $\omega$ on the monomials of $\La[x_\rho,h]$ by setting, for any monomial $x^ah^k:=\prod_{\rho\in\De(1)}x_\rho^{a_\rho}h^k$~:
$$
\omega(x^ah^k)=\sum_{\rho\in\De(1)} -a_\rho \varphi(v_\rho).
$$
In particular, $\omega(h^k)=0$ for any integer $k$.
For convenience, we extend this function to any polynomial $P$ by setting~: $$\omega(P)=\max\{\omega(x^ah^k), x^ah^k\text{\ monomial of $P$}\}.$$
The \emph{initial form} of a polynomial $P=\sum_{i}\al_ix^{a_i}h^{k_i}$ is
$$
\mathrm{in}_\omega(P)=\sum_{i:  \omega(x^{a_i}h^{k_i})=\omega(P)}\al_ix^{a_i}h^{k_i}.
$$
This is not a term in general. The initial ideal $\mathrm{in}_\omega(I)$ of an ideal $I$  is 
the ideal generated by initial forms of elements of $I$.

Also define a new monomial order $\preceq $ on the variable $x_\rho,h$ by setting~:
$$
x^ah^k\preceq  x^{a'}h^{k'}\Longleftrightarrow 
\begin{cases}
\omega(x^ah^k)<\omega(x^{a'}h^{k'})\\
\text{or}\\
\omega(x^ah^k)=\omega(x^{a'}h^{k'}) \text{\ and }x^ah^k\trianglelefteq x^{a'}h^{k'}
\end{cases}
$$
Where $\trianglelefteq$ is any fixed monomial order  on the variables $\{x_\rho, h\}$.
The leading monomial of a polynomial $P$ for the order $\preceq$ will be denoted by $\Lm(P)$.

\begin{lem}\label{lem:initial_term_of_R_d} For any 
$d$ in the Mori cone of $Y$, $\Lm(\qsrpol_d^h)=\mathrm{in}_\omega(\qsrpol_d^h)=x^{d^+}.$
\end{lem}
\begin{proof} We compute~:
\begin{align*} 
\omega(x^{d^+})-\omega(h^{-K_Y.d}x^{d^-})&
= \sum_{\rho\in\De(1)} -d^+_\rho\varphi(v_\rho)-\sum_{\rho\in\De(1)} -d^-_\rho\varphi(v_\rho)=\sum_{\rho\in\De(1)} -d_\rho\varphi(v_\rho)= A_\varphi.d
> 0. 
\end{align*}
\end{proof}

\subsubsection{Primitive collections and classes}
\label{subsection:primitive,collections}

Primitive classes are specific elements in $H_2(Y,\Z)$ that generate the Mori cone of $Y$. They were introduced in (\cite{batyrev-quantum-1993} and \cite{cox-primitive-2008}).

\begin{defn} \label{defi:primitive,collection}
A subset  $\{\rho_1,\ldots,\rho_l\}$ of $\De(1)$ is called a primitive collection for $\De$ if
 $\{\rho_1,\ldots,\rho_l\}$ is not contained in a single cone of $\De$ but every proper subset is.
\end{defn}
Let $C=\{\rho_1,\ldots,\rho_l\}$ be a primitive collection, and
$v_1,\ldots,v_l$ be the generating vectors of $\rho_1\cap
N',\ldots,\rho_l\cap N'$.  Let $\sigma$ be the minimal cone of $\De$
containing $v=\sum_{i=1}^ lv_i$. Denote by $\rho'_1,\ldots,\rho'_s$ the
rays of $\sigma$ and $v'_1,\ldots,v'_s$ the primitive vectors of the
$\rho'_i$.  Since $\sigma$ is the minimal cone of $\De$ containing $v$,
the vector $v$ is in the relative interior of $\sigma$ and there exists $a_1,\ldots, a_s,$ real positive numbers,
such that~: $v=a_1v'_1+\cdots+a_s v'_s$.
Moreover, since $v$ is in ${N'}$ and the $v'_j$ are part of a basis  of
${N'}$ ($Y$ is non singular), the $a_{j}$'s are uniquely defined in $\N_{>0}$.

\begin{rem} With the above notations~: 
$
\{v_1,\ldots,v_l\}\cap\{v'_1,\ldots, v'_s\}=\varnothing.
$
(\cite {cox-primitive-2008}, proposition 1.9).
\end{rem}

Let $C=\{\rho_1,\ldots,\rho_l\}$ be a primitive collection and $v=\sum_{i=1}^l v_i=a_1v'_1+\cdots+a_s v'_s$ be as above. 
Since
$
\sum_{i=1}^l v_i-\sum_{j=1}^sa_jv'_j=0,
$
the exact sequence (\ref{eq:toric_exact_sequence}) shows that there exists a well defined element $d^C\in H_2(Y,\Z)$ such that~:
$$
d^C_{\rho}=
\left\{
\begin{array}{cl}
1 & \text{ if $\rho\in C$},\\
-a_j & \text {if $\rho=\R^+v'_j$, $j\in \{1,\ldots,s$\}},\\
0 & \text{otherwise}.
\end{array}
\right.
$$ 
\begin{notn}\label{notn:primitive_classes}
A \emph{primitive class } is a class $d^C\in H_2(Y,\Z)$ corresponding to  a primitive collection as above. 
We denote by~:
$$
\gp:=\{d^C \in H_2(Y,\Z), C \text{\ primitive collection}\}
$$
the set of primitive classes.
\end{notn}

\begin{prop}[\cite{cox-primitive-2008}, Propositions 1.9. and 1.10]
  \label{prop:mori_primitive_classes} Each primitive class is contained in the Mori cone $\mori{Y}$. The Mori cone is generated by primitive classes.
\end{prop}

\subsubsection{Initial ideal of the graded $\QSR$ ideal.}
\begin{lem}\label{prop:groebner_basis_over_field} Assume that the 
anticanonical divisor of $Y$ is nef. Let $F$ be the fraction field of $\La$. The set $\{\qsrpol^h_c, c\in\gp\}$ is a Gröbner basis, for the order $\preceq$, 
of the ideal generated by $\{\qsrpol^h_d, d\in \mori{Y}\}$ in $F[x_\rho,h]$.
\end{lem}

\begin{proof} 
First prove that, for any $d\in H_2(Y,\Z)$, there exists
a set of homogeneous polynomials $\{B_c\in F[x_\rho,h], c\in\gp\}$ such that~:
\begin{align}
\label{eq:standard_expression}
R_d^h &=\sum_{c\in\gp} B_cR^h_c,\quad\ \text{and }\quad \forall c\in \gp,\  \Lm(B_cR^h_c)\preceq  \Lm(R^h_d).
\end{align}
Let $E$ be the set of polynomials $\qsrpol^h_d$ which can not be expressed as in (\ref{eq:standard_expression}). Assume that $E$ is not empty and consider $R^h_d\in E$, whose leading monomial is minimal. 
Write $\qsrpol_d^h=x^{d^+}h^{k^+}-Q^d x^{d^-}h^{k^-}$ where $k=K_Y.d$.
Two cases may occur~:

$a)$ $\Lm(\qsrpol ^h_d)=x^{d^+}h^{k^+}$. If $d^+$ is supported by a cone, then $-d\in\mori Y$ by Lemma
\ref{lem:facts_Mori-cone}. By Lemma (\ref{lem:initial_term_of_R_d}), $\Lm(\qsrpol_{-d})=x^{(-d)^+}=x^{d^-}$ ; and $x^{d^-}\prec x^{d^+}$ which does not satisfy the assumption.

Then $d^+$ is not supported by a cone and there exists a primitive collection $C$ contained in the support of $d^+$. Denote by $c$ the class of $C$ and put $a=d^+-c^+\in \N^{\De(1)}$. Notice that $\min(d^-,a+c^-)+(d-c)^+=d^+-c^++c^-=a+c^-$ and $\min(d^-,a+c^-)+(d-c)^-=d^-$, which gives~:
\begin{align*}
\qsrpol^h _d-x^ah^{k^+}\qsrpol^h _c&=Q^c h^{k'} x^{\min(d^-,a+c^-)}\qsrpol ^h_{d-c}
\end{align*}
where $k'$ is the integer ensuring homogeneity. We also have~:
$$
\Lm(x^{\min(d^-,a+c^-)}\qsrpol ^h_{d-c})= \max(x^{d^+-c^++c^-},x^{d-})
$$
But  $x^{c^-}\prec x^{c^+}$ by Lemma (\ref{lem:initial_term_of_R_d}) and $x^{d^-}\prec x^{d^+}$ by assumption. It follows that~:
$$
\Lm(x^{\min(d^-,a+c^-)}\qsrpol^h_{d-c})\prec \Lm (R_d^h) 
$$
and $\Lm(\qsrpol ^h_{d-c})\prec \Lm(R_d^h)$.
By the minimality assumption, $\qsrpol_{d-c}^h$ admit a standard expression with zero remainder (\ref{eq:standard_expression}). This gives in return such an expression for $\qsrpol^h_d$ which contradicts $\qsrpol_d\in E$.

$b)$ $\Lm(\qsrpol^h_d)=x^{d^-}h^{k^-}$. 
Since $\qsrpol^h_{-d}=-Q^{-d}\qsrpol^h_d$ and 
$\Lm(\qsrpol^h_{-d})=\Lm(\qsrpol^h_{d})$ one may replace $d$ by $-d$ and apply the first case.

\smallskip

By (\ref{eq:standard_expression}), $\{\qsrpol^h_c, c\in\gp\}$ is a set of generators of the ideal and we can apply the Buchberger's criterion~: Let $c_1,c_2$ be two primitive classes. We have~:
\begin{align*}
S(\qsrpol^h_{c_1},\qsrpol^h_{c_2}):=&\frac{ \Lm(\qsrpol^h_{c_1})}{\text{gcd}\left(\Lm(\qsrpol^h_{c_1}),\Lm(\qsrpol^h_{c_2})\right)}\qsrpol^h_{c_2}-
\frac{ \Lm(\qsrpol^h_{c_2})}{\text{gcd}\left(\Lm(\qsrpol^h_{c_1}),\Lm(\qsrpol^h_{c_2})\right)}\qsrpol^h_{c_1}\\
=&x^{c_1^+-\min(c_1^+,c_2^+)}\qsrpol_{c_2}-x^{c_2^+-\min(c_1^+,c_2^+)}\qsrpol_{c_1}
=x^{\min(c_1^-,c_2^-)}Q^{c_1}\qsrpol^h_{c_2-c_1}.
\end{align*}
But (\ref{eq:standard_expression}) gives a normal expression with a zero remainder for $\qsrpol_{c_2-c_1}$ and then for $S(\qsrpol_{c_1},\qsrpol_{c_2})$.
\end{proof}

\begin{prop} \label{prop:initial_ideal_of_graded_QSR} The initial ideal of $\QSR^h$ for the weight function $\omega$ is~:
$$
\mathrm{in}_\omega(\QSR^h)=\langle \mathrm{in}_\omega(R^h_c), c\in\gp\rangle = \langle x^{c^+}, c\in\gp\rangle =\langle \ x^a,\ a\text{ is not supported by a cone}\rangle.
$$
\end{prop}
\begin{proof} Let $a$ be in $\N^{\De(1)}$ not supported by a cone. There exists a primitive class $c\in\gp$ such that the support of $c^+$ is contained in the support of $a$. Then $a-c^+\in \N^{\De(1)}$ and the leading form of $x^{a-c^+}\qsrpol^h_c\in\QSR^h$ is $x^a$. This, and Lemma \ref{lem:initial_term_of_R_d}, proves the two equalities on the right, and the
inclusion $\langle \mathrm{in}_\omega(R^h_c), c\in\gp\rangle\subset \mathrm{in}_\omega(\QSR^h)$.
It remains to show that $\mathrm{in}_\omega(\QSR^h)\subset \langle \mathrm{in}_\omega(R^h_c), c\in\gp\rangle$.

Let $P$ be in $\QSR^h$. Using Lemma \ref{prop:groebner_basis_over_field}, we can write~:
$$
P=\sum_{c\in\gp}Q_c\qsrpol^h_c
$$
where, for any $c\in\gp$, $Q_c\in F[x_\rho,h]$ ($F=\text{Frac}(\La)$) and $\Lm(Q_c\qsrpol^h_c)\preceq \Lm(P)$~; this implies   
$\omega(Q_c\qsrpol^h_c)\leq \omega(P)$. The initial form of $P$ is~:
$$
\mathrm{in}_\omega(P)=\sum_{c\in\gp, \omega(Q_c\qsrpol^h_c)=\omega(P)}\mathrm{in}_\omega(Q_c\qsrpol^h_c)=\sum_{c\in\gp, \omega(Q_c\qsrpol^h_c)=\omega(P)}\mathrm{in}_\omega(Q_c)x^{c^+}.
$$ 
it follows that $\mathrm{in}_\omega(P)$ is in $\La[x_\rho,h]\cap (\sum_{c\in\gp} F[x_\rho,h]x^{c^+})$ which is equal to 
$\sum_{c\in\gp} \La[x_\rho,h]x^{c^+}$ since the $x^{c^+}$ are monomials.
\end{proof}

The initial form of $R_c^h$ are unitary terms. This gives~: 

\begin{cor}\label{cor:initial_ideal_of_specialized_QSR} Let $p$ be any closed point of $\Spec \La$ and  $\kappa\simeq \C$ be its residual field. 
Let $\ov{\QSR}^h$ be the images of $\QSR^h$ in $\kappa[x_\rho,h]$.
The initial ideal of $\ov{\QSR}^h$ for the weight function $\omega$ is~:
$$
\mathrm{in}_\omega(\ov{\QSR}^h)=\langle x^{c^+}, c\in\gp\rangle =\langle \ x^a,\ a\text{ is not supported by a cone}\rangle.
$$
\end{cor}

Forgetting the variable $h$, we can restrict the weight function $\omega$ to $\La[x_\rho]$~; we still denote it by $\omega$. Proposition \ref{prop:initial_ideal_of_graded_QSR} above, and specialization to $h=1$ gives~:
\begin{cor} \label{cor:initial_ideal_of_QSR} The initial ideal of $\QSR$ for the weight function $\omega$ is~:
$$
\mathrm{in}_\omega(\QSR)=\langle \mathrm{in}_\omega(R_c), c\in\gp\rangle = \langle x^{c^+}, c\in\gp\rangle =\langle \ x^a,\ a\text{ is not supported by a cone}\rangle.
$$
\end{cor}

\subsubsection{Proof of Proposition \ref{prop:quotient_by_QSR}}

We first show that $\mathbf Q\ra \specmori$ is flat, then show, by Groebner degeneration, that each fiber is Cohen-Macaulay.

\noindent{\emph {Flatness.}} Let us prove that $\La[x_\rho]/\QSR$ is a free $\La$-module~:
for any $P$ in $\La[x_\rho]$, denote by $\ov P$ its image in $\La[x_\rho]/\QSR$.
Let $A$ be the set of monomials of $\La[x_\rho]$ not contained in $\mathrm{in}_\omega(\QSR)$. By Corollary \ref{cor:initial_ideal_of_QSR}
$A=\{x^a, a\in \N^{\De(1) } \mid a\text {\ is supported by a cone}\}$. We claim that $\ov A=\{\ov {x^a}, x^a\in A\}$
is a base of $\La[x_\rho]/\QSR$.

Let $x^{a_1},\ldots, x^{a_l}$ be in $A$ and $\al_1,\ldots, \al_l$ be
in $\La$. If $\sum_i \al_i \ov{x^{a_i}}=0$, then $\sum_i \al_i
x^{a_i}\in \QSR$ and $\mathrm{in}_\omega(\sum_i \al_i x^{a_i})\in \mathrm{in}_\omega(\QSR)$. Since every $a_i$ is supported by a cone,  $\al_i=0$ for any $i$, and $\ov A$ is free over $\La$.

Suppose now that $\ov A$ does not generate $\La[x_\rho]/\QSR$ as a $\La$-module. Let $x^a$
be the smallest monomial for $\preceq$ such that $\ov {x^a}\notin \La.\ov A$.  Then $a$ is not supported by a cone. There exists a primitive class $d$, and $b\in \N^{\De(1)}$ such that $a=b+d^+$ and
$x^a=x^b\qsrpol _d +Q^d x^{b+d^-}$. We deduce that $\ov {x^a}=Q^d \ov{
  x^{b+d^-}}$. Since $x^{b+d^-}\prec x^a$, the class $\ov
{x^{b+d^-}}$ belongs to $\La. \ov A$, hence $\ov {x^a}\in \La.\ov A$~; this is a contradiction.

\smallskip

{\emph {Cohen-Macaulayness and relative dimension.}}
Consider the graded ring $\La[x_\rho,h]$ and the ideal $\QSR^h$ defined in \ref{def:graded_batyrev_algebra}. Put $n'=n+k=\dim Y$. We first prove that $\La[x_\rho,h]/\QSR^h$ 
is Cohen-Macaulay of relative dimension $n'+1$ over $\La$.

Since $\specmori$ is a toric affine variety, it is a Cohen-Macaulay scheme. By \cite{Bruns-Herzog-CM}, Theorem 2.1.7, it is sufficient to show
that over every closed point $p$ of $\specmori$, the fiber $\mathbf Q_p$ of $\mathbf Q\ra\specmori$ 
is a Cohen-Macaulay scheme of dimension $n'+1$.

Let $p$ be a closed point of $\specmori$, $\kappa\simeq \C$ its residual field, and  denote by ${\ov\QSR}^h$ the image of $\QSR^h$ in $\kappa [x_\rho,h]$.
By Proposition \ref{prop:initial_ideal_of_graded_QSR}, the initial ideal of $\ov\QSR^h$ is the Stanley-Reisner ideal $\SR=
\langle x^a, a\in\N^{\De(1)}, \text{\ not supported by a cone} \rangle $ defined in  \ref{def:SR_ideal}. 
 Since $\ov \QSR^h$ is a graded ring, we can perform Gröbner degeneration, \ie construct a flat and proper family over $\A_1=\Spec\C[t]$ whose fibre over $0$ is $\proj( \C[x_\rho,h]/\SR)$ and whose fibre over any other point is $\proj( \C[x_\rho,h]/\ov{\QSR}^h)$. From \cite{Bruns-Herzog-CM}, Theorem 5.1.4, and  Corollary 5.4.6, we know that
  $\C[x_\rho]/\SR$ is a Cohen-Macaulay ring of dimension $n'$. Then $\C[x_\rho,h]/\SR$ and   $\C[x_\rho,h]/\ov{\QSR}^h$ both are Cohen-Macaulay rings of dimension $n'+1$.

Since $\QSR^h$ is homogeneous, the polynomial $h-1$ is not a zero divisor of $\La[x_\rho,h]/\QSR^h$. Then 
$\La[x_\rho]/\QSR=\La[x_\rho,h]/(\QSR^h,h-1)$ is a Cohen-Macaulay ring. \qed

%%%%%%%%%%%%%%%%%%%%%%%%%%%%%%%%%%%%%%%%%%%%%%%%%%%%%%%%%%%%%%%%%%%%%%
%%%%%%%%%%%%%%%%%%%%%%%%%%%%%%%%%%%%%%%%%%%%%%%%%%%%%%%%%%%%%%%%%%%%%%
\subsection{Proof of Theorem \ref{thm:Batyrev_locally_free}}
\label{subsection:Proof_of_thm_Batyrev_locally_free}
%\label{subsection:Properties_of_Batyrev_algebra}
%%%%%%%%%%%%%%%%%%%%%%%%%%%%%%%%%%%%%%%%%%%%%%%%%%%%%%%%%%%%%%%%%%%%%%
%%%%%%%%%%%%%%%%%%%%%%%%%%%%%%%%%%%%%%%%%%%%%%%%%%%%%%%%%%%%%%%%%%%%%%

Put $\mathbf B=\Spec \bat$ and consider the scheme morphism $f: \mathbf B\ra \specmori$.
We first study the fiber of $f$ over $\0$. We then define a convenient neighbourhood $\freemori$ of $\0$ with help of a graded version of the Batyrev algebra and show that $\mathbf B\ra \specmori$ is finite, flat, of degree $\dim H^{2*}(X)$
over $\freemori$~; this prove the Theorem for the twisted Batyrev algebra. We finally prove the Theorem for the residual Batyrev algebra.

\subsubsection{Fibre of $\mathbf B \ra\specmori$ over $\0$.}

Recall the definition of the Stanley-Reisner ideal of $\De$ (\cf \ref{def:SR_ideal}).
By Proposition \ref{prop:cohomology_of_toric_varieties} we have~:
\begin{align}
\label{eq:fiber_of_B_over_0}
\bat \otimes (\La/\0) \isom  \C[x_\rho]/(\SR+\Lin)&\isom    H^{2*}(Y,\C)= H^{2*}(X,\C)\\
    {x}_\rho &\longmapsto  [D_\rho]\nonumber
\end{align}

\subsubsection{Definition of a neighbourhood of $\0$}

First define a graded version of the Batyrev algebra of $Y$~:

\begin{defn}\label{def:graded_batyrev_algebra} Assume that the canonical divisor $-K_Y$ is nef, and 
consider the graded $\La$-algebra $\La[x_\rho,h]$ and the graded quantum Stanley-Reisner ideal $\QSR^h $
defined in \ref{def:graded_QSR-ideal}.

The \textit{linear ideal} of $\La[x_\rho,h]$, is the homogeneous ideal $\Lin $ generated by the polynomials $Z_u=\sum_{\rho\in \De(1)} \langle u,v_\rho\rangle x_\rho$, $u\in M'$.
The \textit{graded Batyrev algebra} of $\De$ is the $\La$-algebra~:
$$
\bat^h:=\La[x_\rho,h]/(\QSR^h+\Lin).
$$
\end{defn}

 Put~:
\begin{itemize}
\item $\mathbf Q := \Spec(\La[x_\rho]/\QSR)$.
\item $\mathbf P := \proj(\La[x_\rho,h])$ and $\pi: \mathbf P \ra \specmori$, the natural projective morphism.
\item $\mathbf H\subset \mathbf P $, the relative hyperplane at infinity, defined by $h=0$.
\item $ {\mathbf B^h} := \proj(\La[x_\rho,h]/(\QSR^h+\Lin))$.
\end{itemize}
By definition, $\mathbf B = {\mathbf B^h}   \cap (\mathbf P\setminus \mathbf H) $.

\begin{lem}\label{lem:freeness_neighbourhood}
Set~:
$$\freemori:=\specmori \setminus \pi(\mathbf \bat^h \cap \mathbf H).$$ 
then $\freemori$ is an open Zariski neighbourhood of $\0$.
\end{lem}
\begin{proof}
Since $\pi$ is projective, $\freemori$ is an open subset of $\specmori$.
Let us show that $\0\in\freemori$.
The intersection $(\mathbf \bat^h \cap \mathbf H)\cap (\pi^{-1}(\0))$ is defined by the homogeneous ideal
$
\langle h\rangle+\SR+\Lin
$
in $\C[x_\rho,h]$. By Proposition \ref{prop:cohomology_of_toric_varieties}, $\C[x_\rho,h]/(\langle h\rangle+\SR+\Lin)\isom \C[x_\rho]/(\SR+\Lin)$
is isomorphic to $H^{2*}(Y,\C)$. Then the ideal $\langle h\rangle+\SR+\Lin$ defines a zero dimensional scheme supported by the origin of $\Spec \C[x_\rho,h]$~; its radical is the irrelevant ideal of the graded ring
$\C[x_\rho,h]$, and $(\mathbf \bat^h \cap \mathbf H)\cap (\pi^{-1}(\0))=\emptyset$.
\end{proof}

%%%%%%%%%%%%%%%%%%%%%%%%%%%%%%%%%%%%%%%%%%%%%%%%%%%%%%%%%%%%%%%%%%%%%%
%%%%%%%%%%%%%%%%%%%%%%%%%%%%%%%%%%%%%%%%%%%%%%%%%%%%%%%%%%%%%%%%%%%%%%
\subsubsection{Local freeness and rank of the twisted Batyrev algebra of $(X,\vb)$.}

%%%%%%%%%%%%%%%%%%%%%%%%%%%%%%%%%%%%%%%%%%%%%%%%%%%%%%%%%%%%%%%%%%%%%%
%%%%%%%%%%%%%%%%%%%%%%%%%%%%%%%%%%%%%%%%%%%%%%%%%%%%%%%%%%%%%%%%%%%%%%

Let $\mathbf B^h_{\scriptscriptstyle{\freemori}}$ be the pull-back of $\mathbf B^h $ by the open inclusion
$\freemori\hookrightarrow \specmori$~; we make use of the same notation for any other scheme defined over $\specmori$.

By Definition of $\freemori$, $\mathbf \bat^h_{\scriptscriptstyle{\freemori}}$
does not meet the relative hyperplane
$\mathbf H_{\scriptscriptstyle{\freemori}}$, hence $\mathbf \bat_{\scriptscriptstyle{\freemori}}=\mathbf \bat^h_{\scriptscriptstyle{\freemori}}$. Moreover, as a closed subscheme of the projective bundle $\mathbf P_{\scriptscriptstyle{\freemori}}$ which do not meet a relative hyperplane, $\mathbf  \bat_{\scriptscriptstyle{\freemori}}$ has relative dimension zero. Thus, $\mathbf \bat_{\scriptscriptstyle{\freemori}}\ra \freemori$ is a finite
and proper morphism.

By Proposition \ref{prop:quotient_by_QSR}. $\mathbf Q_{\scriptscriptstyle{\freemori}}\lra {\freemori}$ is a flat morphism of relative dimension $n'=\dim Y$ between Cohen-Macaulay schemes.  
One get the following diagram~:
$$
\xymatrix{
\mathbf B_{\scriptscriptstyle{\freemori}}\ar[rd]_{\stackrel{\text{\tiny rel.}}{\text{\tiny dim. $0$}}}\ar@{^{(}->}[r]&\mathbf Q_{\scriptscriptstyle{\freemori}}
\ar[d]^{{\text{\tiny rel. dim. $n'=\dim Y$}}}\\
& \freemori
}
$$

Let $(e_1,\ldots,e_{n'})$ be a basis of ${M'}=\Hom(N\oplus \Z^k,\Z)$.  
Let $p$ be a closed point of $\freemori$ and denote by $\ov Z_i$ the image of $Z_i:=Z_{e_i}$ in the quotient 
of $\La[x_\rho]$ by the maximal ideal defining $p$. In the Cohen-Macaulay fiber $\mathbf Q_p$ over $p$, the
scheme $\mathbf B_p$ has codimension $n'$ and is defined by a sequence
 of the same length $n'$ (namely $(\ov{Z}_1,\ldots,\ov{Z}_{n'})$). Then, by
\cite{Bruns-Herzog-CM}, theorem 2.1.2, $(\ov{Z}_1,\ldots,\ov{Z}_{n'})$
is a regular sequence.

Since $\mathbf Q_{\scriptscriptstyle{\freemori}}\ra {\freemori}$ is
flat, and $(\ov{Z}_1,\ldots,\ov{Z}_{n'})$ is a regular sequence over
any point of ${\freemori}$, the morphism $\mathbf B_{\scriptscriptstyle{\freemori}}\ra {\freemori}$ is flat (\cite{Matsumura-commutative} Theorem 22.5 and Corollary).
 The degree of this finite morphism can be computed as the length of the fibre $\mathbf B^h_{\0}$ over $\0$. From isomorphism \ref{eq:fiber_of_B_over_0}, it is equal to 
 $\dim H^{2*}(Y)$. \qed

%%%%%%%%%%%%%%%%%%%%%%%%%%%%%%%%%%%%%%%%%%%%%%%%%%%%%%%%%%%%%%%%
%%%%%%%%%%%%%%%%%%%%%%%%%%%%%%%%%%%%%%%%%%%%%%%%%%%%%%%%%%%%%%%%
\subsubsection{Local freeness and rank of the residual Batyrev algebra of $(X,\vb)$.}
\label{subsec:residual_Batyrev_algebra}
%%%%%%%%%%%%%%%%%%%%%%%%%%%%%%%%%%%%%%%%%%%%%%%%%%%%%%%%%%%%%%%%
%%%%%%%%%%%%%%%%%%%%%%%%%%%%%%%%%%%%%%%%%%%%%%%%%%%%%%%%%%%%%%%%

Denote by $\ov{x}_\top $ the image of $x_\top $ in $\bat=\La[x_\rho]/(\QSR +\Lin)$, and by $m_{\ov{x}_\top}~: \bat\lra\bat$
the morphism of multiplication by $\ov{x}_\top$ in $\bat$.
This multiplication induces an isomorphism~:
$$
\bat^{\res}=\La[x_\rho]/(\GKZcom:x_\top)\isom \ov{x}_\top\bat=\im(m_{\ov{x}_\top}),
$$
which gives an exact sequence~:
\begin{equation}
0\lra \bat^\res \lra \bat \lra \bat/\ov{x}_\top\bat \lra 0  \label{exact1}
\end{equation}

Let $d$ be a class of $\mori{Y}$, and $\rho_i$ be the ray of $\raysfiber$ corresponding to a line bundle $\lb_i$. Since $\lb_i$ is ample and the Chern class of $\lb_i$ is
$[-D_{\rho_i}]$, we have $d_{\rho_i}=D_{\rho_i}.d<0$. Then, we have~:

\begin{align}\label{eq:Rd,xctop}
  \qsrpol _d=x^{d^+}-Q^dx_\top x^{d^--\eps},
\end{align}

 where $\eps=(\eps_\rho)_{\rho\in\De(1)}, \eps_\rho=1$ if $\rho\in\raysfiber$, 
$\eps_\rho=0$ if $\rho\in\raysbase$. 

As a consequence, the image of $\qsrpol_d$ in $\bat/\ov{x}_\top\bat=\La[x_\rho]/(\QSR +\Lin+\langle x_\top \rangle )$, is $x^{d^+}$ and we have~:
\begin{align*}
\bat/\ov{x}_\top\bat&\isom \La[x_\rho]/(\langle x^{d^+}, d\in \mori{Y}\rangle + \langle Z_u, u\in M'\rangle  + \langle x_\top \rangle )\\
&\isom \La\otimes 
\biggl(\C[x_\rho]/(\SR+\Lin + \langle x_\top \rangle)
\biggr)
\end{align*}
Using Proposition \ref{prop:cohomology_of_toric_varieties}, we get~:
$$
\La[x_\rho]/(\QSR +\Lin+\langle x_\top \rangle )\isom
\La\otimes \bigl(H^{2*}(X,\C)/\< c_\top \>\bigr).
$$
Thus, $\bat/{\ov{x}_\top}\bat$ is a free $\La$-module of rank 
$\dim_\C H^{2*}(Y)\, /\, c_\top H^{2*}(Y)=\dim_\C \ker \mc.$

Restricting the exact sequence \eqref{exact1}  to ${\freemori}$, and using Theorem \ref{thm:Batyrev_locally_free}, we find that
$(\bat^{\res})|_{{\freemori}}$ is a locally free $\La$-module of rank ($\dim H^{2*}(Y)-\dim \ker \mc$) over ${\freemori}$.  \qed

%%%%%%%%%%%%%%%%%%%%%%%%%%%%%%%%%%%%%%%%%%%%%%%%%%%%%%%%%%%%%%%%
\section{GKZ modules for toric varieties with a split vector bundle} 
\label{sec:GKZ_modules}

GKZ systems were defined and studied by
Gelfand-Kapranov-Zelevinski{\u\i} in the end of the eighties (\cf
\cite{GKZ-holonomic-system-1987},
\cite{GKZ-Equations-hypergeometric-1988},
\cite{GKZ-hypergeo-toric-1990} and
\cite{GKZ-Generalized-Euler-integral-1990}).  Our
approach is closer to the one of \cite{Givental-HMS-1995}, \cite{Givental-1998-Mirror-complete-intersection},
\cite[\S 5.5.3 and \S 11.2]{Cox-Katz-Mirror-Symmetry} or
\cite{Iritani-2009-Integral-structure-QH}.

%%%%%%%%%%%%%%%%%%%%%%%%%%%%%%%%%%%%%%%%%%%%%%%%%%%%%%%%%%%%%%%%
%%%%%%%%%%%%%%%%%%%%%%%%%%%%%%%%%%%%%%%%%%%%%%%%%%%%%%%%%%%%%%%%
\subsection{Definition and main Theorem for GKZ-modules}
\label{subsec:GKZ}
%%%%%%%%%%%%%%%%%%%%%%%%%%%%%%%%%%%%%%%%%%%%%%%%%%%%%%%%%%%%%%%%
%%%%%%%%%%%%%%%%%%%%%%%%%%%%%%%%%%%%%%%%%%%%%%%%%%%%%%%%%%%%%%%%

For Batyrev algebras, the natural base ring is $\La=\C[\mori{X}]= \C[\mori{Y}]$ as defined in Notation
\ref{notn:La,Pi,specmori,spectoric} or \ref{notn:anneau_Lambda}. 
When dealing with differential operators, we need to work over 
a smooth subvariety of $\specmori=\Spec\La$.

% Using Notation \ref{notn:La,Pi,specmori,spectoric} we consider
%  $\Pi=\C[H^2(X,\Z)]=\C[H^2(Y,\Z)]$, $(B_1,\ldots,B_r)$ the base 

Put, as in \ref{def:batyrev_algebra}, $\La=\C[Q^d,d\in \mori{Y}]$.
Consider the ring $\C[Q^d, d\in H_2(Y,\Z)]$, which is the localization of $\La$ where the $Q^d$ are made invertible.
Let $(B_1,\ldots, B_r)$ be the fixed base of $H_2(X,\Z)$ and  $(T_1,\ldots, T_r)$ be its dual base in $H^2(X,\Z)$
(\cf \ref{notn:base_of_cohomology}). Put $q_i=Q^{B_i}$. 
\begin{notn}\label{notn:toric_locus}Set~: 
\begin{align*}
\C[q_a^\pm]&:=\C[q_{1}^{\pm 1}, \ldots ,q_{r}^{\pm 1}]=\C[Q^d, d\in H_2(Y,\Z)],\\
\spectoric &:=\Spec \C[q_a^\pm]\qquad \freetoric:=\freemori\cap \spectoric
\end{align*}
where $\freemori$ is the neighbourhood of $\0$ defined in Lemma \ref{lem:freeness_neighbourhood}~; $\freemori$ is 
 the locus over which the Batyrev algebra is ensured to be locally free, and $\freetoric$ will play the same role for differential modules. We have~:
$$
\begin{array}{cccccl}
\0 & \in & \freemori & \subset  & \specmori &\\
   &     & \cup       &         &  \cup &\\
&     &  \freetoric &  \subset & \spectoric.& 
\end{array}
\qquad \text{ and $\0\notin \spectoric$.}
$$
\end{notn}
For any $d=\sum_{a=1}^r d_aB_a\in H_2(X,\Z)$ we write~:
$$
q^d:=\prod_{a=1}^rq_a^{d_a}\in \C[q_a^\pm].
$$

Let $z$ be another variable~;  we write $\C_z$ for $\Spec \C[z]$, or $\C$ when no confusion can occur. 
Consider the non-commutative ring~:
    \begin{align}\label{def:differential_ring_D}
      \D:=\cc[q_{1}^{\pm 1}, \ldots ,q_{r}^{\pm 1},z]\langle z\delta_{q_{1}}, \ldots ,z\delta_{q_{r}},z\delta_{z}\rangle=\cc[q_a^\pm,z]\langle z\delta_q,z\delta_z\rangle,
    \end{align}
    where the non commutative relations are 
    $(z\de_{q_i}) q_i=q_i(z\de_{q_i})+zq_i$ and $(z\de_z) z=z(z\de_z)+z^2$.
\begin{notn} 1. \emph{Quantisation:}\label{notn:quantization}
To any class $\tau=\sum_{a=1}^{r}t_{a}T_{a}\in H^2(X)$ we associate the operator 
$$\widehat{\tau}:= \sum_{a=1}^{r}t_{a}z\delta_{q_a}\in \D$$
If $\lb$ is a line bundle or a divisor on $X$ we also write
$\widehat{\lb}:=\widehat{c_1(\lb)}.$ 
Finally put~:
$$
\hatctop:=\prod_{i=1}^k \widehat \lb_i \in \D.
$$

2. \emph{Pochhammer symbol with a variable $z$:} For any element $a$ of a $\Z[z]$-algebra, and any $k\in \N$ define~:
\begin{align}\label{def:Pochhammer}
[a]_0=1, \qquad [a]_k:=a(a-z)\cdots (a-(k-1)z)\text{\quad if $k>0$}.
\end{align}
This is a variant of the traditional Pochhammer symbol.
\end{notn}

\begin{defn}\label{def:idealGKZ}
For any $d\in H_2(X,\Z)$ put~:
\begin{align*}
\quad \square_d &:= 
\prod_{i=1}^k \left[\widehat{\tordiv_i}+zd_{i}^- \right]_{d_i^-} \prod_{\theta\in\Sigma(1)}\left[\widehat{D}_{\theta}\right]_{d_\theta^+}-
q^d\prod_{i=1}^k \left[\widehat{\tordiv_i}+zd_{i}^+ \right]_{d_i^+} \prod_{\theta\in\Sigma(1)}\left[\widehat{D}_{\theta}\right]_{d_\theta^-}.
\end{align*}
where $d_i=\tordiv_i.d$, $d_\theta=D_\theta.d$.
Also define the Euler field~:
\begin{align*}
\eulerfield&:=z\delta_z+\widehat{c_{1}(\tang_{X})}+\widehat{c_1( \vb^{\dual})}.
\end{align*}
  \begin{enumerate}
  \item  The \emph{GKZ-ideal $\GKZid$} of $\D$ associated to $(\Sigma,\tordiv_1,\ldots,\tordiv_k)$ is the left ideal generated by the operators
    $\square_d$ and $\eulerfield$~:
    $$\GKZid:=\langle\eulerfield, \square_d, d\in H_2(X,\Z)\rangle$$
\item The \emph{quotient ideal $(\GKZid:\hatctop)$ of $\GKZid$ with respect to $\hatctop$}, is the left ideal of $\D$ generated by~:
$\{P\in \D \mid  \hatctop P \in \G\}.  $
$$
(\GKZid:\hatctop):=\langle P\in \D \mid  \hatctop P \in \G\rangle.
$$
\end{enumerate}
\end{defn}

\begin{rem} 1. The set $\{P\in \D \mid  \hatctop P \in \G\}$ contains the ideal $\G$~; however, 
unlike the commutative case, it is not an ideal of $\D$ in general,
but only a $\C[z]$-module (as an example, in
$\C[q]\<\de_q\>$, fix $\I=\langle \de_q\rangle$~; then $q\in(\I:\de_q)$
but $\de_q q\notin\I$).

2. If $\rho\in\raysfiber$ corresponds to a divisor $\tordiv_i$, we have $[-D_\rho]=[\tordiv_i]$. This enables us to write~:
$$
\quad \square_d=\hspace{-0.1cm}\prod_{\rho\in\raysfiber}\hspace{-0.1cm}\left[-\widehat{D_\rho}+z d_\rho^+\right]_{d_\rho^+}\hspace{-0.1cm}\prod_{\rho\in\raysbase}\hspace{-0.1cm}\left[\widehat{D_\rho}\right]_{d_\rho^+}
-q^d\hspace{-0.1cm}\prod_{\rho\in\raysfiber}\hspace{-0.1cm}\left[-\widehat{D_\rho}+zd_\rho^-\right]_{d_\rho^-}\hspace{-0.1cm}\prod_{\rho\in\raysbase}\hspace{-0.1cm}\left[\widehat{D_\rho}\right]_{d_\rho^-}
$$
Note that, in this writing, the sign in
front of $\widehat{D_\rho}$ differs for rays coming from the base $X$ or
from the line bundles $\lb_i$.  We follow here the conventions of  \cite{Cox-Katz-Mirror-Symmetry}, taking account of
 their  "Erratum to 
Proposition 5.5.4".

\end{rem}
\begin{defn}\label{def:GKZ} 
	Let $\D=\cc[q_a^\pm,z]\langle z\delta_q,z\delta_z\rangle$ be the non commutative ring defined above.
	Let $\gd$ be the corresponding sheaf of $\go_{\spectoric\times\C}$-algebras.
  \begin{enumerate}
  \item The \emph{twisted GKZ module} associated to $(\Sigma,\tordiv_1,\ldots,\tordiv_k)$ is the left $\D$-module  
  $$\GKZmod:=\D/\G,$$
  the corresponding sheaf of $\gd$-modules is denoted by ${\GKZmodsheaf}$.
  \item The \emph{residual GKZ module} ${\GKZmod^\res}$ is the left $\D$-module 
   $${\GKZmod^\res}:=\D/(\GKZid:\hatctop),$$
    the corresponding sheaf of $\gd$-modules is denoted by ${\GKZmodsheaf}^\res$.
\end{enumerate}
\end{defn}

\begin{rem}Up to isomorphism, $\GKZmod$ and $\GKZmod^\res$ does not depend on the specific choices of
the fan $\Sigma$ and toric divisors $\tordiv_i$. 

%\Etienne{Referee asks why ? j'ai marqué ceci, dis-moi si c'est convaincant pour toi ? }
Indeed, the GKZ system is defined from the
following exact sequence
\begin{displaymath}
  \xymatrix{0\ar[r]&H_{2}(X,\mathbb{Z})\ar[r]&\mathbb{Z}^{\#\Delta(1)}\ar[r]^{\beta}&N\times
    \mathbb{Z}^{k}\ar[r]&0} 
\end{displaymath}
where $\beta(e_{\theta})=w_{\theta}$.
If one use an other fan, than we have an isomorphic exact sequence which gives an isomorphic GKZ
system. Notice that this exact sequence just depends on the rays of $\Delta$ and not on the higher
dimensional cone of $\Delta$.
\end{rem}

\begin{rem}\label{rem:alternative_definitions_of_GKZ}
%\Etienne{Referee ask is it  parallel statement for $M^{\res}$ ?}
We will need alternative definitions of the $\mathrm{GKZ}$ modules~:

(1) \emph{Removing $z\de_z$~:} Put $\D':=\cc[q_a^\pm,z]\langle z\delta_{q}\rangle$ and ${\GKZid}'=\langle
  \square_d , \ d\in H_2(X,\Z)\rangle\subset \D'$. The Euler operator
  $\eulerfield$ of the ideal $\GKZid$ enables us to remove
  $z\delta_{z}$ in the quotient, which gives two isomorphisms of $\cc[q_a^\pm,z]$-module~:
   \begin{align}\label{eq:1}
    \GKZmod \isom \D'/{\GKZid}' &&     \GKZmod^{\res} \isom \D'/(\GKZid':\hatctop) 
  \end{align}
%\Etienne{Les iso viennent du fait qu'on peut faire une division par $\hatctop$ et tous les
%  operateurs de $\hatctop$ commutent entre eux}

(2) \emph{Differential operators indexed by rays of $\De$~:} For any $\rho\in\De(1)$ write, in a unique way~:
\begin{align}\label{notn:coordinates_of_D_rho}
[D_\rho]&=\sum_{a=1}^r D_\rho^a T_a\in H^2(Y,\Z).
\end{align}

Consider a set of indeterminates $\la_\rho, \rho\in\De(1)$.
Put $\D'':=  \C[q_a^\pm,z]\langle z\de_{\la_\rho}, \rho\in \De(1)\rangle,$
where the non commutative relations are~: 
$ z\de_{\la_\rho}.q_a=q_a.z\de_{\la_\rho}+D_\rho^a zq_a$.
For any $d\in H_2(X,\Z)$, put~:
\begin{align}\label{eq:square,double,prime}
\sq''_d&:=\hspace{-0.1cm}\prod_{\rho\in\raysfiber}\hspace{-0.1cm}[-z\de_{\la_\rho}+zd_\rho^+]_{d_\rho^+}\hspace{-0.1cm}\prod_{\rho\in\raysbase}\hspace{-0.1cm}[z\de_{\la_\rho}]_{d_\rho^+}
 -q^d\hspace{-0.1cm}\prod_{\rho\in\raysfiber}\hspace{-0.1cm}[-z\de_{\la_\rho}+zd_\rho^-]_{d_\rho^-}\hspace{-0.1cm}\prod_{\rho\in\raysbase}\hspace{-0.1cm}[z\de_{\la_\rho}]_{d_\rho^-},\\
 \gz''_u&:=\sum_{\rho\in \Delta(1)} \langle u,v_\rho\rangle z\de_{\la_\rho}, u\in M' \label{eq:operator,Z,double,prime}.
\end{align}
 Put $\GKZid'':=\langle \sq''_d, \gz''_{u}\rangle$.
Then there is an isomorphism of $\C[q_a^\pm,z]$-modules~:
\begin{align}\label{eq:other,presentation,GKZ,mod}
f~:~\D''/\GKZid'' &\lra \D'/\GKZid'\simeq \GKZmod\\ 
z\de_ {\la_\rho}&\longmapsto  \sum_{a=1}^r D_\rho^a  z\de_{q_a}\nonumber
\end{align}
The previous isomorphism $f$ induces an isomorphism $f^{\res}$ between the residual's modules, that is 
\begin{displaymath}
  \xymatrix{\D''/\GKZid''\ar[r]^{f}_{\sim} \ar[d]& \D'/\GKZid'\simeq \GKZmod \ar[d] \\
\D''/\left(\GKZid'': \prod_{\rho\in \raysfiber} -z\de_{\la_{\rho}}\right)\ar[r]^-{f^{\res}}_-{\sim} &
 \D'/(\GKZid':\hatctop)  \simeq \GKZmod^{\res}
}
\end{displaymath}
\end{rem}

The main property of GKZ sheaves of $\gd$-modules is given by~:

\begin{thm}\label{thm:GKZ_locally,free} 
Let $X$ be a toric smooth projective variety endowed with a split vector bundle $\vb=\oplus_{i=1}^k\lb_i$. 
Assume that $\omega_X\otimes {\lb_1}^{\dual}\otimes \cdots \otimes {\lb_k}^{\dual}$ and each line bundle $\lb_i$ is nef. 
Let $c_\top$ be the top chern class of $\vb$ and let
 $\mc$ be the morphism of multiplication  by $\ctop$ in $H^ {2*}(X)$.

Let  $\GKZmodsheaf$ and $\GKZmodsheaf^\res$ be the twisted and residual GKZ sheaf  of $\gd$-modules associated to $(X,\vb)$, as defined in \ref{def:GKZ}.
Let $\freetoric$ be the open subset of $\spectoric$ defined in \ref{notn:toric_locus}. We have~:
  \begin{enumerate} 
  \item Over $\freetoric\times\C$, $\GKZmodsheaf$ is a 
  locally free $\go_{\freetoric\times\C}$-modules of rank $\dim
  H^{2*}(X)$.
  \item Over $\freetoric\times\C$, $\GKZmodsheaf^\res$ is a 
  locally free $\go_{\freetoric\times\C}$-modules of rank $(\dim
  H^{2*}(X)-\dim \ker \mc)$.
  \end{enumerate}
\end{thm}

\begin{proof}
This theorem follows from Proposition \ref{prop:coherence_of_GKZ_sheaves} ($\GKZmodsheaf_{|\freetoric\times \C}$ and $\GKZmodsheaf_{|\freetoric\times \C}^\res$ are coherent), Proposition \ref{prop:freeness_of_M} ($\GKZmodsheaf_{|\freetoric\times \C}$ is locally free of the expected rank) and 
Proposition \ref{prop:Res,locally,free} below.
In this last proposition we only prove that $\GKZmodsheaf^\res_{|\freetoric\times \C}$ is locally free over $ \freetoric\times \C^*$ (that is on $z\neq 0$) and   isomorphic to the residual Batyrev algebra on $z=0$. By Nakayama's Lemma, this only gives an inequality on the dimension
of $\GKZmodsheaf^\res_{|\freetoric\times \C^*}$.

We are left to show that $\GKZmodsheaf^\res_{|\freetoric\times \C^*}$ has the expected rank over $z\neq 0$. 
This point follows from Mirror symmetry and will be proved in section
\ref{sec:isomorphism-theorems} (\cf Remark \ref{rem:missing_point-for-freeness_of_GKZres}). 
\end{proof}

\subsection{Coherence of GKZ sheaves associated to $(X,\vb)$}

\begin{prop}\label{prop:coherence_of_GKZ_sheaves}
Under assumptions of Theorem \ref{thm:GKZ_locally,free},
$\GKZmodsheaf_{|\freetoric\times \C}$ and $\GKZmodsheaf_{|\freetoric\times \C}^\res$ are coherent sheaves of $\go_{\freetoric\times \C}$-modules.
\end{prop}

\begin{proof}
  If ${\GKZmodsheaf}$ is coherent then the surjective
  morphism ${\GKZmodsheaf}\to {\GKZmodsheaf^\res}$ implies that ${\GKZmodsheaf^\res}$ is finitely generated.
Hence, it is sufficient to  show that ${\GKZmodsheaf}$ is coherent over $\freetoric\times \C$.
  
  A usual proof of coherence for a differential module,
  consists in finding a good filtration and proving that the characteristic
  variety is supported by the zero section of the cotangent bundle
  (\cf \cite[\S 3]{Sevenheck-GKZ-log-Frob-manifold} and
  \cite[Proposition 1.2.8]{Sabbah-Asterisque}).  In our case, $\GKZmod$ is not a $\cc[q_a^\pm,z]\langle
  \partial_{q},\partial_{z}\rangle$-module but only a $\C[q_a^\pm,z]\langle z\de_q\rangle$-module~; we must adapt the
  classical proof (\cf for instance \cite[Proposition
  2.2.5]{Hotta-D-module})~:
  
  We make use of notations of Remark \ref{rem:alternative_definitions_of_GKZ}.(1) and put 
 $\GKZmod':=\D'/{\GKZid}'$.
  Let us define the following increasing filtration of $\D'$:
\begin{align*}
  F_{k}\D'&:=\biggl\{P\in \D'\mid P(q,z,z\delta_{q})=\sum_{\stackrel{\alpha\in \nn^{r}}{|\alpha|\leq k}}P_{\alpha}(q,z)(z\delta_{q})^{\alpha}\biggr\}
\end{align*}
where
$(z\delta_{q})^{\alpha}:=(z\delta_{q_{1}})^{\alpha_{1}}\cdots(z\delta_{q_{r}})^{\alpha_{r}}$.
Let $\Gr \D'$ be the graduated ring of $\D'$ defined by this filtration.  
Denote by $y_{a}$ the class of
$z\delta_{q_{a}}$ in $\Gr \D'$, then $\Gr \D'$ is isomorphic to
$\cc[q_a^\pm,z][y_{1}, \ldots ,y_{r}]$.  
We define the \emph{symbol} of an element $P=\sum_{\alpha\in
  \nn^{r}}P_{\alpha}(q,z)(z\delta_{q})^{\alpha}$ of $\D'$ by~:
\begin{align*}
\sigma(P)=\sum_{\stackrel{\alpha\in \nn^{r}}{|\alpha|=\deg P}}P_{\alpha}(q,z)y^{\alpha}.
\end{align*}
We also define an increasing filtration on $\GKZmod'$  by
\begin{align*}
  F_{k}\GKZmod':=F_{k}\D'/{\GKZid}'_{k},
    &\text{\quad where \quad}{\GKZid}'_{k}:=F_{k}\D'\cap {\GKZid}'.
\end{align*}
One can check that $(F_{k}\GKZmod')_{k\geq 0}$
satisfies the properties of a good filtration~; in particular,
for any $k$ in $\nn$, 
  $F_{k}\GKZmod'$ is a coherent $\cc[q_a^\pm,z]$-module.
We have 
$
  \Gr \GKZmod'=\Gr \D'/\Gr {\GKZid}', 
$
which shows that the annihilator ideal of $ \Gr \GKZmod'$ in $\Gr\D'$ is $\Gr {\GKZid}'$.
Recall that the characteristic variety of $\GKZmod'$ is the subscheme of $\Spec \Gr\D'$ defined by the radical of the annihilator of $\Gr \GKZmod'$.
Put $\mathbf A^r_{\spectoric\times \C}=\Spec \Gr \D'=\Spec \cc[q_a^\pm,z][y_i]$~; denote by $\mathbf C\subset \mathbf A^r_{\spectoric\times \C}$ the characteristic variety of  $\GKZmod'$ defined by the ideal $\sqrt{\mathrm{Ann}\Gr \GKZmod'}$. 
Let $\freetoric$ be the open subset of $\spectoric$ defined in Notations \ref{notn:toric_locus} and 
 $\mathbf C_{\freetoric\times \C}\subset \mathbf A^r_{\freetoric\times \C}$ be the pull-back of $\mathbf C$ by the open immersion ${\freetoric\times \C}\hookrightarrow {\spectoric\times \C}$. 

\begin{lem}\label{lem:caract,section,nulle}
The characteristic variety $\mathbf C_{\freetoric\times \C}$
is the image of the zero section of the trivial bundle
$\mathbf  A^r_{\freetoric\times \C}\ra \freetoric\times \C$. It
is defined by the ideal $\langle y_1,\ldots, y_r\rangle$.
\end{lem}

 \begin {proof} By definition of the symbol, 
the characteristic variety is contained in the closed subscheme of $\mathbf A^r_{\spectoric\times \C}$  defined by the ideal
$$
 J=\langle \sigma(\square_d),d\in H_2(X,\Z)\rangle\subset \cc[q_a^\pm,z][y_1,\ldots,y_r].
 $$
Consider the Batyrev $\La$-algebra $\bat$ defined in \ref{def:graded_batyrev_algebra}. After localisation of $\La$ and tensorization by $\C[z]$, one get
a $\C[q_a^\pm,z]$-graded algebra. 
 There is a natural surjective morphism~:
 \begin{align*}
\al: \C[q_a^\pm,z][ x_\rho,h]&\lra \Gr \D'=\cc[q_a^\pm,z][y_1,\ldots,y_r]\\
h &\longmapsto 0\\
x_\rho &\longmapsto 
\begin{cases}
\sum_{a=1}^r D_\rho^a y_a & \text{if\ $\rho\in\raysbase$}\\
- (\sum_{a=1}^r D_\rho^a y_a) & \text{if\ $\rho\in\raysfiber$}\\
\end{cases}
\end{align*}where the integers $D_\rho^a$ are defined by~: $[D_\rho]=\sum_{a=1}^r
D_\rho^a T_a$ (cf. \ref{notn:coordinates_of_D_rho}). 
One check that, taking the quotients, 
the morphism $\al$ gives an isomorphism~:
$$
\C[q_a^\pm,z][ x_\rho,h]/(\QSR^h+\Lin+\langle h\rangle) \simeq \cc[q_a^\pm,z][y_1,\ldots,y_r]/J.
$$
Let $p$ be a closed point of $\freetoric$, and $\kappa\simeq \C$ its residual field. Let $\ov{\QSR}^h$ be the
image of $\QSR$ in $\kappa [x_\rho,h]$, and $\ov J$ be the image of $J$ in $\kappa[z][y_1,\ldots,y_r]$. 
By definition of  $\freemori$ (Lemma \ref{lem:freeness_neighbourhood}), the radical of $(\ov \QSR^h+\ov \Lin+\langle h\rangle)$
is the ''irrelevant'' ideal $\langle h, x_\rho, \rho\in\De (1)\rangle$. 
This shows that the radical of $\ov J$ is equal to $\langle \al(x_\rho), \rho\in\De (1)\rangle= \langle y_1,\ldots, y_r\rangle$. 
\end{proof}
Denote by $\gd'$ and $\GKZidsheaf'$ the sheaves associated to $\D'$ and $\GKZid'$.
Consider the sheaf of ideals $\mathcal I$ in $\Gr \mathcal{D}'$, generated by $\{y_1,\ldots, y_r\}$.
By Lemma  \ref{lem:caract,section,nulle} above, there exists $m_0\in \N$ such that $\mathcal I^{m_0}_{|\freetoric\times\C}  \subset \Gr {{\GKZidsheaf}}'_{|\freetoric\times\C}$  (one may take
$m_0=m_{1}+\cdots+m_{r}$ where $y_{i}^{m_{i}}\in \Gr
  {{\GKZidsheaf}}'_{|\freetoric\times\C}$).  We have~:
  $$
  F_{m_{0}+k}{\GKZmodsheaf}'_{|\freetoric\times\C}=F_{m_{0}}\mathcal{D}'_{|\freetoric\times\C}\cdot F_{k}\GKZmodsheaf'_{|\freetoric\times\C}
  $$
which shows that the increasing filtration $F_{k}{\GKZmodsheaf}'_{|\freetoric\times\C}$
is stationary after $m_{0}$. But we know that  $F_{k}\GKZmodsheaf'$ is a coherent $\go_{\spectoric\times \C}$-module. 
\end{proof}

\subsection{Local freeness and rank of the twisted GKZ sheaf associated to $(X,\vb)$}

\begin{prop} \label{prop:freeness_of_M} Under assumptions of Theorem \ref{thm:GKZ_locally,free},
the $\go_{\freetoric\times\C}$-module
  $\GKZmodsheaf_{|\freetoric\times\C}$ is locally free of rank $\dim
  H^{2*}(X)$.
\end{prop}

\begin{proof}
  The following proof is inspired from Theorem 2.14 of
  \cite{Sevenheck-GKZ-log-Frob-manifold}, with modifications taking into account 
the twisting by $\vb$ and the use of $q_a$ variables instead of $(\la_\rho)_{\rho\in\De(1)}$.

\noindent \textbf{Step 1.} \emph{$\GKZmodsheaf/z\GKZmodsheaf$ is locally free of rank $H^{2*}(X)$.}

Let $\bat$ be the Batyrev algebra $\La[x_\rho]/\<\QSR+\Lin\>$ defined in \ref{def:batyrev_algebra}. 
Localizing $\La$ by inverting $Q^d$ ($d\neq 0$) gives  $\C[q_a^\pm]$. There is an isomorphism of $\C[q_a^\pm]$-algebra~:
\begin{align}\label{eq:map:batyrev->GKZ_0}
B\otimes \C[q_a^\pm]\simeq  \C[q_a^\pm,x_\rho]/\! \<\QSR+\Lin\>&\lra \GKZmod/z\GKZmod=\D'/(\<z\>+\GKZid') \\
  x_\rho & \longmapsto
\begin{cases}
\sum_{a=1}^r D_\rho^a z\de_{q_a}& \mbox{ if } \rho\in \raysbase \\
-\sum_{a=1}^r D_\rho^a z\de_{q_a}& \mbox{ if }\rho\in \raysfiber
\end{cases}\nonumber
\end{align}
By Theorem \ref{thm:Batyrev_locally_free}, $B$ is locally free of rank $\dim
H^{2*}(X)$ over $\freemori$~; then $\GKZmodsheaf/z\GKZmodsheaf$ is locally free of rank $\dim
H^{2*}(X)$ over $\freetoric=\freemori\cap \spectoric$.

\noindent \textbf{Step 2.} \emph{$\GKZmodsheaf$ is locally free over $\freetoric\times \C^*$.}

By Proposition \ref{prop:coherence_of_GKZ_sheaves}, $\GKZmodsheaf|_{\freetoric\times\C}$ is a coherent
$\go_{\freetoric\times\C}$-modules. If $z$ is invertible, Theorem 1.4.10 of \cite{Hotta-D-module} shows that
the coherent sheaf $\GKZmodsheaf$ is actually locally free. 

\noindent \textbf{Step 3.} \emph{Up to a pull-back, $\GKZmod$ is a  GKZ-module studied in Adolphson's article \cite{adolphson_hypergeometric}.}

Let $\{\la_\rho, \rho\in\De(1)\}$, be a set of indeterminates. Put
$\D^{{1}}=\C[\la_\rho^{\pm }]\langle \partial_{\la_\rho}\rangle$, with the usual relations $\partial_{\la_{\rho}}.\la_\rho=\la_\rho. \partial_{\la_\rho}+1$.
For any $d\in H_2(X,\Z)$, put $\square^{{1}}_d=\partial_\la^{d^+}-\partial_\la^{d^-}$. Consider the vector
$\beta=(0_N,-1,\ldots, -1)\in N\times \Z^k$ and for any $u\in M'$ put $\gz^{{1}}_u=\sum_\rho \langle u,v_\rho\rangle \la_\rho\partial_{\la_\rho}-\langle u,\beta\rangle$. Then the $\D^{{1}}$-module $\D^{{1}}/\langle \square^{{1}}_d,\gz^{{1}}_u\rangle$
is studied in  \cite{adolphson_hypergeometric}.

Let $\varphi$ be the injective morphism~:
\begin{align}\label{eq:morphisme_phi_de_q_aux_lambda}
  \varphi:\C[q_a^\pm]&\longhookrightarrow \C[\lambda^{\pm}]\\
q_{a}&\longmapsto \prod_{\rho\in\raysfiber}(-\la_\rho)^{D_\rho^a}\prod_{\rho\in\raysbase}\la_\rho^{D_\rho^a},%=(-1)^{c_1(\vb).B_a}\la^{B_a} \nonumber
\end{align}
where the $D_\rho^a$ are defined in Remark \ref{rem:alternative_definitions_of_GKZ}.(2).
viewing $\C[\la^{\pm}]$ as a $\C[q_a^\pm]$-algebra, we claim that there exist an isomorphism~: 
\begin{align}\label{eq:isomorphisme_Adolphson-GKZ}
\GKZmod \otimes_{\C[q_a^\pm]}\C[\la_\rho^\pm]  \isom \D^{{1}}/\langle \square^{{1}}_d,\gz^{{1}}_u\rangle \otimes_\C \C[z^\pm]
\end{align}

To construct this isomorphism, put $\D^{{2}}=\C[\la_\rho^{\pm},z]\langle z\partial_{\la_\rho}\rangle$, with the relations $z\partial_{\la_{\rho}}.\la_\rho=\la_\rho. z\partial_{\la_\rho}+z$.
For any $d\in H_2(X,\Z)$, put $\square^{{2}}_d=(z\partial_\la)^{d^+}-(z\partial_\la)^{d^-}$. Consider as above the vector
$\beta=(0_N,-1,\ldots, -1)\in N\times \Z^k$ and for any $u\in M'$ put $\gz^{{2}}_u=\sum_\rho \langle u,v_\rho\rangle \la_\rho z\partial_{\la_\rho}-\langle u,\beta\rangle z$.
Sending $\la_\rho$ to $z\la_\rho$, $z\partial_{\la_\rho}$ to $\partial_{\la_\rho}$ and $z$ to $z$, one get an isomorphism~:
$$
\D^{{2}}/\langle \square^{{2}}_d,\gz^{{2}}_u\rangle \isom \D^{{1}}/\langle \square^{{1}}_d,\gz^{{1}}_u\rangle \otimes_\C \C[z^\pm].
$$

Then consider as in remark \ref{rem:alternative_definitions_of_GKZ}.(2) the module $\D''/\langle \square''_d,Z''_u\rangle$, isomorphic to $\GKZmod$.
Put, in $\D^{{2}}$, $\ell:=\prod_{\rho\in\raysfiber}\la_\rho$.
There is an injective morphism of non commutative $\C[z^\pm]$-algebras~:
\begin{align*}
\D''& \longhookrightarrow \D^{{2}} \\ 
z\de_{\la_\rho}&\longmapsto \ell^{-1}(\la_\rho .z\partial_{\la_\rho})\ell
=
\begin{cases}
\la_\rho .z\partial_{\la_\rho} &\text{if $\rho\in\raysbase$}\\
\la_\rho .z\partial_{\la_\rho}+z &\text{if $\rho\in\raysfiber$}
\end{cases}
\\
q_a&\longmapsto  \prod_{\rho\in\raysfiber}(-\la_\rho)^{D_\rho^a}\prod_{\rho\in\raysbase}\la_\rho^{D_\rho^a},
\end{align*}
 which gives~:
\begin{align*}
\GKZmod\otimes_{\C[q_a^\pm]}\C[\la_\rho^\pm]\simeq \D''/\langle \square_d'',\gz''_u\rangle\otimes_{\C[q_a^\pm]}\C[\la_\rho^\pm]& \isom \D^{{2}}/\langle \square^{{2}}_d,\gz^{{2}}_u\rangle.
\end{align*}

\noindent \textbf{Step 4.} \emph{The rank of $\GKZmodsheaf$ over $\freetoric\times \C^* $ is $\dim H^{2*}(X,\C)$.}

The morphism $\varphi$ defined in \ref{eq:morphisme_phi_de_q_aux_lambda} is injective~; this gives a surjective morphism 
$h: \Spec \C[\la_\rho^\pm]\ra \spectoric=\Spec \C[q_a^\pm]$, and $\mathbf O=h^{-1}(\freetoric)$
is a dense  open  subset  of the irreducible smooth variety $\Spec \C[\la_\rho^\pm]$.

The isomorphism \ref{eq:isomorphisme_Adolphson-GKZ} ensures that, over $\mathbf O$, the differential module
$\D^{{1}}/\langle \square^{{1}}_d,\gz^{{1}}_u\rangle$ is locally free of rank equals to the generic rank of $\GKZmodsheaf$. Moreover, by Corollary 5.11 of \cite{adolphson_hypergeometric}
the  rank  of  $\D^{{1}}/\langle \square^{{1}}_d,\gz^{{1}}_u\rangle$ is $(n+k)!\text{Vol}(\Gamma_{\Delta})$ where
$\Gamma_{\Delta}$ is the convex hull of the points
$\{0,v_\rho,\rho\in\De(1)\}$  in $N'_\R$.

Since all the $\tordiv_i$ are  nef, the fan $\De$ is convex, and $0$ is not in the interior of this convex
hull. Since the divisor $-K_X-\sum_{i=1}^kL_i$ is nef, the vectors
$(v_1,\ldots, v_k)\in N\times \Z^k$ defined by the toric divisors
$L_i$ all are either vertices or contained in faces of $\Gamma_{\Delta}$ which do not
contain $0$.  Hence, $\Gamma_{\Delta}$ is a "disjoint" (except for faces) union of the
simplexes $\Gamma_{\Delta}(\tau):=(v_1,\ldots, v_k, (v_{\rho_\theta})_{\theta\in\tau})$ where $\tau$ is any simplex defined by $0\in N_\R$ and generating
vectors of rays of $\Sigma$ (we make use of notations of Section
\ref{subsection:Fan}).
Let $\Gamma_{\Sigma}$ be the convex hull of the points
$\{0,w_\theta,\theta\in\Sigma(1)\}$ in $N_\R$. We have~:
\begin{align*}
\rank (\GKZmodsheaf) &= (n+k)!\text{Vol}(\Gamma_{\Delta})=\sum_{\tau, \text{ simplex of $\Sigma$}}(n+k)!\text{Vol}(\Gamma_{\Delta}(\tau))\\
&=\sum_{\tau, \text{ simplex of $\Sigma$}}|\det(v_1,\ldots, v_k, (v_{\rho_\theta})_{\theta\in\tau})|
=\sum_{\tau, \text{ simplex of $\Sigma$}}|\det((w_{\theta})_{\theta\in\tau})|\\
&=\sum_{\tau, \text{ simplex of $\Sigma$}}n!\text{Vol}(\Gamma_{\Sigma}(\tau))=n!\text{Vol}(\Gamma_{\Sigma})=\dim H^{2*}(X).
\end{align*}
\end{proof}

\subsection{Local freeness and rank of the residual GKZ sheaf associated to $(X,\vb)$}

\begin{prop}\label{prop:Res,locally,free} 
Under assumptions of Theorem \ref{thm:GKZ_locally,free}. 
  \begin{enumerate}
  \item On $z=0$,  the $\mathcal{O}_{\freetoric}$-module
    $({\GKZmodsheaf^\res}/z \GKZmodsheaf^{\res})|_{\freetoric}$ is locally free of rank
    $\dim_{\cc}\overline{ H^{2*}(X)}=\dim_{\cc}
    H^{2*}(X)-\dim_{\cc}\ker( \mc)$.
  \item  On $z\neq 0$, the $\mathcal{O}_{\freetoric\times \C^{*}}$-module
    ${\GKZmodsheaf^\res}\mid_{\freetoric \times \C^{*}}$ is locally free of rank less than  $\dim_{\cc}\overline{ H^{2*}(X)}$.
  \end{enumerate}
\end{prop}

\begin{proof}
On $z\neq 0$, ${\GKZmodsheaf^\res}|_{\freetoric\times\cc^{*}}$ is locally free by Theorem 1.4.10 of \cite{Hotta-D-module}, as 
as in Step 2 of the proof of Proposition \ref{prop:freeness_of_M}. 
By Nakayama's lemma, it is enough to prove the first statement.

Consider the residual Batyrev $\La$-algebra $\bat^{\res}=\La[x_\rho]/(\GKZcom
:x_\top)$ defined in Subsection \ref{subsec:residual_Batyrev_ring}.
By Proposition \ref{thm:Batyrev_locally_free}.2,
$\bat^{\res}$ is a locally free module of rank $\dim_{\cc}
    H^{2*}(X)-\dim_{\cc}\ker( \mc)$
over the open subscheme
$\freemori\subset \specmori$ defined in Lemma \ref{lem:freeness_neighbourhood}.
Proposition \ref{prop:Res,locally,free} follows from the Lemma below. \end{proof}

\begin{lem}
Consider  the alternative definition of $\GKZmod$ given in Remark \ref{rem:alternative_definitions_of_GKZ}.(2)~:
$\GKZmod=\D''/\GKZid''$, where $\D''=\C[q_a^\pm,z]\langle z\de_{\la_\rho}\rangle$ and $\GKZid'':= \langle \sq''_d,\gz''_u\rangle$. Put $\hatctop=\prod_{\rho\in\raysfiber} (-z\de_{\la_\rho})\in\D''$. 

Then
$
\GKZmod^\res \simeq \D''/(\GKZid'': \hatctop)
$
and
the following application is a well defined isomorphism of commutative $\C[q_a^\pm]$-algebras~:
\begin{align}\label{eq:isomorphisme_Bres_GKZres_z=0}
\GKZmod^\res/z\GKZmod^\res &{\lra} \bat^\res\otimes_\La \C[q_a^\pm]\\
z &\longmapsto 0\nonumber \\
z\de_{\la_\rho} &\longmapsto  \nonumber
\begin{cases}
x_\rho & \text{if $\rho\in\raysbase$}\\
-x_\rho & \text{if $\rho\in\raysfiber$}
\end{cases}
\end{align}
\end{lem}
\begin{proof}
The first isomorphism $\GKZmod^\res \simeq \D''/(\GKZid'': \hatctop)$ is immediate. We make use of the same notation, $\hatctop$, in $\D$ or in $\D''$.

 Consider the morphism of $\C[q_a^\pm]$-algebras~:
 \begin{align*}
 h: \C[q_a^\pm,z]\langle z\de_{\la_\rho}\rangle  &\lra \C[q_a^\pm][x_\rho]\\
 z& \longmapsto 0\\
 z\de_{\la_\rho} & \longmapsto 
 \begin{cases}
 x_\rho & \text{if $\rho\in\raysbase$}\\
-x_\rho & \text{if $\rho\in\raysfiber$},
 \end{cases}
 \end{align*}
well defined since $z$ is sent to $0$.
For any $d\in \mori{Y}$ 	and $u\in M'$ we have~: 
$$h(\square''_d)=R_d, \quad h(\gz''_u)=Z_u.$$

To prove that \ref{eq:isomorphisme_Bres_GKZres_z=0} is a well defined isomorphism, we must show that each polynomial $P\in (G:x_\top)$ in
$\C[q_a^\pm][x_\rho]$
possesses an antecedent for $h$ in $(\GKZid :\hatctop)$. 

Let us choose a section of $h$ as a morphism of $\C[q_a^\pm]$-module.
First consider the following isomorphism of $\C$-algebras~:
\begin{align*}
\widehat{.}: \C[x_\rho] & \lra \C[z\de_{\la_\rho}]\\
x_\rho & \longmapsto   \widehat{x_\rho}=
\begin{cases}
z\de_{\la_\rho} &\text{ if $\rho\in \raysbase$},\\
  -z\de_{\la_\rho} &\text{ if $\rho\in \raysfiber$},
 \end{cases}
\end{align*}
and extend it $\C[q_a^\pm]$-linearly to $\widehat{.}: \C[q_a^\pm][x_\rho]  \ra \C[q_a^\pm,z]\langle z\de_{\la_\rho}\rangle$. 
For any $P\in \C[q_a^\pm][x_\rho]$, one check that $h\bigl(\widehat P\bigr)=P$.

Let $P$ be in $(G:x_\top)$. Recall that the ideal $\QSR$ is generated by polynomials $\qsrpol_d, d\in \gp,$ where $\gp$ is the set of primitive classes. Let $(u_i, i\in I=\{1,\ldots,k\})$  be a base of the dual lattice of $N'=N\oplus \Z^k$. Put 
$Z_i:=\sum_{\rho\in \De(1)} \langle u_i,v_\rho\rangle x_\rho$.
The ideal $\Lin$ is generated by polynomials $\{Z_i, i\in I\}$. 
Then we can write~:
\begin{align}\label{eq:ecriture_de_x_top}
x_\top P&=\sum_{d \in \gp}A_d\qsrpol_d +\sum_{i\in I} B_iZ_i,\qquad  A_d, B_i\in \C[q_a^\pm][x_\rho].
\end{align}
We need to find $\widetilde P\in (\GKZid'' :\hatctop)$ such that $h\bigl(\widetilde P\bigr)=P$.
For that, we may assume that $x_\top$ does not divide any monomial of $A_d$ or $B_i$.
If not we have, for any $d\in\gp$ or $i\in I$ a unique decomposition~:
$$
A_d=A_{d,1}+x_\top A_{d,2},\quad B_i=B_{i,1}+x_\top B_{i,2}.
$$
where $x_\top$ does not divide any monomial of $A_{d,1}$ or $B_{i,1}$.
Put $P_2=\sum A_{d,2}\qsrpol_d+\sum B_{i,2} Z_i$ and $\widetilde{P_2}=\sum \widehat{A_{d,2}}\square''_d+\sum \widehat{B_{i,2}}\gz''_i$.
Since $\widetilde{P_2}$ is in $\GKZid''$, it is also in $(\GKZid'':\hatctop)$~; moreover, $h\bigl(\widetilde{P_2}\bigr)=P_2$.
If we find $\widetilde {P_1}\in (\GKZid'':\hatctop)$ such that $h\bigl(\widetilde {P_1}\bigr)=P-x_\top P_2$,
then $\widetilde {P}=\widetilde {P_1}+\widetilde {P_2}$ is an antecedent of $P$ in $(\GKZid'':\hatctop)$.

Assume now that $x_\top$ does not divide any monomial of $A_d$ or $B_i$.

Since each primitive class is in the Mori cone, and each line bundle $\lb_i$ is ample,
$x_\top=\prod_{\rho\in \raysfiber}x_\rho$ divides $x^{d^-}$ for any $d\in\gp$ and  we can write~:
$$
\qsrpol _d=x^{d^+}-Q^dx_\top x^{d^--\eps},
$$
 where $\eps=(\eps_\rho)_{\rho\in\De(1)}, \eps_\rho=1$ if $\rho\in\raysfiber$, 
$\eps_\rho=0$ if $\rho\in\raysbase$. 
In the same way, for any $i\in I$, we write
$$
Z_i=Z'_i+a_ix_\top, 
$$
where $a_i\in\C$, and $x_\top$  does not divide any term of $Z'_i$. Since $\deg x_\top=k$, $a_i=0$ if $k>1$, 
and $a_i=\langle u_i,v_{\rho_\top}\rangle$ if $k=1$.

Finally, for any $d\in\gp$, $i\in I$, we write~:
\begin{align*}
A_d& =\sum_{\al \in\Z^r} q^{\al}A_{d,\al}, \quad B_i =\sum_{\al \in\Z^r} q^{\al}B_{i,\al}, \quad \text{where }A_{d,\al}, B_{i,\al}\in \C[x_\rho]
\end{align*}
In the quotient ring $\C[q_a^\pm][x_\rho]/(x_\top)$ we obtain from (\ref{eq:ecriture_de_x_top}), for any $\al\in\Z^r$~:
\begin{align}\label{eq:annulation_somme_poly}
\sum_{d \in \gp}A_{d,\al}x^{d^+} +\sum_{i\in I} B_{i,\al}Z'_i=0.
\end{align}

For any $\al\in\Z^r$, put
\begin{align*}
\widetilde P_\al &= -\sum_{d\in\gp}\widehat{ A_{d,\al}} q^d\hspace{-1mm}
\prod_{\rho\in\raysfiber}\hspace{-1mm}[-z\de_{\la_\rho}+zd_\rho^--z]_{d_\rho^--1}\hspace{-1mm}
\prod_{\rho\in\raysbase}\hspace{-1mm}[z\de_{\la_\rho}]_{d_\rho^-}+ 
\sum_{i\in I}\widehat{ B_{i,\al}} a_i,
 \end{align*}
 where we make use of the Pochammer symbol defined in Notation \ref{notn:quantization}. Set~:
 \begin{align*}
\widetilde P &= \sum _{\al\in\Z^r}q^{\al}\widetilde P_\al.
\end{align*}
Then $h(\widetilde P_\al)=-\sum_{d\in\gp}{ A_{d,\al}} q^d x^{\gamma_d}+\sum_{i\in I}{ B_{i,\al}} a_i$, and
 (\ref{eq:annulation_somme_poly}) gives~: 
\begin{align*}
x_\top(P-h(\widetilde P)) &= \sum_{\al \in\Z^r}q^\al \left( 
 \sum_{d \in \gp}A_{d,\al}x^{d^+} +\sum_{i\in I} B_{i,\al}Z'_i 
   \right)=0 
\end{align*}
Which proves that $P=h(\widetilde P)$.

We claim that $\widetilde P\in (\GKZid'': \hatctop)$. 
By definition of the quotient ideal $(\GKZid'': \hatctop)$, it is sufficient to show that, for any $\al \in \Z^r$,
$\hatctop\widetilde P_\al\in\GKZid$. 

Recall that, for any $\rho\in \De(1)$,  $z\de_{\la_\rho}.q_a=q_a.z\de_{\la_\rho}+D_\rho^a zq_a$
(Remark \ref{rem:alternative_definitions_of_GKZ}.(2)). Then we have, for $d\in\mori{Y}$~:
\begin{align}\label{eq:produit_par_classe_top}
\hatctop q^d=\hspace{-3.5mm}\prod_{\rho\in \raysfiber}\hspace{-3.5mm}(-z\de_{\la_\rho})\prod_a q_a^{T_a.d}=
\hspace{-3.5mm}\prod_{\rho\in \raysfiber}\hspace{-3.5mm}q^d(-z\de_{\la_\rho}-z\sum_{a=1}^r(T_a.d)D_\rho^a)=
q^d\hspace{-3.5mm}\prod_{\rho\in \raysfiber}\hspace{-3.5mm}(-z\de_{\la_\rho}+zd_\rho^-)
\end{align}
since $d_\rho=d_\rho^-$ for any $d\in\mori{Y}$ and $\rho\in\raysfiber$.

Applying morphism $\widehat{\ .\ }$ to equality (\ref{eq:annulation_somme_poly}), which does not contain any variable $q_a$, we have~:
$$
\sum_{d \in \gp}\widehat{A_{d,\al}}\widehat{x^{d^+}} +\sum_{i\in I} \widehat{B_{i,\al}}\widehat{Z'_i}=0.
$$
Moreover, since $d$ is a primitive class, coefficients of $d^+=(d_\rho^+)_{\rho\in \De(1)}$ are either equal to $0$ or $1$, 
and $d_\rho^+=0$ if $\rho\in\raysfiber$. Thus
$
\widehat{x^{d^+}}=\prod_{\rho\in \De(1)}(z\de_{\la_ \rho})^{d_\rho^+}=\prod_{\rho\in \raysfiber}[-z\de_{\la_\rho}+zd_\rho^+]_{d_\rho^+}\prod_{\rho\in \raysbase}[z\de_{\la_\rho}]_{d_\rho^+},
$
which gives~:
\begin{align}
\sum_{d \in \gp}\widehat{A_{d,\al}}\prod_{\rho\in \raysfiber}[-z\de_{\la_\rho}+zd_\rho^+]_{d_\rho^+}\prod_{\rho\in \raysbase}[z\de_{\la_\rho}]_{d_\rho^+} +\sum_{i\in I} \widehat{B_{i,\al}}\widehat{Z'_i}=0.
\label{eq:annulation_somme_operateurs}
\end{align}
Finally, equalities (\ref{eq:annulation_somme_operateurs})
and  (\ref{eq:produit_par_classe_top}) gives~:
\begin{align*}
\hatctop \widetilde{P_\al} &= -\sum_{d\in\gp}\hatctop \widehat A_{d,\al} q^d
\prod_{\rho\in\raysfiber}[-z\de_{\la_\rho}+d_\rho^--1]_{d_\rho^--1}\prod_{\rho\in\raysbase}[z\de_{\la_\rho}]_{d_\rho^-}+ 
 \sum_{i\in I}\hatctop \widehat B_{i,\al} a_i\\
 	&=  \sum_{d \in \gp}\widehat{A_{d,\al}}\square''_d + \sum_{i\in I} \widehat{B_{i,\al}}\gz''_i \in \GKZid''. 
\end{align*}
\end{proof}

%%%%%%%%%%%%%%%%%%%%%%%%%%%%%%%%%%%%%%%%%%%%%%%%%%%%%%%%%%%%%%%%%%%%%%
\section{Isomorphisms between quantum $\gd$-modules and GKZ modules}
\label{sec:isomorphism-theorems}

%%%%%%%%%%%%%%%%%%%%%%%%%%%%%%%%%%%%%%%%%%%%%%%%%%%%%%%%%%%%%%%%%%%%%%
%%%%%%%%%%%%%%%%%%%%%%%%%%%%%%%%%%%%%%%%%%%%%%%%%%%%%%%%%%%%%%%%%%%%%%
\subsection{The mirror Theorem of Givental and Lian-Liu-Yau}
\label{subsec:mirror,sym}
%%%%%%%%%%%%%%%%%%%%%%%%%%%%%%%%%%%%%%%%%%%%%%%%%%%%%%%%%%%%%%%%%%%
%%%%%%%%%%%%%%%%%%%%%%%%%%%%%%%%%%%%%%%%%%%%%%%%%%%%%%%%%%%%%%%%%%%

The mirror theorem was proved by Givental (\cf \cite[Theorem
0.1]{Givental-1998-Mirror-complete-intersection} and \cite[Corrolary
5]{Givental-Coates-2007-QRR}) and by Lian-Liu-Yau \cite{Lian-Liu-Yau-mirror-principle-1-1999}. Our
technics are closed to the work of Givental that  we recall now. As before, $X$
is a smooth toric projective variety endowed with $k$ globally generated line bundles $\lb_{1},
\ldots ,\lb_{k}$ such that $(\omega_{X}\otimes \lb_1\otimes\cdots\otimes\lb_k)^{\dual }$ is nef. We
put $\vb=\oplus_{i=1}^{k}\lb_{i}$.

Denote by $t_{0}$  the coordinate on $H^{0}(X)$ associated to $T_{0}=\mathbf{1}$. In the definition below, we denote by $\vb_{0,1,d}(1)$ the vector bundle on $X_{0,1,d}$ defined in  Subsection \ref{subsubsec:twisted_product}.
 \begin{defn}\label{defi:twisted,J}
   We define the cohomological multi-valued function $J^{\tw}$ by~:
   \begin{align*}
     J^{\tw}(t_{0},q,z)&:=e^{t_{0}/z}q^{T/z}\left(1+
       z^{-1}\sum_{\stackrel{d \in H_{2}(X,\zz)}{d\neq
           0}}{q^{d}}{e_{1}}_*\left(\frac{c_\top (\vb_{0,1,d}(1))}{z-\psi}
         \cap [X_{0,1,d}]^{\vir}\right)\right)
   \end{align*}
 \end{defn}
 where $q$ is in the domain of convergence of the quantum product $\convtoric\subset \spectoric$, $z$ is in $\C$
 and $q^{T/z}=\prod_{a=1}^rq_a^{T_a/z}:=e^{z^{-1}\sum_{a=1}^{r}T_{a}\log (q_{a})}$.

The proposition below is the twisted version of Lemma 10.3.3 of \cite{Cox-Katz-Mirror-Symmetry}.
\begin{prop}\label{prop:J,expression}
Let $L^{\tw}$ be the multivalued section of $\Hom(F,F)$ defined in (\ref{eq:definition_of_L}).
In $H^{2*}(X)$, we have
   \begin{align*}
     \ctopvb J^{\tw}(t_{0},q,z)&=\ctopvb (e^{-t_{0}/z}L^{\tw}(q,z))^{-1}\mathbf{1}
  \end{align*}
In the reduced cohomology ring $H^{2*}(X)/\ker \mc$ we have
  \begin{align*}
\overline{J^{\tw}}(t_{0},q,z)={(e^{-t_{0}/z}\overline{L}}(q,z))^{-1}\overline{\mathbf{1}}.    
  \end{align*}
\end{prop}

\begin{rem}
Notice that $\ctopvb J^{\tw}(t_{0},q,z)$ is exactly $J_{\mathcal{V}}$ of \cite[p.358]{Cox-Katz-Mirror-Symmetry}.
\end{rem}

\begin{proof}[Proof of Proposition \ref{prop:J,expression}]
The first equalities is obtained by 
repeating the proof of Lemma 10.3.3 in \cite{Cox-Katz-Mirror-Symmetry} where one
changes the standard Gromov-Witten axioms by the twisted axioms (see Appendix
\ref{sec:twisted-axiom-for-GW}).
This first equality implies that
$\overline{J^{\tw}(t_{0},q,z)}=e^{t_{0}/z}\overline{(L^{\tw}(q,z))^{-1}\mathbf{1}}$ which is
$(\overline{L}(q,z))^{-1}\overline{\mathbf{1}}$ by definition of $\overline{L}$ (\cf Formula
\eqref{eq:defi;Lred}).
\end{proof}

Recall notations from section \ref{subsection:Fan} and \ref{subsection:Batyrev_algebra}~: to a ray $\theta\in \Sigma(1)$, we associate a toric
divisor denoted by $D_{\theta}$. For any class $d\in H_{2}(X,\zz)$ and any $i\in\{1,\ldots,k\}$ we
put
\begin{align*}
  d_{\theta}&:=\int_{d}D_{\theta} \mbox{\quad and\quad }
  d_{{i}}:=\int_{d}\tordiv_{i}%=\int_{d}c_{1}(\gl_{i}).
\end{align*}
We define a  cohomological multi-valued
function by
  \begin{align}
    \label{eq:definition_of_I}
    I(q,z):=q^{T/z}\sum_{d\in H_{2}(X,\zz)}q^{d}\prod_{i=1}^{k}
\frac{\prod_{m=-\infty}^{d_i}([\tordiv_{i}]+mz)}{\prod_{m=-\infty}^{0}([\tordiv_{i}]+mz)}
\prod_{\theta\in\Sigma(1)}\frac{\prod_{m=-\infty}^{0}([D_{\theta}]+mz)}
{\prod_{m=-\infty}^{d_{\theta}}([D_{\theta}]+mz)}
  \end{align}
where $q^{T/z}:=e^{z^{-1}\sum_{a=1}^{r}T_{a}\log(q_{a})}$.

We develop the $I$-function in power series in $z^{-1}$ and a direct computation gives~:
\begin{align}\label{eq:34}
  I(q,z)&=F(q)\mathbf{1} + z^{-1}G(q) + O(z^{-2})
\end{align}
where $F$ is an invertible univariate scalar function and $G$ takes value in $H^{\leq 2}(X)$.

There exists a natural map $\al: H^2(X,\C)\ra \spectoric$ defined by~:
\begin{align*}
\al: H^2(X,\C) &  \lra  \spectoric = \Spec \C[H_2(X,\Z)]\\
 \tau & \longmapsto  q:= \left[ d \mapsto q^d=\exp \left(2i\pi \int_d\tau\right)\right],
\end{align*}
so that $\al(\sum_{a=1}^r t_aT_a)=(e^{2i\pi t_a})_{a\in\{1,\ldots, r\}}$.

\begin{defn}\label{defn:mirror_map}
The \emph{mirror map} of $(X,\vb)$ is the composite map
\begin{align}\label{eq:mirror,map}
\Mir:  \spectoric& \lra   H^0(X) \times \spectoric \\  
q &  \longmapsto  (\text{Id}\times \al)\left(\frac{G(q)}{F(q)}\right)\nonumber
\end{align}
where $\al: H^2(X,\C)\ra \spectoric$ is defined above, and  $F, G$ are the functions appearing in (\ref{eq:definition_of_I}).
One can check that the mirror map is univariate.
\end{defn}

The mirror theorem of Givental (\cf \cite[Theorem
0.1]{Givental-1998-Mirror-complete-intersection} and \cite[Corrolary
5]{Givental-Coates-2007-QRR}~; see also \cite[Theorem
11.2.16]{Cox-Katz-Mirror-Symmetry} or Lian-Liu-Yau \cite{Lian-Liu-Yau-mirror-principle-1-1999}) tells us the following.
\begin{thm}\cite[Corrolary
  7]{Givental-Coates-2007-QRR}\label{thm:mirror,symmetry,Givental}
Let $\Mir$ be the mirror map defined in \ref{defn:mirror_map}.  

  There exists an open subset 
  $$\mirrortoric=\{(q_a)_{a\in \{1,\ldots,r\}}, |q_a|<\de, \de\in \R_{>0}\}$$ 
  of $\spectoric$ such that 
  \begin{enumerate}
  \item $\Mir(\mirrortoric)$ is contained in  $H^{0}(X)\times \convtoric $ where $\convtoric\subset \spectoric$ is the 
  convergence domain of the quantum product (see Notation \ref{notn:parametre_M}),
\item $ \Mir(q)=(0,q)+O(q)$,
\item $    J^{\tw}(\Mir(q),z)={I(q,z)}/{F(q)}$.
\end{enumerate}
\end{thm}

%%%%%%%%%%%%%%%%%%%%%%%%%%%%%%%%%%%%%%%%%%%%%%%%%%%%%%%%%%%%%%%%%%%%%%
%%%%%%%%%%%%%%%%%%%%%%%%%%%%%%%%%%%%%%%%%%%%%%%%%%%%%%%%%%%%%%%%%%%%%%
\subsection{Quantum $\mathcal{D}$-module of a toric complete intersection in terms of residual GKZ system. }
\label{subsec:main,thm}
%%%%%%%%%%%%%%%%%%%%%%%%%%%%%%%%%%%%%%%%%%%%%%%%%%%%%%%%%%%%%%%%%%%%%%
%%%%%%%%%%%%%%%%%%%%%%%%%%%%%%%%%%%%%%%%%%%%%%%%%%%%%%%%%%%%%%%%%%%%%%

In order to relate the GKZ modules defined in section \ref{sec:GKZ_modules} and quantum $\mathcal{D}$-modules defined in section \ref{sec:quantum-D-modules} we make use of the mirror map. As the target of this map is not $\convtoric$ but  
$H^0(X)\times \convtoric$ (Theorem \ref{thm:mirror,symmetry,Givental}), we first need to extend the base space of the various quantum $\mathcal{D}$-modules defined over $\convtoric$.
We will keep the same notations for these extended $\mathcal{D}$-modules~:
\begin{itemize}
\item  The twisted quantum $\mathcal{D}$-module $\QDM(X,\vb)$ is the trivial bundle $F^{\tw}$ with fibre $H^{2*}(X)$ over $H^0(X)\times \convtoric\times \C_z$ endowed with the connection~:
\begin{displaymath}
 \nabla_{\delta_{z}}=\delta_{z} -\frac{1}{z}\mathfrak{E}\Zprod{q}+\mu , \quad   \nabla_{\partial_{t_{0}}}=\partial_{t_{0}}+\frac{1}{z} \mathbf{1}\twprod{q}\quad
      \text{\quad and\quad }
\forall a
  \in \{1, \ldots ,r\}, \nabla_{\delta_{a}}=\delta_{a}+\frac{1}{z} T_{a}\twprod{q},\
\end{displaymath}
where $\mathfrak{E}=c_{1}(\tang_{X})-c_{1}(\vb)+t_{0}\mathbf{1}$ and
$\mu$ is the unchanged endomorphism of $H^{2*}(X)$ defined in \ref{eq:definition_de_nabla_twist}.
The fundamental solution $L^{\tw}$ is also extended in~:
\begin{align*}
  L^{\tw}(t_{0}, q,z)&:=e^{-t_{0}/z}L^{\tw}(q,z) 
\end{align*}
\item  The reduced quantum $\mathcal{D}$-module $\ov{\QDM}(X,\vb)$ is extended over
$H^0(X)\times \convtoric\times \C_z$ by taking the quotients of $\QDM(X,\vb)$. We do the same for the fundamental solution, which 
gives  $\overline{L}(t_{0},q,z):=e^{-t_{0}/z}\overline{L}(q,z) $.
\item The  ambient quantum $\mathcal{D}$-module $\ov{\QDM}_{\amb}(Z,\vb)$ is the trivial bundle $F^{Z}_{\amb}$  with fiber $H^{2*}_{\amb}(X)$ over 
$H^0(X)\times \convtoric\times \C_z$ endowed with the connection~:
\begin{displaymath}
\nabla^{Z}_{\delta_{z}}=\delta_{z} -\frac{1}{z}\mathfrak{E}^{Z}\Zprod{q}+\mu^{Z},\quad
  \nabla^{Z}_{\partial_{t_{0}}}=\partial_{t_{0}}+\frac{1}{z} \mathbf{1}\Zprod{q}
           \text{\quad and\quad }\forall a
  \in \{1, \ldots ,r\}, \nabla^{Z}_{\delta_{a}}=\delta_{a}+\frac{1}{z} T_{a}\Zprod{q},
\end{displaymath}
where $\mathfrak{E}^{Z}:=c_{1}(\tang_{Z})+t_{0}\mathbf{1}$ and $\mu^{Z}(\psi_{a})=\psi_{a}(\deg(\psi_{a})-\dim_{\cc}Z)/2$.
\end{itemize}

We have~:

\begin{thm}\label{thm:main_theorem} Let $X$ be a projective smooth toric variety endowed with $k$ line bundles
  $\lb_{1}, \ldots ,\lb_{k}$~; put $\vb:=\oplus_{i=1}^{k}\gl_{i}$. Assume that each $\lb_i$ is  globally generated and that $(\omega_{X}\otimes \lb_{1}\otimes \ldots  \otimes\lb_{k})^{\dual }$ is nef. 

Consider the mirror map $\Mir$ and the open subset $\mirrortoric$ of $\spectoric$ defined in Theorem \ref{thm:mirror,symmetry,Givental}.
 For $\varepsilon  \in \R_{>0}$, put
  \begin{displaymath}
\mirrortoric_{\varepsilon}:=\{(q_{1}, \ldots ,q_{r}) \in \mirrortoric \mid 0< |q_{a}| <\varepsilon \}.
  \end{displaymath}
There exists $\varepsilon$ in $\R_{>0}$ such that
\begin{enumerate}
\item Let $\GKZmodsheaf$ be the GKZ sheaf and $\QDM(X,\vb)$ the twisted quantum
  $\mathcal{D}$-module.
Over ${\mirrortoric_{\varepsilon} \times\cc}$, the morphism
\begin{align}\label{eq:twisted_iso_GKZ-Quantum}
{\GKZmodsheaf} &\stackrel{\sim}{\longrightarrow} (\Mir\times\Id)^*\QDM(X,\vb) \\
1 &\longmapsto L^{\tw}(\Mir(q),z)I(q,z)\nonumber
\end{align}
is a well-defined isomorphism of $\mathcal{D}$-modules.

% Over ${\mirrortoric_{\varepsilon} \times\cc}$, there exists an isomorphism of $\gd$ modules between the GKZ sheaf $\GKZmodsheaf$  
% and the pull back of the twisted quantum $\gd$-modules $\QDM(X,\vb)$ by the mirror map~:
% \begin{align}\label{eq:twisted_iso_GKZ-Quantum}
% {\GKZmodsheaf} &\stackrel{\sim}{\longrightarrow} (\Mir\times\Id)^*\QDM(X,\vb) \\
% 1 &\longmapsto L^{\tw}(\Mir(q),z)I(q,z)\nonumber
% \end{align}
\item Let $\lb_{1}, \ldots ,\lb_{k}$ be $k$ ample line bundles. Over ${\mirrortoric_{\varepsilon} \times\cc}$,   the  morphism
\begin{align}\label{eq:residual_iso_GKZ-Quantum}
         {\GKZmodsheaf^\res}|_{\mirrortoric_{\varepsilon} \times\cc} &\stackrel{\sim}{\longrightarrow} (\Mir\times\Id)^{*}\overline{\QDM}(X,\vb)\simeq (\Mir\times\Id)^{*}{\QDM}_{\amb}(Z)\\
         1 &\longmapsto \ov L(\Mir(q),z)\ov{I(q,z)}\nonumber
\end{align}
is a well defined isomorphism of $\gd$-modules.
\end{enumerate}
 \end{thm}

\begin{rem} 
\begin{enumerate}
\item The first point of Theorem \ref{thm:main_theorem} should be known by specialists. However, we did not find a precise reference in our settings.

\item The second point constitutes our main result.
It answers to the               question addressed in the \cite[p.94-95 and
p.101]{Cox-Katz-Mirror-Symmetry}: ``What differential equations shall we add to $\GKZidsheaf$ to get
an isomorphism with $ \QDM_{\amb}(Z)$ ?''.

This result should permit us to compute algorithmically a finite system of differential equation
defining $\QDM_{\amb}(Z)$.
We present it in the Remark \ref{rem:algo}.
\end{enumerate}
\end{rem}

  \begin{lem}\label{lem:well-defined,twisted_morphism} Under the assumption of Theorem \ref{thm:main_theorem}.1,
  the morphism of $\gd$-modules
    \begin{align*}
 {\GKZmodsheaf} &\stackrel{\sim}{\longrightarrow} (\Mir\times\Id)^*\QDM(X,\vb) \\
1 &\longmapsto L^{\tw}(\Mir(q),z)I(q,z)\nonumber
    \end{align*}
    is well defined over $\mirrortoric$.
    \end{lem}
    \begin{proof}[Proof of Lemma \ref{lem:well-defined,twisted_morphism}]
A direct computation shows that 
\begin{displaymath}
  L^{\tw}(\Mir(q),z)I(q,z)=L^{\tw}(\Mir(q),z)J^{\tw}(\Mir(q),z)/F(q)
\end{displaymath}
 is univariate.

    We make use of notations of Definitions \ref{def:GKZ} and \ref{def:idealGKZ}~; we have
    $\GKZmodsheaf$ the sheaf associated to $\mathbb{D}/\mathbb{G}$, where 
    \begin{displaymath}
\mathbb{D}:=\mathbb{C}[q_{a}^{\pm},z]\langle z\delta_{q},z\delta_{z}\rangle, \quad \mathbb{G}:=\<\eulerfield,\sq_d, d\in H_2(X,\Z)\>.
    \end{displaymath}
    % There is a well defined morphism of $\gd$-modules~:
    % $$
    % \begin{array}{ccl}
    % \gd  &\lra &(\Mir\times\Id)^*\QDM(X,\vb)\\
    % P(q_a,z,z\de_{a},z\de_z)  &\longmapsto &P(q,z,z\nabla_{\de_a},z\nabla_{\de_z})L^{\tw}(\Mir(q),z)I(q,z)\\
    %    & &= L^{\tw}(\Mir(q),z)z^{-\mu}z^{c_{1}(\tang_{X})-c_{1}(\vb)}
    %     P(q,z,z\delta_{a},z\delta_{z})z^{-c_{1}(\tang_{X})+c_{1}(\vb)}z^{\mu}I(q,z),
    % \end{array}
    % $$
    % where the last equality comes from Proposition \ref{prop:nabla,flat+Ltw}. Since $\GKZidsheaf$ is generated by $\{\sq_d, d\in H_2(X,\Z)\}$ and $\eulerfield$
   It is sufficient to prove that,  for any $ d\in H_2(X,\Z)$~:
  \begin{align*}
  \square_{d}\left(z^{-c_1(\tang_{X})-c_{1}(\vb)}z^{\mu}I(q,z)\right)&=0,\\
   \text {and\quad } \eulerfield\left(z^{-c_1(\tang_{X}) -c_{1}(\vb)}z^{\mu}I(q,z)\right)&=0.
  \end{align*}
Put   
\begin{displaymath}
  A_{d}(z):=\prod_{i=1}^{k}
 \frac{\prod_{m=-\infty}^{d_i}([\tordiv_{i}]+mz)}{\prod_{m=-\infty}^{0}([\tordiv_{i}]+mz)}
 \prod_{\theta\in\Sigma(1)}\frac{\prod_{m=-\infty}^{0}([D_{\theta}]+mz)} {\prod_{m=-\infty}^{d_{\theta}}([D_{\theta}]+mz)}.
\end{displaymath}
For any $\alpha\in  H^2(X)$,  we have $[\mu,\alpha]=\alpha$. This implies that
\begin{align}\label{eq:7}
z^{\mu}\frac{\alpha}{z}&=\alpha z^{\mu}.  
\end{align}
From this we deduce that $z^{\mu}A_{d}(z)=z^{-d_{\tang_X}+d_{\vb}}A_{d}(1)$ where we set, as usual~: $d_{\tang_X}=\int_{d}c_{1}(\tang_X)$ 
and $d_{\vb}=\int_{d}c_{1}(\vb)$.
Using the definition \eqref{eq:definition_of_I} of the function $I$ we find~:
\begin{align}\label{eq:10}
  z^{-c_1(\tang_X)+c_{1}(\vb)}z^{\mu}I(q,z)=\sum_{d\in H_{2}(X,\zz)}q^{T+d}z^{-c_1(\tang_X)+c_{1}(\vb)-d_{\tang_X}+d_{\vb}}A_{d}(1).
\end{align}
For any class $\alpha\in  H^2(X)$, a direct computation shows that
\begin{align}
  \label{eq:19}
  \widehat{\alpha}q^{T+d}&=q^{T+d}z(\alpha+d_{\alpha}),\\
z\delta_{z}
(z^{-c_1(\tang_X)+c_{1}(\vb)-d_{\tang_X}+d_{\vb}})&=z(-c_1(\tang_X)+c_{1}(\vb)-d_{\tang_X} +d_{\vb})z^{-c_1(\tang_X)+c_{1}(\vb)-d_{\tang_X}+d_{\vb}}.\label{eq:14}
\end{align}
We deduce that 
\begin{align*}
 z\delta_{z}
  \left(q^{T+d}z^{-c_1(\tang_X)+c_{1}(\vb)-d_{\tang_X}+d_{\vb}}\right)&= \left(- \widehat{c_{1}(\tang_X)}+c_{1}(\widehat{\vb})\right)(q^{T+d}z^{-c_1(\tang_X)+c_{1}(\vb)-d_{\tang_X}+d_{\vb}}).
\end{align*}
which gives~: $\eulerfield\left(z^{-c_1(\tang_{X}) -c_{1}(\vb)}z^{\mu}I(q,z)\right)=0$.

Using Formula \eqref{eq:19}, the equality
$\square_{d}(z^{-c_1(\tang_X)+c_{1}(\vb)}z^{\mu}I(q,z))=0$ 
reduces to the following relation~:
\begin{align}\label{eq:33}
  A_{d-d'}(1)\prod_{i=1}^{k}\prod_{\nu=1}^{d_i^{+}}([\tordiv_{i}]+(d-d')_{\tordiv_{i}}+\nu)\prod_{\theta\in\Sigma(1)}\prod_{\nu=0}^{d_{\theta}^{+}-1}([D_{\theta}]+(d-d')_{\theta}-\nu)\\
  =A_{d}(1)\prod_{i=1}^{k}\prod_{\nu=1}^{d_i^{-}}([\tordiv_{i}]+d_i+\nu)\prod_{\theta\in\Sigma(1)}\prod_{\nu=0}^{d_{\theta}^{-}-1}([D_{\theta}]+d_{\theta}-\nu). \nonumber
\end{align}
This formula can be proved for any  $d,d'\in H_{2}(X,\zz)$ by  direct computation.
\end{proof}

 \begin{lem}\label{lem:well-defined,residual_morphism} Under the assumption of Theorem \ref{thm:main_theorem}.1,
  the morphism of $\gd$-modules
    \begin{align*}
 {\GKZmodsheaf^\res} &\stackrel{\sim}{\longrightarrow} (\Mir\times\Id)^*\ov{\QDM}(X,\vb) \\
1 &\longmapsto L^{\tw}(\Mir(q),z)I(q,z)\nonumber
    \end{align*}
    is well defined over $\mirrortoric$.
    \end{lem}
 
\begin{proof}[Proof of Lemma \ref{lem:well-defined,residual_morphism}] Let
 $R(q,z,z\delta_{q},z\delta_{z})\in \D$ be in the quotient
  ideal $(\G:\hatctop)$. We have to show that the cohomological valued function
  $R(q,z,z\delta_{q},z\delta_{z})z^{-c_1(\tang_X)+c_{1}(\vb)}z^{\mu}I(q,z)$ belongs to $\ker \mc$ where
  $\mc$ is the endomorphism of $H^{2*}(X)$~: $\alpha\mapsto \ctopvb\cup\alpha$.

It is enough to prove it when $R$ is a generator of the
  ideal $(\G:\hatctop)$ \ie $\hatctop R \in \G$.    From Formulas \eqref{eq:19} and (\ref{eq:14}), we
  deduce that
 \begin{align}\label{eq:11}
   &R(q,z,z\delta_{q},z\delta_{z})q^{T+d}z^{-c_1(\tang_X)+c_{1}(\vb)-d_{\tang_X}+d_{\vb}}\\&=
   R\left(q,z,z(T+d),z(-c_1(\tang_X)+c_{1}(\vb)-d_{\tang_X}+d_{\vb})\right)q^{T+d}z^{-c_1(\tang_X)+c_{1}(\vb)-d_{\tang_X}+d_{\vb}}.\nonumber
 \end{align}
 We decompose
 \begin{displaymath}
 R(q,z,z\delta_{q},z\delta_{z})=\sum_{\stackrel{d'\in
   H_{2}(X,\zz)}{\mbox{\tiny{finite}}}}q^{d'}R_{d'}(z,z\delta_{q},z\delta_{z}).
 \end{displaymath}
 From Equalities 
 \eqref{eq:10} and \eqref{eq:11}, we deduce that
\begin{align*}
  R(q,z,z\delta_{q},z\delta_{z})z^{-c_1(\tang_X)+c_{1}(\vb)}z^{\mu}I(q,z)&=\sum_{d\in H_{2}(X,\zz)}q^{d+T}z^{-c_1(\tang_X)+c_{1}(\vb)-d_{\tang_X}+d_{\vb}}B_{d}(z)
\end{align*}
where
\begin{displaymath}
  B_{d}(z):=\sum_{\stackrel{d'\in H_{2}(X,\zz)}{{\mbox{\tiny{finite}}}}}R_{d'}\left(z,z(T+d),z(-c_1(\tang_X)+c_{1}(\vb)-d_{\tang_X}+d_{\vb})\right)A_{d-d'}(1).
\end{displaymath}
To prove the lemma, it is enough to show that $\ctopvb B_{d}(z)=0$ for all
$d\in H_{2}(X,\zz)$. By $\hatctop R \in \G$ and Lemma
\ref{lem:well-defined,twisted_morphism}, we have
\begin{displaymath}
\begin{array}{lcl}
\hatctop R(q,z,z\delta_{q},z\delta_{z})z^{-c_1(\tang_X)+c_{1}(\vb)}z^{\mu}I(q,z)&=&0 \\
\sum_{d\in H_{2}(X,\zz)}q^{d+T}z^{-c_1(\tang_X)+c_{1}(\vb)-d_{\tang_X}+d_{\vb}}\left(\prod_{i=1}^{k}z\left([\tordiv_{i}]+d_i\right)\right)B_{d}(z)&=&0.
\end{array}
\end{displaymath}
As $\ctopvb B_{d}:\cc \to H^{*}(X)$ is a polynomial function in
$z$, it is enough to prove that it vanishes on $\cc^{*}$. Assume $z\in
\cc^{*}$.  As $q\in(\cc^*)^{r}$, we deduce that $q^{T}$ and
$z^{-c_1(\tang_X)+c_{1}(\vb)}$ are invertible in $H^{*}(X)$. Denote by
$I_{d}:=\{i\in\{1, \ldots ,k\}\mid d_i=0\}$ and $I^{c}_{d}$ its
complementary set.  For $i\in I_{d}^{c}$, the class
$[\tordiv_{i}]+d_i$ is invertible in $H^{*}(X)$. So we deduce
that
\begin{displaymath}
  \left(\prod_{i\in I_{d}}[\tordiv_{i}]\right)B_{d}(z)=0.
\end{displaymath}
This implies that $\ctopvb B_{d}(z)=0$ as
$\ctopvb=\prod_{i=1}^{k}[\tordiv_{i}]$.

\end{proof}
%
%\begin{proof}[Proof of Lemma ref{lem:well-defined,morphism}]From Theorem \ref{thm:mirror,symmetry,Givental} we have that $J^{\tw}(\Mir(q),z)=I(q,z)/F(q)$.  Lemma
%\ref{lem:I,satisfies,GKZ,operators} below shows that the morphism $\varphi$
%is well defined.
%Lemma \ref{lem:I,satisfies,residual,GKZ,operators} below implies that for any $R\in (\G:\hatctop)$ we have 
%\begin{displaymath}
%  \overline{Rz^{-c_1(\tang_{X})+c_{1}(\vb)}z^{\mu}I(q,z)}=0.
%\end{displaymath}
%This implies that $\varphi'$ is well defined. 
%By Proposition \ref{prop:J,expression}, we have 
%$$\overline{J^{\tw}}(\Mir(q),z)=({\overline{L}}(\Mir(q),z))^{-1}\overline{\mathbf{1}}.$$
%We deduce that 
%\begin{align*}
%  \pi\circ\varphi(P(q,z,z\delta_{q},\delta_{z}))
%&=P(q,z,z\Mir^*\nabla_{\delta_{q}},\Mir^*\nabla_{\delta_{z}})
%{\overline{L}(\Mir(q),z)\overline{J^{\tw}}(\Mir(q),z)}\\
%  &=P(q,z,z\Mir^*\nabla_{\delta_{q}},\Mir^*\nabla_{\delta_{z}})\overline{\mathbf{1}}.
%\end{align*}
%\end{proof}
%

\begin{proof}[Proof of Theorem \ref{thm:main_theorem}]
Let us first prove that $\varphi$ is an isomorphism.
 By Theorem \ref{thm:GKZ_locally,free}, $\rank
  {\GKZmodsheaf}=\rank F$, so it is enough to
  prove that the morphism $\varphi$ is surjective 
  in a neighbourhood of
$\0$.
From \cite[Proof of
   Proposition 5.5.4 p.100]{Cox-Katz-Mirror-Symmetry} we deduce that
   the "$d$" term in the definition of the $I$ function (see
   \eqref{eq:definition_of_I}) vanishes when $d\notin \mori{X}$, so that we have:
 \begin{align}\label{eq:I,function,Mori}
     I(q,z)=q^{T/z}\sum_{d\in \mori{X}}q^{d}A_{d}(z).
   \end{align}
   Then from \eqref{eq:19} we have, for any $\alpha\in H^2(X)$~:
  \begin{displaymath}
    \widehat{\alpha}I(q,z)=q^{T/z}(\alpha+O(q)).
  \end{displaymath}
  As $ H^{2*}(X)$ is generated by $ H^2(X)$, we deduce that for any
  $a\in\{0,\ldots,s-1\}$, there exists an operator $P_{a}(q,z,z\delta_{q})$
  such that
\begin{displaymath}
  P_{a}(q,z,z\delta_{q})I(q,z)F(q)^{-1}=q^{T/z}(T_{a}+O(q))
\end{displaymath}
where $F(q)$ is defined in (\ref{eq:34})~;   notice that we do not need $z\delta_{z}$ in the operator $P_{a}$.
From the definition of the function $L^{\tw}(t_{0},q,z)$ (\cf Equality
\eqref{eq:definition_of_L}), we deduce that
\begin{displaymath}
L^{\tw}(t_{0},q,z)\gamma= e^{-t_{0}/z}q^{-T/z}(\gamma+O(q)).  
\end{displaymath}
By the mirror Theorem \ref{thm:mirror,symmetry,Givental} we have that 
\begin{displaymath}
  \Mir(q)=q + O(q).
\end{displaymath}
Putting the last three arguments together, for any $a\in\{0, \ldots ,s-1\}$ we have 
\begin{displaymath}
  \varphi(P_{a}(q,z,z\delta_{q}))=L^{\tw}(\Mir(q),z)q^{T/z}(T_{a}+O(q))=T_{a}+o(1).
\end{displaymath}
This proves the surjectivity of $\varphi$ near the point $\0$. As it is an open condition, it is true in a neighbourhood of $q=0$.

Let us prove that $\varphi'$ is an isomorphism.  First, the surjectivity
of $\varphi$ implies the surjectivity of $\pi\circ \varphi$. We deduce
that $\varphi'$ is also surjective.  On $z\neq 0$, Proposition
\ref{prop:Res,locally,free} implies that the rank of
$\GKZmodsheaf^{\res}$ is less than $\rank \overline{F}$. Hence the
surjectivity implies that its rank is $\rank \overline{F}$. This also implies that $\GKZmodsheaf^{\res}$ is locally free on $\freetoric\times \C$
of rank $\dim \overline{H^{2*}(X)}_{\C}=\rank \overline{F}$.
We deduce that $\varphi'$ is an isomorphism.
\end{proof}

\begin{rem}\label{rem:missing_point-for-freeness_of_GKZres}
The last point of this proof is the missing argument to finish the proof of Theorem \ref{thm:GKZ_locally,free}.2.
\end{rem}

%  \begin{prop}
%    As a $\cc[q^{\pm},z]$-module,
%  the rank of ${\GKZmodsheaf^\res}:=D/\Res(\hatctop,\mathcal\G)$ is lower than the
%   rank of $ \overline{\QDM(X,L)}$.
%  \end{prop}

% \begin{lem}
%   \begin{enumerate}
%   \item The $\cc[q^{\pm},z]$-module
%     $\mathcal{N}_{\GKZ}:=D/\mathcal\GKZid$is free
%     outside $ z=0$ that is the $\mathcal{N}_{\GKZ}[z^{-1}]$ is locally
%     free.
%   \item The $\cc[q^{\pm},z]$-module
%     ${\GKZmodsheaf^\res}:=D/\Res(\hatctop,\mathcal\G)$ is free outside
%     $ z=0$ that is the
%     ${\GKZmodsheaf^\res}[z^{-1}]$
%     is locally free.
%   \end{enumerate}
% \end{lem}

% \begin{lem} The ring $ \mathcal{N}_{GKZ}/z \mathcal{N}_{GKZ}$ is
%   isomorphic to the Batyrev ring whose dimension as $ \cc$-vector
%   space is the dimension of $  H^{2*}(X,\cc)$.
% \end{lem}

% \begin{lem} Denote by $
%   {\GKZmodsheaf^\res}:=D/\Res(\hatctop,\mathcal\G)$.
%   The ring $ {\GKZmodsheaf^\res}/z {\GKZmodsheaf^\res}$ is
%   isomorphic to the residual Batyrev ring with respect to
%   $c:=c_{1}(L)$ whose dimension as $\cc$-vector space is the dimension
%   of $
%   \overline{ H^{2*}(X)}:= H^{2*}(X)-\dim(\ker \mc)$.
% \end{lem}

%%% Local Variables: 
%%% mode: latex
%%% TeX-master: "Mann-Mignon-Crelle-Revision"
%%% End: 

%%%%%%%%%%%%%%%%%%%%%%%%%%%%%%%%%%%%%%%%%%%%%%%%%%%%%%%%%%%%%%%%%%%%%%%

\section{Examples:  hypersurface  in $\pp^{n}$ and in $\Bl_{\pt}\pp^{n}$}

\label{sec:exampl-:-hypers}

In the following examples, we want to give explicit computations
of the quotient ideal
$(\GKZid:\hatctop)$. The first example is $\pp^{n}$
with the line bundle $\go(a)$ and the second one is the blow up of $\pp^{n}$ at one point with an
appropriate bundle (see below).
In a forthcoming paper, we will prove the following general statement 
\begin{thm}\label{thm:grobner}
Let $X$ be a smooth projective toric variety with $\lb_{1}, \ldots ,\lb_{k}$ nef
line bundles on X such that $\omega_{X}\otimes \lb_{1}^{\vee}\otimes\cdots \otimes \lb_{k}^{\vee}$
is nef.
 Put $\D':=\cc[q_a^\pm,z]\langle z\delta_{q}\rangle$ and ${\GKZid}'$ the left ideal generated by $
  \square_d$ for $ d\in H_2(X,\Z)$ (see Remark \ref{rem:alternative_definitions_of_GKZ}).
  Let $P \in \GKZid'$, we can write
  \begin{displaymath}
       P=\sum_{c\in \mathcal{P}}B_{c}\sq_{c}, \quad \deg(B_{c}\sq_{c})\leq \deg(P).
  \end{displaymath}
where the degree means the degree as differential operators in $\D'$ and $\mathcal{P}$ is the set of
primitive classes (see Notation \ref{notn:primitive_classes} and Definition \ref{defi:primitive,collection} ).
\end{thm}
%\Etienne{referee asks for a sketch of proof and algo}

\begin{rem}\label{rem:algo}
Let us explain how one could use this theorem to get an algorithm to compute the residual ideal
$\langle \GKZid' : \hatctop \rangle$ in order to get, via the isomorphism of Theorem \ref{thm:main_theorem}, a
presentation of $\QDM_{\amb}(Z)$.

\begin{enumerate}
\item First, Theorem~\ref{thm:grobner}  implies that the generators of the ideal $\GKZid'$ can be indexed by the 
primitive classes, \ie $\GKZid'=\langle \sq_{c}, \quad c\in\mathcal{P}\rangle$. 
\item As the line bundle $\lb_{i}$ are ample, for any $d \in \mori{X}$, we see that the operator $\sq_{d}$
is of the form $P_1 -\hatctop q^{d} P_{2}$ where $P_{1},P_{2}$ are two operators in $\mathbb{D}'$
(see \eqref{eq:Rd,xctop} for a similar statement in the commutative case). 
Let $c_{1},c_{2}$ be two primitive classes. Using the same ideas that $S$-polynomials for Groebner
basis (in the commutative case), we can find three operators $T,U,V \in \mathbb{D'}$ such that
\begin{displaymath}
U\sq_{c_{1}}-V\sq_{c_{2}}=\hatctop T_{c_{1},c_{2}}
\end{displaymath}
This means that for each pair of primitive class $c_{1},c_{2}$, we get an operator $T_{c_{1},c_{2}}$
in the residual ideal  $( \GKZid' : \hatctop )$.
\item We think that  the residual ideal is generated by the $\sq_{c}$ for $c\in \mathcal{P}$ and by
  $T_{c_{1},c_{2}}$ for $c_{1},c_{2}\in \mathcal{P}$. At this point, we do not have a complete proof of this
  statement. We hope that an induction, like in Proposition \ref{prop:quot,ideal,blowup}, could work.
\end{enumerate}
\end{rem}

\begin{proof}[Ideas of proof of Theorem \ref{thm:grobner}]
  We only give some ideas for a proof because it is quite long and technical.\\
  \textbf{Calabi-Yau case \ie
    $\omega_{X}\otimes \lb_{1}^{\vee}\otimes\cdots \otimes \lb_{k}^{\vee}=\mathcal{O}_{X}$:}The
  theorem follows immediately from the homogeneity of the operator $\square_{d}$ for any
  $d \in H_{2}(X,\mathbb{Z})$.
\\
\textbf{ Non Calabi-Yau case:} This case is more difficult.
Let's recall some notations of \eqref{eq:other,presentation,GKZ,mod}.
We use the following isomorphism \begin{align*}%\label{eq:other,presentation,GKZ,mod}
f~:~\D''/\langle \sq''_d, \gz''_u\rangle &\lra \D'/\GKZid'\simeq \GKZmod\\ 
z\de_ {\la_\rho}&\longmapsto  \sum_{a=1}^r D_\rho^a  z\de_{q_a} \nonumber
\end{align*}
where $\D'':=\cc[q_a^\pm,z]\langle z\lambda{q}_{\rho},\rho\in \Delta(1)\rangle$ and 
\begin{align*}
\sq''_d&:=\hspace{-0.1cm}\prod_{\rho\in\raysfiber}\hspace{-0.1cm}[-z\de_{\la_\rho}+zd_\rho^+]_{d_\rho^+}\hspace{-0.1cm}\prod_{\rho\in\raysbase}\hspace{-0.1cm}[z\de_{\la_\rho}]_{d_\rho^+}
 -q^d\hspace{-0.1cm}\prod_{\rho\in\raysfiber}\hspace{-0.1cm}[-z\de_{\la_\rho}+zd_\rho^-]_{d_\rho^-}\hspace{-0.1cm}\prod_{\rho\in\raysbase}\hspace{-0.1cm}[z\de_{\la_\rho}]_{d_\rho^-},\\
 \gz''_u&:=\sum_\rho \langle u,v_\rho\rangle z\de_{\la_\rho}, u\in M'.
\end{align*}
Using a suitable monomial order, we can prove a first result.
Let  $P \in \langle \sq''_{d} \rangle \subset \D''$, we can write
  \begin{displaymath}
       P=\sum_{c\in \mathcal{P}}B_{c}\sq''_{c}, \quad \deg(B_{c}\sq''_{c})\leq \deg(P).
  \end{displaymath}
Then, we have to incorporate the $\gz''_u$ operators into the picture. This is the tricky part.   We
consider the ideals generated by the symbols which is a monomial ideals in a commutative ring. We use the Taylor's complex (see
\cite{Lyubeznik-Taylor-resolution-monomial-ideals-1988}) which plays the role of the Koszul
resolution for monomial ideals.
Then we pass to the ideal $\langle \sq''_{d} , \gz_{u}\rangle$ and use the isomorphism $f$ to conclude.  
\end{proof}

%%%%%%%%%%%%%%%%%%%%%%%%%%%%%%%%%%%%%%%%%%%%%%%%%%%%%%%%%%%%%%%%%%%%%%%%%%%%%%%%%%%%%%%%%%%%%%%%%
%===============================================================================================%
\subsection{The projective space $X=\pp^n$ and the invertible sheaf $\mathcal{L}=\go(a)$.}
\label{sec:proj-space-x=ppn}
%===============================================================================================%
%%%%%%%%%%%%%%%%%%%%%%%%%%%%%%%%%%%%%%%%%%%%%%%%%%%%%%%%%%%%%%%%%%%%%%%%%%%%%%%%%%%%%%%%%%%%%%%%%

We have $H^2(\pp^n,\Z)\cong\Z$ and $H_2(\pp^n,\Z)\cong\Z$.  Denote by
$h$ the homology class of a line in $\pp^n$, and by $H$ the Chern
class of $\go(1)$. They both generates their respective group, and we
have $\int_h H=1$.  The nef cone in $H^2(X,\Z)$ is $\N.H$ and its
dual, the Mori cone in $H_2(X,\Z)$, is equal to $\N.h$.  The 
ring  $\La$ is $\C[Q^d, d\in\mori{X}]\simeq \C[q] $ where we set
$q:=Q^h$. The ring $\Pi$ is $\Pi=\C[Q^d, d\in H_2(X,\Z)]\simeq
\C[q^{\pm}]:= \C[q,q^{-1}].$ We put $\lb=\go(a)$ for $a\in\Z$. The
sheaf $\go(a)$ is ample if and only if $a>0$.  The sheaf
$\omega_{\pp^n}^\vee\otimes \go (a)^\vee=\go (n+1-a)$ is nef if and
only if $n+1-a\geq 0$. We have $0 < a \leq n+1.$ The different cases
are~:
$$
\begin{array}{ll}
\text{\emph {Calabi-Yau} } & a=n+1.\\
\text{\emph {Fano} } & 1\leq a\leq n, \text{where $(\omega_X\otimes \lb)^\vee=\go(n+1-a)$ is ample}.
\end{array}
$$
We make use of the notations of Subsection
\ref{subsection:Fan}. Let us choose a fan for
$\pp^{n}$~: Denote by $N$ the lattice $\Z^n$ and by $(e_1,\ldots,e_n)$
its canonical basis. Put $ w_1:=e_1, \ldots w_n:=e_n, \
w_{n+1}:=-e_1-\cdots -e_n$.  These are the lattice generators of the
rays $\theta_i$, where $\theta_i=\R^+ w_i$ for any $i\in \{1,\ldots, n+1\}$.
We set $\Sigma(1):=\{\theta_1,\ldots,\theta_{n+1}\}$.  The set of maximal
cones is 
$$
\Sigma(n)=\{\mbox{every (necessarily convex) cone generated by $n$ vectors in $\Sigma(1)$}\}.
$$
Denote by $D_{\theta}$ the toric divisor associated to the ray $\theta \in
\Sigma(1)$. We have $[D_{\theta}]=H$ in $H^{2*}(X)$.  There is only one
primitive collection (see \S.\ref{subsection:primitive,collections})
$P=\{\theta_1,\ldots,\theta_{n+1}\}=\Sigma(1) $. The primitive class is
$\gp=\{h\}$.

Let us compute the quotient ideal.
We will use the alternative definition \ref{rem:alternative_definitions_of_GKZ}.1 of the
GKZ module, that is $\GKZmod=\D'/\GKZid'$, with $\D'=\C[q^\pm]\langle z\de_q\rangle$ and $\GKZid'=\langle \sq_h\rangle$.
 We have $c_\top =
c_1(\lb)=c_1(\go(a))=aH$ and $\hatctop=az\de_q.$

\begin{prop}
We have:
$$
(\GKZid':\hatctop)=\<\sq_h,\frac 1 a
(z\de_{q})^{n}-q(az\de_{q}+z).\ldots.(az\de_{q}+(a-1)z)\>.
 $$
  
\end{prop}
\begin{proof}
  The operator $ P_0=\frac 1 a
  (z\de_{q})^{n}-q(az\de_{q}+2z).\ldots.(az\de_{q}+az)$ is in
  $(\GKZid':\hatctop)$~: since $z\de_q.q=q(z\de_q+z)$, we have
  $\hatctop.P = a\sq_{h}$.
We prove now, by induction on the degree, that any operator $P$ in
$(\GKZid':\hatctop)$ is in $\<P_0\>$. First notice that
$a.y\sigma(P_0)=\sigma(\sq_h)$, even in the Calabi-Yau case where $\sigma$ is the symbol.
Let $P$ be in $(\GKZid':\hatctop)$.

If $\deg P=0$ (and $P\neq 0$). We have, $az\de_q.P=Q.\sq_h$, where $Q\in \D$ 
and $\deg (a z\de_q.P)=1$, $\deg Q.\sq_h=\deg Q+\deg \sq_h=\deg Q+(n+1)$ (recall that $a\leq n+1$).
It follows that $n=0$, which is impossible.

Assume it is true for $\deg P=l$. If $\deg P=l+1$, we still have $a
z\de_q.P=Q.\sq_h$.  Passing to the symbol we get~:
$ay\sigma(P)=\sigma(Q).\sigma(\sq_h)=ay \sigma(Q).\sigma(P_0)$.  It
follows that the polynomials $P$ and $QP_0$ both are in
$\Quot(\hatctop, \GKZid')$ and have the same symbol. Hence,
$P-QP_0$ is in $\Quot(\hatctop, \GKZid')$ and has degree strictly less
than $l$.  By induction, $P-QP_0\in\< P_0\>$ and
$P \in\<P_0\>$.
\end{proof}

\subsection{The blown-up plane $X=\Bl_{\pt}\pp^n$ and the sheaf
  $\mathcal{L}=\go(aH+bE)$.}\label{sec:blown-up-plane-}
%%%%%%%%%%%%%%%%%%%%%%%%%%%%%%%%%%%%%%%%%%%%%%%%%%%%%%%%%%%%%%%%%%%%%%%%%%%%%%%%%%%%%%%%%%%%%%%%%
Denote by $N=\mathbb{Z}^{n}$ the lattice and by $(e_{1}, \ldots ,e_{n})$ the canonical basis of $N$.
The fan $\Sigma$ of $X$ is given by the rays
\begin{displaymath}
v_{0}=-e_{n},\quad  \forall i \in \{1, \ldots ,n\},\ v_{i}=e_{i},\  v_{n+1}=(-1, \ldots ,-1). 
\end{displaymath}
The maximal cone in $N\otimes \mathbb{R}$ are
\begin{align*}
  \forall i\in\{1, \ldots ,n+1\}\setminus\{n\},\  \sigma_{i}=\sum^{n+1}_{\stackrel{j=1}{j\neq
      i}}\mathbb{R}^{+}v_{j}, \mbox{ and }  
\sigma_{n,i}=\mathbb{R}^{+}v_{0}+\sum_{\stackrel{j=1}{j\neq i}}^{n-1}\mathbb{R}^{+}v_{j}
\end{align*}
We have $H^2(X,\Z)\cong\Z^2$ and $H_2(X,\Z)\cong\Z^2$.  Let $E$ be the exceptional divisor, and $H$
the strict transform by the blown-up of an hyperplane of $\pp^n$ which does not meet the blown-up
point.  We also denote by $E$ and $H$ their Chern classes. Denote by $e$ the homology class of $E$
and $h$ the homology class of $H$.  We choose the following bases which are dual to each others~:
\begin{itemize}
	\item Base of $H^2(X,\Z)$: $(T_1=H-E,T_2=H)$.
	\item Base of $H_2(X,\Z)$: $(B_1=e,B_2=h-e)$.
\end{itemize}
Notice that $c_1(\omega_X)=(n+1)H-(n-1)E$.
We denote by $D_\theta$ the toric divisor
associated to $\theta \in \Sigma(1)$ and $[D_\theta]$ its class in $H^2(X,\Z)$. We have for
$i\in\{1, \ldots ,n+1\}\setminus\{n\},[D_i]=H-E,\ [D_n]=H,\ [D_0]=E.  $ There are two primitive
collections, $P_1=\{\theta_0,\theta_n\}$ and $P_2=\{\theta_i, i\notin\{ 0,n\}\}$.  The primitive
classes are $\gp = \{e,h-e\}.$
%-----------------------------------------------------------------------------------------------%
%subsubsection{ Nef and Mori cone.}
%-----------------------------------------------------------------------------------------------%
The nef cone in $H^2(X,\Z)$ is $\R^+H+\R^+(H-E)$, an its dual, the Mori cone in $H_2(X,\Z)$  is equal to $\R^+e+\R^+(h-e)$ (see Figure \ref{fig:nef_mori_blp2}).
\begin{figure}[!h]
  \centering
  \begin{tikzpicture}[scale=0.8]
    \fill[color=gray!20] (0,0) -- (3,0) -- (3,-3) -- cycle;
    \foreach \k in {-2,-2,-1,0,1,2,3}{\draw [dotted, very thin](\k,-3)
      -- (\k,3);\draw [dotted, very thin](-2,\k) -- (3,\k);} 
      \draw  [dotted, very thin](-2,-3) -- (3,-3); 
    \draw [very    thick,>=latex,->] (0,0) -- (1,0) node [above,near end] {$\ \ \ T_2=H$}; 
    \draw [thin, gray] (0,0) -- (3,-3) ; 
    \draw [very    thick,>=latex,->] (0,0) -- (1,-1) node [below left,near    end] {$T_1=H-E$}; 
    \draw [thin,gray ] (0,0) -- (3,0) ; 
    \draw [>=latex,->,gray] (0,0) -- (0,1) node [above left ,near end]{$E$}; 
    \draw [black] (0,0) node {$\bullet$}; \draw (0.5,-3.7)
    node {The nef cone in $H^2(\Bl_{p}\pp^2,\Z)$.};
    %\draw (0.5,-4.2)
    %node {\emph { (base $(H,E)$, dual of $(h,-e)$)}};
    \draw (5,0.4) node {\scriptsize de Rham duality}; \draw (5,0) node
    {$\longleftarrow \hspace{-0.3cm} \longrightarrow$ };
\end{tikzpicture}
\ \ 
 \begin{tikzpicture}[scale=0.8]
 \fill[color=gray!20] (0,0) -- (0,3)-- (3,3) -- (3,-3) -- cycle;
 \foreach \k in {-2,-2,-1,0,1,2,3}{\draw [dotted, very thin](\k,-3) -- (\k,3);\draw [dotted, very thin](-2,\k) -- (3,\k);}
\draw [dotted, very thin](-2,-3) -- (3,-3);
\draw [>=latex,->,gray] (0,0) -- (1,0)  node [above,near end] {$h$};
\draw [very thick,>=latex,->,black] (0,0) -- (0,1)  node [left ,near end] {$e$};
\draw [thin, black] (0,0) -- (0,3) ;
\draw [very thick,>=latex,->,black] (0,0) -- (1,-1) node [left,near end] {$h-e$};
\draw [thin,black ] (0,0) -- (3,-3) ;
\draw [black] (0,0) node  {$\bullet$};
\draw (0.5,-3.7) node {The Mori cone in $H_2(Bl_{\pt}\pp^2,\Z)$.};
%\draw (0.5,-4.2) node ;%{\emph {(base $(h,-e)$, dual of $(H,E)$)}};
\end{tikzpicture}
 \caption{Nef and Mori cone of $\Bl_p\pp^2$.}
 \label {fig:nef_mori_blp2}
\end{figure}
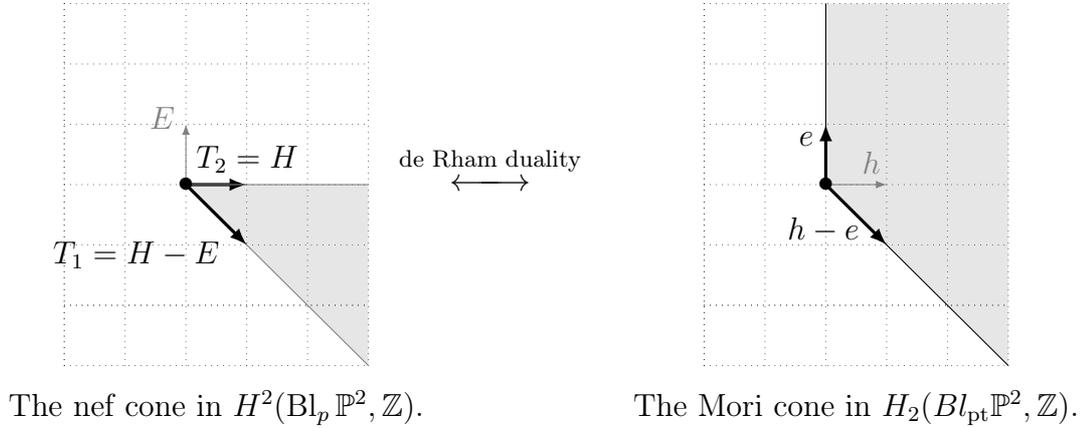
Following the choice of our base, we put $q_1:=Q^e$, $q_2:=Q^{h-e}$.
The ring $\La$ is $\C[q_{1},q_{2}]$.  We want $\go
(aH+bE)$ to be ample and $\omega_{Bl_p\pp^n}^\vee\otimes \go
(aH+bE)^\vee=\go ((n+1-a)H-(n-1+b)E)$ to be nef.  This leads to the following cases 
\begin{displaymath}
  \{(a,b)\in \mathbb{Z}^{2} \mid b\in \{-1, \ldots ,1-n\}, \ a+b \in \{1,2\}\}
\end{displaymath}
The Calabi-Yau case is $(a,b)=(n+1,1-n)$. 
We have
\begin{displaymath}
  c_{\topp}=-bT_{1}+(a+b)T_{2},\quad\widehat{c}_{\topp}=-b (z\delta_{q_{1}})+(a+b)(z\delta_{q_{2}})
\end{displaymath}

In the differential ring $\D':=\C[z,q_1^\pm,q_2^\pm]\<
z\de_{q_1},z\de_{q_2}\>$ we consider the GKZ ideal $\GKZid'=\<\sq_e,\sq_{h-e}\>$,
where
\begin{align}\label{eq:generator,GKZ,blowup}
  \left\{ 
\begin{array}{cll}
\sq_e &= (z\de_{q_1})^n -q_{1}(z\de_{q_2}-z\de_{q_1})\prod_{\nu=1}^{-b}(
-bz\de_{q_1}+(a+b)z\de_{q_2}+\nu z),\\
\sq_{h-e} &= (z\de_{q_2})(z\de_{q_2}-z\de_{q_1}) -q_2
\prod_{\nu=1}^{a+b}(-bz\de_{q_1}+(a+b)z\de_{q_2}+ \nu z).
\end{array}
 \right\}
\end{align}
The following proposition gives the generator of the quotient ideal.
\begin{prop}\label{prop:quot,ideal,blowup}
The quotient ideal of the GKZ ideal by $\hatctop$ is:
$$(\GKZid':\hatctop)=\<P_0,\sq_e,\sq_{h-e}\>$$
where
\begin{align*}
  P_0:=&-a(z\de_{q_1})^{n-1}+(a+b)(z\delta_{q_{1}})^{n-2}(z\de_{q_2}) + \\
&-ab
q_{1}(z\de_{q_2}-z\delta_{q_{1}})\prod_{\nu=1}^{-b-1}(-b(z\delta_{q_{1}})+(a+b)(z\delta_{q_{2}})+\nu
z)\\
&-(a+b)^{2}q_{2}(z\delta_{q_{1}})^{n-2}\prod_{\nu=1}^{a+b-1}(
-b(z\de_{q_1})+(a+b)(z\de_{q_2})+\nu z)
\end{align*}
\end{prop}

\begin{proof}[Proof of Proposition~\ref{prop:quot,ideal,blowup}]
First, we have $\widehat{c}_{\top}P_{0}\in \GKZid'$ as 
\begin{align}\label{eq:30}
  ab\sq_{e}+(a+b)^{2}(z\delta_{q_{1}})^{n-2}\sq_{h-e}=\widehat{c}_{\top}P_{0}  
\end{align}

Let us prove by induction on the degree of the operator $P\in \D'$ that
if $\widehat{c}_{\top}P \in \GKZid'$ then
$P\in\langle P_{0},\sq_{e},\sq_{h-e} \rangle$.
From Theorem \ref{thm:grobner}, we have
\begin{align}\label{eq:18}
 \widehat{c}_{\topp}P=R_{1}\sq_{e}+R_{2}\sq_{h-e}
\end{align}
where the degree of the operators $R_{1}\sq_{e}$ and $R_{2}\sq_{h-e}$
are less or equal to $\deg(P)+1$.

 Taking the symbol of \ref{eq:18} we get 
 \begin{displaymath}
   \sigma(\widehat{c}_{\topp})\sigma(P)=S_{1}\sigma(\sq_{e})+S_{2}\sigma(\sq_{h-e})
 \end{displaymath}
where $S_{i}$ are either the symbol of $R_{i}$ or $0$. Replacing $\sigma(\sq_{e})$ by  Equality (\ref{eq:30}), we get
 the following equality in $\mathbb{Q}[q_{1},q_{2},y_{1},y_{2}]$
\begin{align}\label{eq:37}
   \sigma(\widehat{c}_{\topp})\sigma(P)=\sigma(\widehat{c}_{\topp})\frac{S_{1}\sigma(P_{0})}{ab} +\sigma(\sq_{h-e})\left(-\frac{(a+b)^{2}}{ab}S_{1}y_{1}^{n-2}+S_{2}\right)
 \end{align}

From \eqref{eq:generator,GKZ,blowup}, we have
  \begin{align}
    % \sq_e &= (z\de_{q_1})^n
    % -\widehat{c}_{\top}q_{1}(z\de_{q_2}-z\de_{q_1})\prod_{\nu=1}^{-b-1}(-bz\delta_{q_{1}}+(a+b)z\delta_{q_{2}}+\nu
    % z),\label{eq:28}\\
\sq_{h-e} &= (z\de_{q_2})(z\de_{q_2}-z\de_{q_1}) -\widehat{c}_{\top}q_2
\prod_{\nu=1}^{a+b-1}(-bz\de_{q_1}+(a+b)z\de_{q_2}+ \nu z).\label{eq:29}
  \end{align}
 In $\mathbb{Q}[q_{1},q_{2},y_{1},y_{2}]/\sigma(\widehat{c}_{\topp})$,  we get from
(\ref{eq:29})
\begin{displaymath}
  0=\overline{y}_{2}(\overline{y}_{2}-\overline{y}_{1})\left(-\frac{(a+b)^{2}}{ab}\overline{S}_{1}\overline{y}_{1}^{n-2}+\overline{S}_{2}\right)
\end{displaymath}
As $\overline{y}_{2}(\overline{y}_{2}-\overline{y}_{1})\neq 0$, there exists
$Q\in\mathbb{Q}[q_{1}q_{2},y_{1},y_{2}]$ such that
\begin{displaymath}
-\frac{(a+b)^{2}}{ab}S_{1}y_{1}^{n-2}+S_{2}=Q\sigma(\widehat{c}_{\topp})
\end{displaymath}
By the degree conditions on $R_{1}\sq_{e}$ and $R_{2}\sq_{h-e}$, we have that $\deg Q= \deg P -2$. 
Putting this in (\ref{eq:37}), we get
\begin{align}\label{eq:38}
   \sigma(\widehat{c}_{\topp})\sigma(P)=\sigma(\widehat{c}_{\topp})\left(\frac{S_{1}\sigma(P_{0})}{ab} +\sigma(\sq_{h-e})Q\right)
 \end{align}
As $\deg S_{1}\sigma(P_{0})=\deg \sigma(\sq_{h-e})Q=\deg P$, we have
\begin{align*}
\sigma(P)%&=-\frac{S_1}{a}\sigma(P_0)-\frac{Q}{a}\sigma(\sq_{h-e})\\
&=\sigma\left(-\frac{R_1}{a}P_0-\frac{\widehat Q}{a}\sq_{h-e}\right)
\end{align*}
where $\widehat Q$ is any operator having symbol $Q$.
The operator
$$
P':=P-\left( -\frac{R_1}{a}P_0-\frac{\widehat Q}{a}\sq_{h-e}\right)
$$
satisfies that $\widehat c_{\top} P'\in\GKZid'$ and has degree strictly
inferior to the degree of $P$.  By induction, we deduce that
$P'\in\<P_0,\sq_{e},\sq_{h-e}\>$ that is $P$ is in $\< P_0,\sq_{e},\sq_{h-e}\>$.
\end{proof}

 \appendix
 %%%%%%%%%%%%%%%%%%%%%%%%%%%%%%%%%%%%%%%%%%%%%%%%%%%%%%%%%%%%%%%%
\section{Twisted Axioms for Gromov-Witten invariants}
 \label{sec:twisted-axiom-for-GW}

\renewcommand\theenumi{\thesection.\arabic{enumi}}

 In this Appendix, we will state (without proof) the twisted axioms
 for twisted Gromov-Witten invariants. For the ``untwisted'' axioms,
 we refer to two papers of Behrend and Manin
 (\cite{Behrend-Manin-stack-stable-mapGWI-1996} and
 \cite{Behrend-GW-in-alg-geo-1997}). As explained in
 \S.\ref{subsubsec:twisted_quantum_product}, the twisted axioms are the
 non-equivariant limit of the equivariant twisted axioms.  %Some of the
% twisted axioms are in \cite{Pand-after-Givental}.
% One should also mention the indirect proof
%  given by Tseng \cite{tseng_orbifold_2010} where the Corollary 4.2.3
%  implies the twisted axioms even though there are not stated there.
%  This appendix is due to lack of references on twisted Gromov-Witten
%  invariants.  Its aim is to fill a gap concerning results well known by
%  experts.

 Recall from Notation \ref{notn:base_of_cohomology} and $T_{0}, \ldots
 ,T_{s-1}$ be a basis of $ H^{2*}(X)$. We denote by $T^{a}$ the
 Poincar\'e dual of $T_{a}$ for $ a\in\{0, \ldots ,s-1\}$.  Let $d$ be
 in $H_{2}(X,\zz)$.  Let $\gamma_{1}, \ldots ,\gamma_{\ell}$ be in
 $H^{2*}(X)$, $m_{1}, \ldots ,m_{\ell}$ be in $\nn$, for any
 $\sigma\in S_{\ell}$ and $j$ be in $\{1, \ldots ,\ell\}$.
\begin{enumerate}
  \item (Twisted  $S_{\ell}$-invariance)\label{item:twisted,symmetry} 
   \begin{align*}
 &\left\langle \widetilde{\tau_{m_{1}}(c_{1}({\vb})\cup\gamma_{1})}, \ldots
   ,\tau_{m_{\ell}}(\gamma_{\ell})\right\rangle_{0,\ell,d}\\
 &=\left\langle \tau_{m_{\sigma(1)}}(\gamma_{\sigma(1)}), \ldots
   ,\widetilde{\tau_{m_{\sigma(j)}}(c_{1}({\vb})\cup\gamma_{\sigma(j)})}, \ldots ,\tau_{m_{\sigma(\ell)}}(\gamma_{\sigma(\ell)})\right\rangle_{0,\ell,d}
 \end{align*}
\item(Twisted Fundamental class equation / string equation
  )\label{item:tw-SE}
\begin{align*}
 &  \left\langle \tau_{m_{1}}(\gamma_{1}), \ldots
     ,\widetilde{\tau_{m_{k}}(\gamma_{k})}, \ldots
     ,\tau_{m_{\ell}}(\gamma_{\ell}),\mathbf{1}\right\rangle_{0,\ell+1,d}\\&=
   \sum_{i\mid m_{i}>0}\left\langle \tau_{m_{1}}(\gamma_{1}), \ldots ,
     \tau_{m_{i}-1}(\gamma_{i}), \ldots ,\widetilde{\tau_{m_{k}}(\gamma_{k})},
     \ldots ,\tau_{m_{\ell}}(\gamma_{\ell})\right\rangle_{0,\ell,d}
 \end{align*}
\item (Consequence of the two above)\begin{align*}
 &  \left\langle \tau_{m_{1}}(\gamma_{1}), \ldots
     ,\widetilde{\tau_{m_{k}}(\gamma_{k})}, \ldots
     ,\tau_{m_{\ell}}(\gamma_{\ell}),\mathbf{1}\right\rangle_{0,\ell+1,d}\\&=
   \sum_{i\mid m_{i}>0}\left\langle \tau_{m_{1}}(\gamma_{1}), \ldots ,
     \tau_{m_{i}-1}(\gamma_{i}), \ldots ,\widetilde{\tau_{m_{k}}(\gamma_{k})},
     \ldots ,\tau_{m_{\ell}}(\gamma_{\ell})\right\rangle_{0,\ell,d}
 \end{align*}
\item (Twisted Divisor axiom)\label{item:tw-diviseur-axiom}\begin{align*}
     &\left\langle \tau_{m_{1}}(\gamma_{1}), \ldots
       ,\widetilde{\tau_{m_{k}}(\gamma_{k})}, \ldots ,\tau_{m_{\ell}}(\gamma_{\ell}),\gamma\right\rangle_{0,\ell+1,d}\\
     &=\left(\int_{d}\gamma\right)\langle\tau_{m_{1}}(\gamma_{1}), \ldots
     ,\widetilde{\tau_{m_{k}}(\gamma_{k})}, \ldots ,\tau_{m_{\ell}}(\gamma_{\ell})
     \rangle_{0,\ell,d}\\
 &+\sum_{i:m_{i}>0}\left\langle \tau_{m_{1}}(\gamma_{1}),
       \ldots ,
       \widetilde{\tau_{m_{i}-1}(\gamma\cup \gamma_{i})}, \ldots ,\tau_{m_{\ell}}(\gamma_{\ell})\right\rangle_{0,\ell,d}\\
 \end{align*}
\item (Twisted Dilaton equation)\label{item:twist-Dilaton}
\begin{align*}
  &   \langle \tau_{m_{1}}(\gamma_{1}), \ldots
     ,\widetilde{\tau_{m_{j}}(\gamma_{j})}, \ldots
     ,\tau_{m_{\ell}}(\gamma_{\ell}),\tau_{1}(\mathbf{1})\rangle_{0,\ell+1,d}\\
 &=(-2+n)
     \langle \tau_{m_{1}}(\gamma_{1}), \ldots
     ,\widetilde{\tau_{m_{j}}(\gamma_{j})}, \ldots
     ,\tau_{m_{\ell}}(\gamma_{\ell})\rangle_{0,\ell,d}
   \end{align*}
 \item (Twisted TRR \ie Topological Recursion Relation)\label{item:TRR}
\begin{align*}
   &  \lcor
       \tau_{m_{1}+1}(\gamma_{1}),\tau_{m_{2}}(\gamma_{2}),\widetilde{\tau_{m_{3}}(\gamma_{3})}\rcor_{0}=\sum_{a=0}^{s-1}\lcor \tau_{m_{2}}(\gamma_{2}),
 \widetilde{\tau_{m_{3}}(\gamma_{3})},T^{a}\rcor_{0} 
 \lcor \tau_{m_{1}}(\gamma_{1}),\widetilde{T}_{a}\rcor_{0} \\
   &  \lcor
       \widetilde{\tau_{m_{1}+1}(\gamma_{1})},\tau_{m_{2}}(\gamma_{2}),
 {\tau_{m_{3}}(\gamma_{3})}\rcor_{0}
 =\sum_{a=0}^{s-1}\lcor \tau_{m_{2}}(\gamma_{2}),
 {\tau_{m_{3}}(\gamma_{3})},\widetilde{T}^{a}\rcor_{0} 
 \lcor \widetilde{\tau_{m_{1}}(\gamma_{1})},{T_{a}}\rcor_{0}
 \end{align*} using the notation
 \begin{align}\label{eq:notation,corrolators}
    \lcor \tau_{m_{1}}(\gamma_{1}), \ldots , \tau_{m_{\ell}}(\gamma_{\ell})\rcor_{0}:=
    \sum_{\ell\geq 0}\sum_{d\in H_{2}(X,\zz)}\frac{1}{\ell!}\left\langle
      \tau_{m_{1}}(\gamma_{1}), \ldots , \tau_{m_{\ell}}(\gamma_{\ell}),\tau, \ldots
      ,\tau\right\rangle_{0,\ell+n,d}
  \end{align}
\item (Twisted WDVV equations)\label{item:tw-WDVV}
\begin{align*}%\label{eq:WDVV2}
&
   \sum_{a=0}^{s-1} \lcor
   \tau_{m_{1}}(\gamma_{1}),{\tau_{m_{2}}(\gamma_{2})},\widetilde{T}_{a}\rcor_{0}
   \lcor
   \tau_{m_{3}}(\gamma_{3}),\widetilde{\tau_{m_{4}}(\gamma_{4})},{T^{a}}\rcor_{0} \\ \nonumber
   &=\sum_{a=0}^{s-1} \lcor
   \tau_{m_{1}}(\gamma_{1}),{\tau_{m_{3}}(\gamma_{3})},\widetilde{T}_{a}\rcor_{0}
   \lcor
   {\tau_{m_{2}}(\gamma_{2})},\widetilde{\tau_{m_{4}}(\gamma_{4})},{T^{a}}\rcor_{0}
 \end{align*}
   \begin{align*}
 &  \sum_{a=0}^{s-1} \lcor
   \tau_{m_{1}}(\gamma_{1}),{\tau_{m_{2}}(\gamma_{2})},{T}_{a}\rcor_{0}
   \lcor
   \tau_{m_{3}}(\gamma_{3}),\widetilde{\tau_{m_{4}}(\gamma_{4})},\widetilde{T}^{a}\rcor_{0} \\ \nonumber
   &=\sum_{a=0}^{s-1} \lcor
   \tau_{m_{1}}(\gamma_{1}),{\tau_{m_{3}}(\gamma_{3})},{T}_{a}\rcor_{0}
   \lcor
   {\tau_{m_{2}}(\gamma_{2})},\widetilde{\tau_{m_{4}}(\gamma_{4})},\widetilde{T}^{a}\rcor_{0}
 \end{align*}

\end{enumerate}

\bibliographystyle{amsalpha}
 \bibliography{biblio}

\def\cprime{$'$}
\providecommand{\bysame}{\leavevmode\hbox to3em{\hrulefill}\thinspace}
\providecommand{\MR}{\relax\ifhmode\unskip\space\fi MR }
\providecommand{\MRhref}[2]{%
  \href{http://www.ams.org/mathscinet-getitem?mr=#1}{#2}
}
\providecommand{\href}[2]{#2}
\begin{thebibliography}{CLCT09}

\bibitem[Ado94]{adolphson_hypergeometric}
Alan Adolphson, \emph{Hypergeometric functions and rings generated by
  monomials}, Duke Math. J. \textbf{73} (1994), no.~2, 269--290.

\bibitem[Bar00]{Bms}
Serguei Barannikov, \emph{{Semi-infinite Hodge structures and mirror symmetry
  for projective spaces}}, Math.AG/0010157 (2000), 17.

\bibitem[Bat93]{batyrev-quantum-1993}
Victor~V. Batyrev, \emph{Quantum cohomology rings of toric manifolds},
  Ast\'erisque (1993), no.~218, 9--34, Journ{\'e}es de G{\'e}om{\'e}trie
  Alg{\'e}brique d'Orsay (Orsay, 1992).

\bibitem[Bat94]{Batyrev-Dual-polyhedra-1994}
\bysame, \emph{Dual polyhedra and mirror symmetry for {C}alabi-{Y}au
  hypersurfaces in toric varieties}, J. Algebraic Geom. \textbf{3} (1994),
  no.~3, 493--535. \MR{1269718 (95c:14046)}

\bibitem[Beh97]{Behrend-GW-in-alg-geo-1997}
K.~Behrend, \emph{Gromov-{W}itten invariants in algebraic geometry}, Invent.
  Math. \textbf{127} (1997), no.~3, 601--617.

\bibitem[BF97]{Fatenchi-Behrend-intrinsic-normal-cone}
K.~Behrend and B.~Fantechi, \emph{The intrinsic normal cone}, Invent. Math.
  \textbf{128} (1997), no.~1, 45--88.

\bibitem[BH93]{Bruns-Herzog-CM}
Winfried Bruns and J{\"u}rgen Herzog, \emph{Cohen-{M}acaulay rings}, Cambridge
  Studies in Advanced Mathematics, vol.~39, Cambridge University Press,
  Cambridge, 1993.

\bibitem[BM96]{Behrend-Manin-stack-stable-mapGWI-1996}
K.~Behrend and Yu. Manin, \emph{Stacks of stable maps and {G}romov-{W}itten
  invariants}, Duke Math. J. \textbf{85} (1996), no.~1, 1--60.

\bibitem[CG07]{Givental-Coates-2007-QRR}
Tom Coates and Alexander Givental, \emph{Quantum {R}iemann-{R}och, {L}efschetz
  and {S}erre}, Ann. of Math. (2) \textbf{165} (2007), no.~1, 15--53.

\bibitem[CK99]{Cox-Katz-Mirror-Symmetry}
David~A. Cox and Sheldon Katz, \emph{Mirror symmetry and algebraic geometry},
  Mathematical Surveys and Monographs, vol.~68, American Mathematical Society,
  Providence, RI, 1999.

\bibitem[CLCT09]{CCLTsqcwps}
Tom Coates, Yuan-Pin Lee, Alessio Corti, and Hsian-Hua Tseng, \emph{The quantum
  orbifold cohomology of weighted projective spaces}, Acta Math. \textbf{202}
  (2009), no.~2, 139--193.

\bibitem[CLS11]{2011-Cox-Little-Schenck}
David~A. Cox, John~B. Little, and Henry~K. Schenck, \emph{Toric varieties},
  xxiv+841. \MR{2810322}

\bibitem[CvR09]{cox-primitive-2008}
David~A. Cox and Christine von Renesse, \emph{Primitive collections and toric
  varieties}, Tohoku Math. J. (2) \textbf{61} (2009), no.~3, 309--332.
  \MR{2568257}

\bibitem[dCM02]{Cataldo-Migliorini-2002-LEF}
Mark Andrea~A. de~Cataldo and Luca Migliorini, \emph{The hard {L}efschetz
  theorem and the topology of semismall maps}, Ann. Sci. \'Ecole Norm. Sup. (4)
  \textbf{35} (2002), no.~5, 759--772.

\bibitem[FP97]{Fulton-Pandharipande-1997-Notes-stable-maps}
W.~Fulton and R.~Pandharipande, \emph{Notes on stable maps and quantum
  cohomology}, Algebraic geometry---{S}anta {C}ruz 1995, Proc. Sympos. Pure
  Math., vol.~62, Amer. Math. Soc., Providence, RI, 1997, pp.~45--96.

\bibitem[Ful93]{Fulton-toric}
William Fulton, \emph{Introduction to toric varieties}, Annals of Mathematics
  Studies, vol. 131, Princeton University Press, Princeton, NJ, 1993, The
  William H. Roever Lectures in Geometry.

\bibitem[GGZ87]{GKZ-holonomic-system-1987}
I.~M. Gel{\cprime}fand, M.~I. Graev, and A.~V. Zelevinski{\u\i},
  \emph{Holonomic systems of equations and series of hypergeometric type},
  Dokl. Akad. Nauk SSSR \textbf{295} (1987), no.~1, 14--19.

\bibitem[Giv95]{Givental-HMS-1995}
A.~B. Givental{\cprime}, \emph{Homological geometry. {I}. {P}rojective
  hypersurfaces}, Selecta Math. (N.S.) \textbf{1} (1995), no.~2, 325--345.

\bibitem[Giv96]{Givental-Equivariant-GW}
Alexander~B. Givental, \emph{Equivariant {G}romov-{W}itten invariants},
  Internat. Math. Res. Notices (1996), no.~13, 613--663, alg-geom/9603021.

\bibitem[Giv98]{Givental-1998-Mirror-complete-intersection}
Alexander Givental, \emph{A mirror theorem for toric complete intersections},
  Topological field theory, primitive forms and related topics ({K}yoto, 1996),
  Progr. Math., vol. 160, Birkh\"auser Boston, Boston, MA, 1998, pp.~141--175.

\bibitem[GKZ90]{GKZ-Generalized-Euler-integral-1990}
I.~M. Gel{\cprime}fand, M.~M. Kapranov, and A.~V. Zelevinsky, \emph{Generalized
  {E}uler integrals and {$A$}-hypergeometric functions}, Adv. Math. \textbf{84}
  (1990), no.~2, 255--271.

\bibitem[Gol07]{Golyshev-classification-pb-2007}
Vasily~V. Golyshev, \emph{Classification problems and mirror duality}, Surveys
  in geometry and number theory: reports on contemporary {R}ussian mathematics,
  London Math. Soc. Lecture Note Ser., vol. 338, Cambridge Univ. Press,
  Cambridge, 2007, pp.~88--121. \MR{2306141 (2008f:14058)}

\bibitem[GS14]{2008-Guest-Sakai-orbifold-QDM}
Martin Guest and Hironori Sakai, \emph{Orbifold quantum {D}-modules associated
  to weighted projective spaces}, Comment. Math. Helv. \textbf{89} (2014),
  no.~2, 273--297.

\bibitem[Gue10]{Guest-book-QDM-2010}
Martin~A. Guest, \emph{Differential equations aspects of quantum cohomology},
  Geometric and topological methods for quantum field theory, Cambridge Univ.
  Press, Cambridge, 2010, pp.~54--85.

\bibitem[GZK88]{GKZ-Equations-hypergeometric-1988}
I.~M. Gel{\cprime}fand, A.~V. Zelevinski{\u\i}, and M.~M. Kapranov,
  \emph{Equations of hypergeometric type and {N}ewton polyhedra}, Dokl. Akad.
  Nauk SSSR \textbf{300} (1988), no.~3, 529--534.

\bibitem[GZK89]{GKZ-hypergeo-toric-1990}
\bysame, \emph{Hypergeometric functions and toric varieties}, Funktsional.
  Anal. i Prilozhen. \textbf{23} (1989), no.~2, 12--26.

\bibitem[Her06]{2006-Hertling-tt*}
Claus Hertling, \emph{{$tt^*$} geometry and mixed {H}odge structures},
  Singularity theory and its applications, Adv. Stud. Pure Math., vol.~43,
  Math. Soc. Japan, Tokyo, 2006, pp.~73--84.

\bibitem[HTT08]{Hotta-D-module}
Ryoshi Hotta, Kiyoshi Takeuchi, and Toshiyuki Tanisaki, \emph{{$D$}-modules,
  perverse sheaves, and representation theory}, Progress in Mathematics, vol.
  236, Birkh\"auser Boston Inc., Boston, MA, 2008, Translated from the 1995
  Japanese edition by Takeuchi.

\bibitem[Iri06]{2006-Iritani-QDM-Floer}
Hiroshi Iritani, \emph{Quantum {$D$}-modules and equivariant {F}loer theory for
  free loop spaces}, Math. Z. \textbf{252} (2006), no.~3, 577--622.

\bibitem[Iri07]{iritani_convergence}
\bysame, \emph{Convergence of quantum cohomology by quantum {L}efschetz}, J.
  Reine Angew. Math. \textbf{610} (2007), 29--69.

\bibitem[Iri08]{2008-Iritani-QDM-general-mirror-transform}
\bysame, \emph{Quantum {$D$}-modules and generalized mirror transformations},
  Topology \textbf{47} (2008), no.~4, 225--276.

\bibitem[Iri09]{Iritani-2009-Integral-structure-QH}
\bysame, \emph{An integral structure in quantum cohomology and mirror symmetry
  for toric orbifolds}, Adv. Math. \textbf{222} (2009), no.~3, 1016--1079.

\bibitem[Iri11]{iritani_quantum_2011}
\bysame, \emph{Quantum cohomology and periods}, Ann. Inst. Fourier (Grenoble)
  \textbf{61} (2011), no.~7, 2909--2958.

\bibitem[Kim99]{Kim-1999-QH-for-G/P}
Bumsig Kim, \emph{Quantum hyperplane section theorem for homogeneous spaces},
  Acta Math. \textbf{183} (1999), no.~1, 71--99. \MR{1719555 (2001i:14076)}

\bibitem[KKP08]{Katzarkov-Pantev-Kontsevich-ncVHS}
L.~{Katzarkov}, M.~{Kontsevich}, and T.~{Pantev}, \emph{{Hodge theoretic
  aspects of mirror symmetry}}, ArXiv e-prints (2008).

\bibitem[Kon95]{Kontseich-1995-HMS}
Maxim Kontsevich, \emph{Homological algebra of mirror symmetry}, Proceedings of
  the {I}nternational {C}ongress of {M}athematicians, {V}ol.\ 1, 2 ({Z}\"urich,
  1994) (Basel), Birkh\"auser, 1995, pp.~120--139. \MR{1403918 (97f:32040)}

\bibitem[LLY99]{Lian-Liu-Yau-mirror-principle-1-1999}
Bong~H. Lian, Kefeng Liu, and Shing-Tung Yau, \emph{Mirror principle. {I} [
  {MR}1621573 (99e:14062)]}, Surveys in differential geometry: differential
  geometry inspired by string theory, Surv. Differ. Geom., vol.~5, Int. Press,
  Boston, MA, 1999, pp.~405--454. \MR{1772275}

\bibitem[Lyu88]{Lyubeznik-Taylor-resolution-monomial-ideals-1988}
Gennady Lyubeznik, \emph{A new explicit finite free resolution of ideals
  generated by monomials in an {$R$}-sequence}, J. Pure Appl. Algebra
  \textbf{51} (1988), no.~1-2, 193--195. \MR{941900 (89c:13020)}

\bibitem[Mat86]{Matsumura-commutative}
Hideyuki Matsumura, \emph{Commutative ring theory}, Cambridge Studies in
  Advanced Mathematics, vol.~8, Cambridge University Press, Cambridge, 1986,
  Translated from the Japanese by M. Reid.

\bibitem[Mav03]{mavlyutov_chiral_2000}
Anvar~R. Mavlyutov, \emph{On the chiral ring of {C}alabi-{Y}au hypersurfaces in
  toric varieties}, Compositio Math. \textbf{138} (2003), no.~3, 289--336.

\bibitem[MR13]{rietsch-Marsh-Grassmiannians-mirro-symmetry-2013arXiv1307.1085M}
R.~{Marsh} and K.~{Rietsch}, \emph{{The {B}-model connection and mirror
  symmetry for Grassmannians}}, ArXiv 1307.1085 (2013).

\bibitem[PRW16]{Rietsch-Pech-Williams-LG-quadrics-2014}
C.~Pech, K.~Rietsch, and L.~Williams, \emph{On {L}andau-{G}inzburg models for
  quadrics and flat sections of {D}ubrovin connections}, Adv. Math.
  \textbf{300} (2016), 275--319.

\bibitem[Rie12]{Rietsch-Mirror-Toda-2012}
Konstanze Rietsch, \emph{A mirror symmetric solution to the quantum {T}oda
  lattice}, Comm. Math. Phys. \textbf{309} (2012), no.~1, 23--49.

\bibitem[RS12]{Reichelt-SevenheckNonaffinLG-2012}
T.~{Reichelt} and C.~{Sevenheck}, \emph{{Non-affine Landau-Ginzburg models and
  intersection cohomology}}, ArXiv 1210.6527, to appear in Annales
  scientifiques de l'École Normale Supérieure (2012).

\bibitem[RS15]{Sevenheck-GKZ-log-Frob-manifold}
Thomas Reichelt and Christian Sevenheck, \emph{Logarithmic {F}robenius
  manifolds, hypergeometric systems and quantum {D}-modules}, J. Algebraic
  Geom. \textbf{24} (2015), no.~2, 201--281.

\bibitem[Sab05]{Sabbah-Asterisque}
Claude Sabbah, \emph{Polarizable twistor {D}-modules}, Ast\'erisque (2005),
  no.~300, vi+208.

\end{thebibliography}

\end{document}